\documentclass{lms} 
\title[Dimer models and cluster categories]
{Dimer models and cluster categories of Grassmannians}
\date{10 June 2016}
\author{Karin Baur, Alastair King and Robert J. Marsh}

\classno{Primary: 13F60, 16G50, 82B20 Secondary: 16G20, 57Q15}

\extraline{This work was supported by the Austrian Science Fund FWF-DK W1230,
the Engineering and Physical Sciences Research Council [grant number EP/
G007497/1], the Institute for Mathematical Research (FIM) at ETH, Zurich, 
and the Mathematical Sciences Research Institute, Berkeley.}

\usepackage{tikz}
\usetikzlibrary{calc,decorations.markings,arrows}
\tikzset{->-/.style={decoration={
  markings,
  mark=at position .5 with {\arrow{>}}},postaction={decorate}}}

\newlength{\qthickness}
\usepackage{color}
\usepackage{enumerate}
\usepackage{xy}
\usepackage{array}
\usepackage{amssymb,amsmath}
\usepackage{psfrag}

\usepackage{epic}
\usepackage{subfigure}
\usepackage{color}

\xyoption{matrix}
\xyoption{all}

\newtheorem{lemma}{Lemma}[section]
\newtheorem{proposition}[lemma]{Proposition}
\newtheorem{theorem}[lemma]{Theorem}
\newtheorem{corollary}[lemma]{Corollary}
\newtheorem{conjecture}[lemma]{Conjecture}

\newnumbered{definition}[lemma]{Definition}
\newnumbered{remark}[lemma]{Remark}

\newcommand{\directedangle}{\measuredangle}

\newcommand{\inn}{\text{in}}
\newcommand{\out}{\text{out}}
\newcommand{\id}{\text{id}}
\newcommand{\opp}{\text{opp}}
\newcommand{\coker}{\operatorname{coker}}
\newcommand{\End}{\operatorname{End}}
\newcommand{\Hom}{\operatorname{Hom}}

\newcommand{\module}{\mathbb{M}}

\newcommand{\I}{\mathcal I}
\newcommand{\LL}{\mathcal L}

\newcommand{\C}{\mathcal C}
\newcommand{\F}{\mathbb C}
\newcommand{\W}{\mathcal W}
\newcommand{\R}{\mathcal R}

\newcommand{\B}{\mathcal B}

\newcommand{\bdry}{\partial}

\newcommand{\weight}{w}
\newcommand{\newt}{f}
\newcommand{\embed}{g}
\newcommand{\newL}{E}

\newcommand{\Qcyc}{Q_{\text{cyc}}}

\newcommand{\Gcyc}{\Gamma_{\text{cyc}}}

\newcommand{\qrelation}{{\sim}}
\newcommand{\rr}{r}

\begin{document}
\maketitle

\begin{abstract}
We associate a dimer algebra $A$ to a Postnikov diagram $D$ (in a disk)
corresponding to a cluster of minors in the cluster structure of the Grassmannian $Gr(k,n)$.
We show that $A$ is isomorphic to the endomorphism algebra of a corresponding
Cohen-Macaulay module $T$ over the algebra $B$ used to categorify the cluster structure of $Gr(k,n)$
by Jensen-King-Su.
It follows that $B$ can be realised as the boundary algebra of $A$, that is, the subalgebra $eAe$
for an idempotent $e$ corresponding to the boundary of the disk.
The construction and proof uses an interpretation of the diagram $D$,
with its associated plabic graph and dual quiver (with faces),
as a dimer model with boundary.
We also discuss the general surface case, in particular computing boundary algebras associated to the annulus.
\end{abstract}

\section*{Introduction}
\label{s:introduction}

Postnikov diagrams
(also known as alternating strand diagrams)
are collections of curves in a disk satisfying certain axioms. They were introduced by Postnikov
in his study of total positivity of the Grassmannian $Gr(k,n)$ of $k$-planes
in $\mathbb{C}^n$~\cite{postnikov}.
A class of Postnikov diagrams was used by Scott~\cite{scott06} 
to show that the homogeneous coordinate ring of $Gr(k,n)$ is a cluster algebra,
in which each such diagram corresponds to a seed whose (extended) cluster consists of minors (i.e., Pl\"{u}cker coordinates).
The combinatorics of the diagram gives both the quiver of the cluster and the minors: 
the $k$-subsets of $\{1,2,\ldots ,n\}$ corresponding to the minors appear as labels of alternating regions 
and the quiver can be read off geometrically.
By a more recent result of Oh-Postnikov-Speyer~\cite{ops},
every cluster consisting of minors arises in this way (see also related results in~\cite{dkk}).

An additive categorification of this cluster algebra structure has been given
by 
Geiss-Leclerc-Schr\"{o}er~\cite{gls08} in terms of a subcategory of the category of
finite dimensional modules over the preprojective algebra of type $A_{n-1}$. 
However, there is a single cluster coefficient, the
minor corresponding to the $k$-subset $\{1,2,\ldots ,k\}$ of
$\{1,2,\ldots ,n\}$, which is not realised in the category.
Thus the categorification in \cite{gls08} is strictly only of the coordinate ring of the affine open cell
in the Grassmannian given by the non-vanishing of this minor. 
The cluster structure is then lifted to the homogeneous coordinate ring in an explicit and natural way.

Recently Jensen-King-Su~\cite{jks} have given a full and direct categorification of the cluster structure on the homogeneous
coordinate ring,
using the category of (maximal) Cohen-Macaulay modules over the completion of an algebra $B$, 
which is a quotient of the preprojective algebra of type $\tilde{A}_{n-1}$. 
In particular, a rank one Cohen-Macaulay $B$-module $\module_I$ is associated to every $k$-subset $I$ of $\{1,2, \ldots ,n\}$. 

Given a Postnikov diagram $D$, let 
\[ T_D = \bigoplus \module_I, \]
where the direct sum is over the $I$ labelling the alternating regions of $D$.
As noted in~\cite{jks}, the completion of $T_D$ is a cluster-tilting
module and
the work of Buan-Iyama-Reiten-Smith~\cite{BIRS11} would lead one to ask whether the completion of the
endomorphism algebra $\End_{B}(T_D)$ is a frozen Jacobian algebra (\cite[Def. 1.1]{BIRS11}).
We will see that this is indeed the case, in a very particular way.

We associate to $D$ a quiver with faces $Q(D)$. The subgraph containing the arrows
incident with internal vertices of $Q(D)$ corresponds to the skew-symmetric matrix associated to $D$ in Scott~\cite[\S5]{scott06}, 
but there are additional arrows between the boundary vertices. The faces of $Q(D)$ correspond to the oriented regions of $D$.
From the construction, $Q(D)$ may be embedded in the disk in which $D$ is drawn and it is natural to interpret it as a dimer model
with boundary, as a generalisation of dimer models on a torus or a more general closed surface
(see \cite{bocklandt12}, \cite{broomhead12}, \cite{davison11}, \cite{fhkvw}). 
Such a generalisation has also been introduced recently by Franco~\cite{francopre12}.

To formalise this, we give an abstract definition of a
dimer model with boundary as a quiver with faces satisfying certain axioms; in particular, the arrows are divided
into internal arrows and boundary arrows.
Such a dimer model has a natural embedding into a compact surface with boundary 
in which each component of the boundary is identified with an unoriented cycle of boundary arrows in $Q$.
In the case without boundary, this corresponds closely to the definition of a dimer model given by Bocklandt~\cite[2.2]{bocklandt12}.
We associate to any dimer model $Q$ with boundary a \emph{dimer algebra} $A_Q$
which coincides with the usual dimer algebra, defined by a superpotential or commutation relations,
in the case where the boundary is empty.
In that case, the completion with respect to the arrow ideal coincides with the 
usual Jacobian algebra of the quiver with potential $(Q,W)$ (as in \cite[\S3]{dwz08}). 
If the boundary is nonempty, the dimer algebra can still be defined via a potential, 
but the relations do not include the derivatives of the potential with respect to boundary arrows.
This is a slightly different convention to \cite[Def. 1.1]{BIRS11}, because it is convenient to allow dimer model quivers to have 2-cycles.

Our main result is that $\End_{B}(T_D)$
is isomorphic to the dimer algebra $A_D=A_{Q(D)}$ associated to the dimer model $Q(D)$.
There is a natural grading on $A_D$ by subsets of $\{1,2,\ldots ,n\}$,
which has a simple definition in terms of~$D$,
and there is also a similar grading on $\End_B(T_D)$.
We use these
to show that there is a well-defined (graded) homomorphism 
\[ g\colon A_D\rightarrow \End_B(T_D) \]
taking each arrow $I\to J$ in $Q(D)$ to the homomorphism $\module_I\to \module_J$
which generates $\Hom_B(\module_I,\module_J)$ freely as a $\mathbb{C}[t]$-module,
where the polynomial ring $\mathbb{C}[t]$ is the centre of~$B$.
We also note that the centre of $A_D$ is $\mathbb{C}[u]$, where $u$ is the sum of all
minimal loops in $Q(D)$, i.e.\ the loops around the faces, and that $g(u)=t$.

To see that $g$ is surjective, we show that between any two vertices $I,J$ of $Q(D)$ there is a path of minimal degree, 
which (as an element of $A_D$) must then map to the generator of $\Hom_B(\module_I,\module_J)$.
Such a path is constructed inductively, with the
induction step depending on a careful analysis of the local behaviour of
strands in $D$ near a vertex or face of $Q(D)$.
As an aside, we note that this local analysis implies that $Q(D)$ can be isoradially
embedded into a planar disk and we relate this to the embedding
of $Q(D)$ constructed as a `plabic tiling' in~\cite[\S9]{ops}.

To see that $g$ is injective, we observe that, since the dimer model $Q(D)$ is `consistent' in an
appropriate sense (cf.~\cite[Rk.~6.4]{ops}), we can adapt arguments from \cite[\S5]{bocklandt12} to
show that any path between two vertices in $Q(D)$ is equal (in $A_D$)
to a path of minimal degree multiplied by a power of $u$. 
Thus $g$ is an isomorphism and we also prove that $g$ induces an isomorphism between the corresponding
completed algebras.

Let $e\in A_D$ be the sum of the primitive idempotents corresponding to the
boundary vertices. 
We call the algebra
$eA_De$ the \emph{boundary algebra} of $Q(D)$
and it is a notable fact that this algebra is 
independent of the choice of Postnikov diagram $D$, 
once $k$ and $n$ are fixed.
This follows immediately from the isomorphism $A_D \cong \End_B(T_D)$,
because the $B$-modules corresponding to
the idempotents in $e$ are the indecomposable projectives
and so 
\[ eA_De \cong \End_B(B)\cong B^{\opp}. \]
However, we also give a direct proof of the independence
by showing that the boundary algebra is invariant under the
untwisting, twisting and geometric exchange
moves~\cite[\S14]{postnikov} (see also~\cite[\S3]{scott06}) for Postnikov diagrams.

Finally, for any integer $k\geq 1$, we consider a notion of
Postnikov diagram of degree $k$ on a marked surface with boundary in which all of the marked points lie on the boundary, generalizing the usual notion of an
Postnikov diagram which can be regarded as the disk case.
We say that such a diagram is a \emph{weak Postnikov diagram}
if it is not required to satisfy the global `consistency'
axioms (conditions (b1) and (b2) in
Definition~\ref{d:asd}). Adapting a construction
of~\cite[\S3]{scott06}, we associate a weak Postnikov diagram
of degree $2$ to a triangulation of any such marked surface.
We compute the corresponding boundary algebra in the case of an annulus
with at least one marked point on each of its boundary components and show that it
is independent of the choice of triangulation.

The structure of the paper is as follows. In Section~\ref{s:setup} we set up some of the
notation. In Section~\ref{s:diagrams}, we recall the definition of a
Postnikov diagram~\cite[\S14]{postnikov} and its corresponding plabic graph, as well as the corresponding quiver~\cite[\S5]{scott06}.
In Section~\ref{s:dimermodels}, we give the definition of a dimer model with boundary and its corresponding dimer algebra, noting that the quiver associated to a
Postnikov diagram can be given such a structure.

In Section~\ref{s:weights} we define a weighting on the arrows in $Q(D)$, computing
the weight of the boundary of a face of $Q(D)$ and the sum of the weights of
the arrows incident with a vertex of $Q(D)$. In Section~\ref{s:angles} we show how these
results can be used to embed $Q(D)$ isoradially into a disk.

In Section~\ref{s:legalarrow}, we show how the results in Section~\ref{s:weights} can be used
to construct the first arrow in the minimal path mentioned above.
In Section~\ref{s:rankone} we recall the
algebra $B$ (and the completed version $\widehat{B}$)
from~\cite{jks} and define the $B$-module $T_D$.
In Section~\ref{s:minimalpath}, we construct a minimal path.
In Section~\ref{s:paths}, we show that there is
a unique element of $A_D$ which can be written as
such a path, and that any path in $Q(D)$ is
equal in $A_D$ to this element multiplied by
a power of a minimal loop.
In Section~\ref{s:isomorphism},
we prove that $A_D$ is isomorphic to $\End_B(T_D)$ and show
that $eA_De$ is isomorphic to $B$. In Section~\ref{s:completion}, we give the
completed version of these results.
In Section~\ref{s:exchange}, we give a proof in terms of Postnikov diagrams that the algebra $eA_De$
is independent of the choice of Postnikov diagram,
and in Section~\ref{s:surfaces} we consider the surface case.

\begin{acknowledgements}
We would like to thank A. Hubery and A. Craw for helpful conversations
relating to Remark~\ref{r:commutation}.
We are grateful for the hospitality and pleasant working environment provided
by ETH Zurich (Spring 2011, Summer 2013)
and MSRI Berkeley (Autumn 2012).
\end{acknowledgements}

\section{Set-up and notation} 
\label{s:setup}

Fix a positive integer $n$ and an integer $k$ with $1\leq k\leq n-1$.
We will write
$\mathbb{Z}_n=\{1,2,\ldots ,n\}$.
We consider a circular graph $C$ with vertices $C_0=\mathbb{Z}_n$
clockwise around a circle and edges $C_1$ also labelled by
$\mathbb{Z}_n$, with edge $i$ joining vertices $i-1$ and $i$;
see Figure~\ref{f:complex} for the case $n=7$.
For integers $a,b\in \{1,2,\ldots ,n\}$, we denote by $[a,b]$ the
closed cyclic interval consisting of the elements of the set
$\{a,a+1,\ldots ,b\}$ reduced mod $n$.
We similarly have the open interval $(a,b)$.

For a subset $S$ of $C_1$, define $S_0$ to be the set of vertices incident
with an edge in $S$. So, in particular, $(a,b)_0=[a,b-1]$,
the set of vertices incident with an edge in the set $(a,b)$.

In general for sets $S,S'$ we write $S-S'$ for the set of elements in $S$
but not in $S'$. For $s\in S$, we use the shorthand $S-s$ for $S-\{s\}$
and for any $z$ we use the shorthand $S+z$ for $S\cup \{z\}$.

\begin{figure}
\[
\begin{tikzpicture}[scale=1,baseline=(bb.base)]
\newcommand{\seventh}{51.4} 
\newcommand{\circradius}{1.5cm}
\newcommand{\inradius}{1.2cm}
\newcommand{\outradius}{1.8cm}
\path (0,0) node (bb) {}; 

\draw[black,thick] (0,0) circle(\circradius);
\foreach \j in {1,...,7}
{\draw (90-\seventh*\j:\circradius) node[black] {$\bullet$};
 \draw (90-\seventh*\j:\outradius) node[black] {$\j$};
 \draw (90+\seventh/2-\seventh*\j:\inradius) node[black] {$\j$}; }
\end{tikzpicture}
\]
\caption{The graph $C$}
\label{f:complex}
\end{figure}
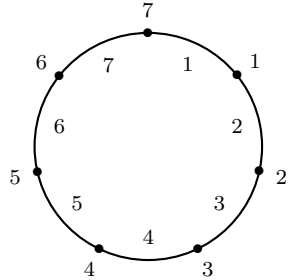

\section{Postnikov diagrams} 
\label{s:diagrams}

We recall a special case of the definition of a \emph{Postnikov diagram} (alternating strand diagram)~\cite[\S14]{postnikov}.

\begin{definition} \label{d:asd}
A $(k,n)$-\emph{Postnikov diagram} $D$
consists of $n$ directed curves, called \emph{strands}, in a disk with $n$ 
marked vertices
on its boundary, labelled by the elements of $C_1$ (in clockwise order). 
The strands are also labelled by the elements of $C_1$,
with strand $i$ starting at vertex $i$ and ending at vertex $i+k$. The following
axioms must be satisfied.
\goodbreak
\emph{Local axioms:}
\begin{enumerate}[({a}1)]
\item Only two strands can cross at a given point and all crossings are
transverse.
\item There are finitely many crossing points.
\item
Proceeding along a given strand, the other strands crossing it alternate between crossing it left to right and right to left.
\end{enumerate}
\goodbreak
\emph{Global axioms:}
\begin{enumerate}[({b}1)]
\item A strand cannot intersect itself.
\item If two strands intersect at distinct points $U$ and $V$, then one strand is oriented from $U$ to $V$ and the other is oriented from $V$ to $U$.
\end{enumerate}

\noindent Note: for axiom~(a3), strands $i-k$ and $i$ are regarded as crossing at the boundary vertex $i$ in the
obvious way. Note also that because the disk is compact, condition (a2) is effectively local.
\end{definition}

We shall often refer to $(k,n)$-Postnikov diagrams as simply Postnikov diagrams
when $k$ and $n$ are clear from the context.
A Postnikov diagram is defined up to isotopies fixing the boundary.
Two Postnikov diagrams are said to be \emph{equivalent}
if one can be obtained from the other using 
the untwisting and twisting moves
illustrated in Figure~\ref{f:twisting} (or the opposite versions, obtained from these diagrams by reflection in a horizontal
line). Note that these moves are local: there must be a disk containing
the initial configuration, and no other
strands are involved in the move.
We call an untwisting or twisting move at the boundary a \emph{boundary untwisting} or \emph{twisting} move (the lower diagram in the figure).

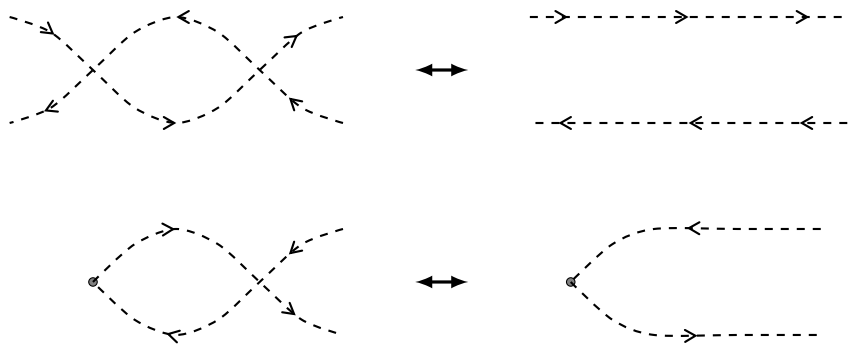
\begin{figure}
\[
\begin{tikzpicture}[scale=0.7,baseline=(bb.base),
  strand/.style={black,dashed,thick},
  doublearrow/.style={black, latex-latex, very thick}]
\newcommand{\strarrow}{\arrow{angle 60}}
\newcommand{\dotrad}{0.1cm} 
\newcommand{\bdrydotrad}{{0.8*\dotrad}} 
\path (0,0) node (bb) {}; 


\draw [strand] plot[smooth]
coordinates {(-3.14,1.0) (-2.36,0.71) (-1.57,0) (-0.79,-0.71) (0,-1) (0.79,-0.71) (1.57,0) (2.36, 0.71) (3.14,1)}
[ postaction=decorate, decoration={markings,
  mark= at position 0.12 with \strarrow, mark= at position 0.5 with \strarrow, mark= at position 0.88 with \strarrow}];


\draw [strand] plot[smooth]
coordinates {(3.14,-1.0) (2.36,-0.71) (1.57,0) (0.79,0.71) (0,1) (-0.79,0.71)
(-1.57,0) (-2.36, -0.71) (-3.14,-1)}
[ postaction=decorate, decoration={markings,
  mark= at position 0.15 with \strarrow, mark= at position 0.5 with \strarrow,
mark= at position 0.91 with \strarrow}];


\draw [strand] plot
coordinates {(6.66,1.0) (12.64,1.0)}
[ postaction=decorate, decoration={markings,
  mark= at position 0.12 with \strarrow, mark= at position 0.5 with \strarrow, mark= at position 0.88 with \strarrow}];


\draw [strand] plot
coordinates {(12.64,-1.0) (6.66,-1.0)}
[ postaction=decorate, decoration={markings,
  mark= at position 0.15 with \strarrow, mark= at position 0.5 with \strarrow, mark= at position 0.91 with \strarrow}];


\draw [doublearrow] (4.5,0) -- (5.5,0);


\begin{scope}[shift={(0,-4)}]


\draw (-1.57,0) circle(\bdrydotrad) [fill=gray];
\draw (7.43,0) circle(\bdrydotrad) [fill=gray];


\draw [strand] plot[smooth]
coordinates {(-1.57,0) (-0.79,0.71) (0,1) (0.79,0.71) (1.57,0) (2.36, -0.71) (3.14,-1)}
[ postaction=decorate, decoration={markings,
  mark= at position 0.33 with \strarrow, mark= at position 0.83 with \strarrow}];


\draw [strand] plot[smooth]
coordinates {(3.14,1) (2.36,0.71) (1.57,0) (0.79,-0.71) (0,-1) (-0.79,-0.71) (-1.57,0)}
[ postaction=decorate, decoration={markings,
  mark= at position 0.2 with \strarrow, mark= at position 0.7 with \strarrow}];


\draw [strand] plot[smooth]
coordinates {(7.43,0) (8.21,-0.71) (9,-1) (10.57,-1) (12.14,-1)}
[ postaction=decorate, decoration={markings,
  mark= at position 0.54 with \strarrow}];


\draw [strand] plot[smooth]
coordinates {(12.14,1) (10.57,1) (9,1) (8.21,0.71) (7.43,0)}
[ postaction=decorate, decoration={markings,
  mark= at position 0.5 with \strarrow}];


\draw [doublearrow] (4.5,0) -- (5.5,0);

\end{scope}

\end{tikzpicture}
\]
\caption{Untwisting and twisting moves in a Postnikov diagram. The moves obtained by reflecting these diagrams in a horizontal line are also allowed}
\label{f:twisting}
\end{figure}

\begin{definition} \label{d:reduced}
We shall say that a
Postnikov diagram is of 
\emph{reduced type} if no untwisting
move or boundary untwisting move
(i.e.\ going from left to right in 
Figure~\ref{f:twisting}) can be applied to 
it.
\end{definition}

Note that in a Postnikov diagram of reduced type, while the first crossing of strand $i$ is with strand $i-k$, as is required, the second crossing must be with a different strand.
Figure~\ref{f:postfree37} shows an example of a $(3,7)$-Postnikov diagram
which is of reduced type.

\begin{figure}
\[
\begin{tikzpicture}[scale=1.1,baseline=(bb.base),
 strand/.style={black, dashed}]

\newcommand{\strarrow}{\arrow{angle 60}}
\newcommand{\bstart}{125} 
\newcommand{\seventh}{51.4} 
\newcommand{\qstart}{150.7} 
\newcommand{\bigrad}{4cm} 
\newcommand{\eps}{11pt} 
\newcommand{\dotrad}{0.1cm} 
\newcommand{\bdrydotrad}{{0.8*\dotrad}} 

\path (0,0) node (bb) {};


\draw (0,0) circle(\bigrad) [thick,gray, densely dotted];

\foreach \n in {1,...,7}
{ \coordinate (b\n) at (\bstart-\seventh*\n:\bigrad);
  \draw (\bstart-\seventh*\n:\bigrad+\eps) node {$\n$}; }


\foreach \n/\a/\r in {8/77/0.79, 10/130/0.5, 12/60/0.2, 14/350/0.5, 16/220/0.3, 18/225/0.75, 20/280/0.75,
    9/117/0.77, 11/92/0.38, 13/30/0.7, 15/290/0.05, 17/185/0.7, 19/250/0.55,  21/330/0.75}
{ \coordinate (b\n)  at (\a:\r*\bigrad); }


\foreach \e/\f/\t in {8/9/0.4, 9/10/0.5, 8/11/0.4, 10/11/0.5, 11/12/0.5, 8/13/0.5, 12/13/0.65, 13/14/0.4, 12/15/0.5,
 14/15/0.5, 15/16/0.5, 16/17/0.65, 10/17/0.6, 17/18/0.45, 18/19/0.5, 19/20/0.5, 20/21/0.5, 16/19/0.5, 14/21/0.5}
{\coordinate (a\e\f) at (${\t}*(b\e) + {1-\t}*(b\f)$); }


\draw [strand] plot[smooth]
coordinates {(b1) (a89) (a910) (a1017) (a1617) (a1619) (a1920) (b4)}
[ postaction=decorate, decoration={markings,
  mark= at position 0.11 with \strarrow, mark= at position 0.255 with \strarrow,
  mark= at position 0.37 with \strarrow, mark= at position 0.52 with \strarrow,
  mark= at position 0.655 with \strarrow, mark= at position 0.775 with \strarrow,
  mark= at position 0.92 with \strarrow }];
\draw [strand] plot[smooth] coordinates {(b2) (a1314) (a1415) (a1516) (a1617) (a1718) (b5)}
[ postaction=decorate, decoration={markings,
  mark= at position 0.13 with \strarrow, mark= at position 0.29 with \strarrow,
  mark= at position 0.47 with \strarrow, mark= at position 0.61 with \strarrow,
  mark= at position 0.75 with \strarrow, mark= at position 0.9 with \strarrow }];
\draw [strand] plot[smooth] coordinates {(b3) (a2021) (a1920) (a1819) (a1718) (b6)}
[ postaction=decorate, decoration={markings,
  mark= at position 0.125 with \strarrow, mark= at position 0.34 with \strarrow,
  mark= at position 0.52 with \strarrow, mark= at position 0.68 with \strarrow,
  mark= at position 0.86 with \strarrow }];
\draw [strand] plot[smooth] coordinates {(b4) (a2021) (a1421) (a1314) (a1213) (a1112) (a1011) (a910) (b7)}
[ postaction=decorate, decoration={markings,
  mark= at position 0.11 with \strarrow, mark= at position 0.27 with \strarrow,
  mark= at position 0.4 with \strarrow, mark= at position 0.53 with \strarrow,
  mark= at position 0.63 with \strarrow, mark= at position 0.725 with \strarrow,
  mark= at position 0.81 with \strarrow, mark= at position 0.92 with \strarrow, }];
\draw [strand] plot[smooth] coordinates {(b5) (a1819) (a1619) (a1516) (a1215) (a1213) (a813) (b1)}
[ postaction=decorate, decoration={markings,
  mark= at position 0.1 with \strarrow, mark= at position 0.24 with \strarrow,
  mark= at position 0.35 with \strarrow, mark= at position 0.46 with \strarrow,
  mark= at position 0.58 with \strarrow, mark= at position 0.735 with \strarrow,
  mark= at position 0.91 with \strarrow }];
  \draw [strand] plot[smooth] coordinates {(b6) (a1017) (a1011) (a811) (a813) (b2)}
[ postaction=decorate, decoration={markings,
  mark= at position 0.14 with \strarrow, mark= at position 0.355 with \strarrow,
  mark= at position 0.5 with \strarrow, mark= at position 0.65 with \strarrow,
  mark= at position 0.86 with \strarrow }];
\draw [strand] plot[smooth] coordinates {(b7) (a89) (a811) (a1112) (a1215) (a1415) (a1421) (b3)}
[ postaction=decorate, decoration={markings,
  mark= at position 0.11 with \strarrow, mark= at position 0.27 with \strarrow,
  mark= at position 0.4 with \strarrow, mark= at position 0.505 with \strarrow,
  mark= at position 0.605 with \strarrow, mark= at position 0.74 with \strarrow,
  mark= at position 0.915 with \strarrow }];


\foreach \n in {1,2,3,4,5,6,7} {\draw (b\n) circle(\bdrydotrad) [fill=gray];}


\foreach \n/\m/\r in {1/567/0.88, 2/671/0.87, 3/712/0.8, 4/123/0.83, 5/234/0.8, 6/345/0.85, 7/456/0.79}
{ \draw (\qstart-\seventh*\n:\r*\bigrad) node (q\m) {$\m$}; }

\foreach \m/\a/\r in {156/104/0.58 , 157/63/0.47, 145/160/0.25, 147/15/0.3, 245/220/0.52, 124/295/0.4 }
{ \draw (\a:\r*\bigrad) node (q\m) {$\m$}; }

 \end{tikzpicture}
\]
\caption{A $(3,7)$-Postnikov diagram}
\label{f:postfree37}
\end{figure}
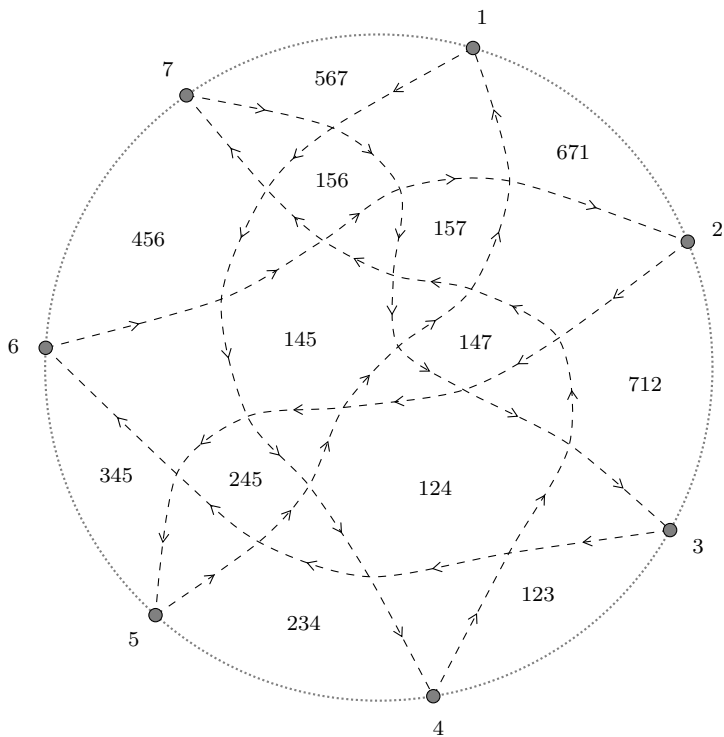

A Postnikov diagram divides the interior of the disk into bounded regions,
the connected components of the complement of the strands in the diagram.
A region not adjacent to the boundary of the disk is called \emph{internal}
and the other regions are referred to as a \emph{boundary} region.
A region is said to be \emph{alternating} if the
strands incident with it alternate in orientation going around the
boundary (ignoring the boundary of the disk). It is said to be
\emph{oriented} if the strands around its boundary are all oriented
clockwise (or all anticlockwise). It easy to check that every region of a Postnikov
diagram must be alternating or oriented.

Each strand divides the disk into two parts, consisting of those regions
on the left hand side of the strand (when flowing along it) and those on the
right hand side. Each alternating region is labelled with the $k$-subset $I$ of
$C_1$ consisting of the numbers of those strands which have the region on
their left hand side. The labels of the alternating regions are all distinct. We denote the set of labels of $D$ by $\C(D)$.
The alternating regions have been labelled in
Figure~\ref{f:postfree37}, using the convention that a subset $\{i_1,i_2,\ldots ,i_s\}$
of $C_1$ is displayed as $i_1i_2\cdots i_s$.

\begin{remark} \label{r:boundarylabels}
For $i\in C_0$, let $\newL_i=[i-k+1,i]\subset C_1$,
i.e.\ the set of labels of the vertices between edges $i-k$ and $i$.
Then the labels of the boundary alternating regions are precisely the
$k$-subsets $\newL_1,\newL_2,\ldots ,\newL_n$ (see~\cite[\S3]{scott06}).
\end{remark}

Recall that a quiver $Q$ is a directed graph encoded by a tuple $Q=(Q_0,Q_1,h,t)$, where $Q_0$ is the
set of vertices, $Q_1$ is the set of arrows and $h,t\colon Q_1\to Q_0$,
so that each $\alpha\in Q_1$ is an arrow $t\alpha\to h\alpha$.
We will write $Q=(Q_0,Q_1)$, with the remaining data implicit, and we will also regarded
it as an oriented 1-dimensional CW-complex.

\begin{definition} \label{d:quiver}
The \emph{quiver} $Q(D)$ of a Postnikov diagram $D$ has vertices $Q_0(D)=\C(D)$ given by the labels of the alternating regions of $D$. The arrows
$Q_1(D)$ correspond to intersection points of two alternating regions,
with orientation as in Figure~\ref{f:arrowconvention1}.
The diagram on the right indicates the
boundary case, which can also occur in the opposite sense.
We refer to the arrows between boundary vertices as \emph{boundary arrows}.
\end{definition}

\begin{remark} \label{r:quiverembedding}
We can embed $Q(D)$ into the disk, with each vertex plotted at some point
in the interior of the alternating region it corresponds to, except for
boundary regions, in which case we plot the point on the boundary of the disk.
Each arrow is drawn within the two regions corresponding to its end-points
and passing through the corresponding crossing in $D$.
Boundary arrows are drawn along the boundary.
\end{remark}

For example, the quiver of the Postnikov diagram in
Figure~\ref{f:postfree37} is shown in Figure~\ref{f:postfreequiver37},
embedded as in Remark~\ref{r:quiverembedding}.

\begin{figure}
\[
\begin{tikzpicture}[scale=1.3,baseline=(bb.base),
  strand/.style={black,dashed,thick},
  quivarrow/.style={black, -latex, very thick}]
\newcommand{\strarrow}{\arrow{angle 60}}
\path (0,0) node (bb) {}; 

\draw [strand] plot[smooth]
coordinates {(0.6,0.8) (0.4,0.4) (0.2,0.15) (0,0) (-0.2,-0.15) (-0.4,-0.4) (-0.6,-0.8)}
[ postaction=decorate, decoration={markings,
  mark= at position 0.25 with \strarrow, mark= at position 0.8 with \strarrow}];
\draw [strand] plot[smooth]
coordinates {(0.6,-0.8) (0.4,-0.4) (0.2,-0.15) (0,0) (-0.2,0.15) (-0.4,0.4) (-0.6,0.8)}
[ postaction=decorate, decoration={markings,
  mark= at position 0.25 with \strarrow, mark= at position 0.8 with \strarrow}];

\draw (-1,0) node (h) {$\bullet$};
\draw (1,0) node (t) {$\bullet$};
\draw [quivarrow] (t)--(h);

\end{tikzpicture}
\qquad\qquad
\begin{tikzpicture}[scale=1.3,baseline=(bb.base),
  strand/.style={black,dashed,thick},
  quivarrow/.style={black, -latex, very thick}]
\newcommand{\strarrow}{\arrow{angle 60}}
\newcommand{\bdrydotrad}{0.07cm} 
\path (0,0) node (bb) {}; 

\draw [strand] plot[smooth]
coordinates {(0.6,0.8) (0.4,0.4) (0.2,0.15) (0,0)}
[ postaction=decorate, decoration={markings,
  mark= at position 0.5 with \strarrow}];
\draw [strand] plot[smooth]
coordinates {(0,0) (-0.2,0.15) (-0.4,0.4) (-0.6,0.8)}
[ postaction=decorate, decoration={markings,
  mark= at position 0.6 with \strarrow}];

\draw (0,0) circle(\bdrydotrad) [fill=gray];
\draw (-1,0) node (h) {$\bullet$};
\draw (1,0) node (t) {$\bullet$};
\draw [quivarrow] (t)--(h);
\end{tikzpicture}
\]
\caption{Orientation convention for the quiver $Q(D)$}
\label{f:arrowconvention1}
\end{figure}
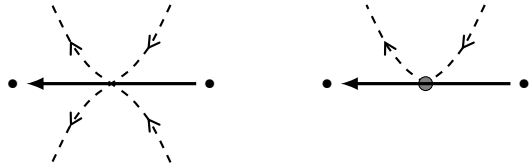

\begin{figure}
\[
\begin{tikzpicture}[scale=1.1,baseline=(bb.base),
 strand/.style={black, dashed},
 quivarrow/.style={black, -latex, thick}]

\newcommand{\strarrow}{\arrow{angle 60}}
\newcommand{\bstart}{125} 
\newcommand{\seventh}{51.4} 
\newcommand{\qstart}{150.7} 
\newcommand{\bigrad}{4cm} 
\newcommand{\eps}{11pt} 
\newcommand{\dotrad}{0.1cm} 
\newcommand{\bdrydotrad}{{0.8*\dotrad}} 

\path (0,0) node (bb) {};

\foreach \n in {1,...,7}
{ \coordinate (b\n) at (\bstart-\seventh*\n:\bigrad);
  \draw (\bstart-\seventh*\n:\bigrad+\eps) node {$\n$}; }


\foreach \n/\a/\r in {8/77/0.79, 10/130/0.5, 12/60/0.2, 14/350/0.5, 16/220/0.3, 18/225/0.75, 20/280/0.75,
    9/117/0.77, 11/92/0.38, 13/30/0.7, 15/290/0.05, 17/185/0.7, 19/250/0.55,  21/330/0.75}
{ \coordinate (b\n)  at (\a:\r*\bigrad); }


\foreach \e/\f/\t in {8/9/0.4, 9/10/0.5, 8/11/0.4, 10/11/0.5, 11/12/0.5, 8/13/0.5, 12/13/0.65, 13/14/0.4, 12/15/0.5,
 14/15/0.5, 15/16/0.5, 16/17/0.65, 10/17/0.6, 17/18/0.45, 18/19/0.5, 19/20/0.5, 20/21/0.5, 16/19/0.5, 14/21/0.5}
{\coordinate (a\e\f) at (${\t}*(b\e) + {1-\t}*(b\f)$); }


\draw [strand] plot[smooth]
coordinates {(b1) (a89) (a910) (a1017) (a1617) (a1619) (a1920) (b4)}
[ postaction=decorate, decoration={markings,
  mark= at position 0.11 with \strarrow, mark= at position 0.255 with \strarrow,
  mark= at position 0.37 with \strarrow, mark= at position 0.52 with \strarrow,
  mark= at position 0.655 with \strarrow, mark= at position 0.775 with \strarrow,
  mark= at position 0.92 with \strarrow }];
\draw [strand] plot[smooth] coordinates {(b2) (a1314) (a1415) (a1516) (a1617) (a1718) (b5)}
[ postaction=decorate, decoration={markings,
  mark= at position 0.13 with \strarrow, mark= at position 0.29 with \strarrow,
  mark= at position 0.47 with \strarrow, mark= at position 0.61 with \strarrow,
  mark= at position 0.75 with \strarrow, mark= at position 0.9 with \strarrow }];
\draw [strand] plot[smooth] coordinates {(b3) (a2021) (a1920) (a1819) (a1718) (b6)}
[ postaction=decorate, decoration={markings,
  mark= at position 0.125 with \strarrow, mark= at position 0.34 with \strarrow,
  mark= at position 0.52 with \strarrow, mark= at position 0.68 with \strarrow,
  mark= at position 0.86 with \strarrow }];
\draw [strand] plot[smooth] coordinates {(b4) (a2021) (a1421) (a1314) (a1213) (a1112) (a1011) (a910) (b7)}
[ postaction=decorate, decoration={markings,
  mark= at position 0.11 with \strarrow, mark= at position 0.27 with \strarrow,
  mark= at position 0.4 with \strarrow, mark= at position 0.53 with \strarrow,
  mark= at position 0.63 with \strarrow, mark= at position 0.725 with \strarrow,
  mark= at position 0.81 with \strarrow, mark= at position 0.92 with \strarrow, }];
\draw [strand] plot[smooth] coordinates {(b5) (a1819) (a1619) (a1516) (a1215) (a1213) (a813) (b1)}
[ postaction=decorate, decoration={markings,
  mark= at position 0.1 with \strarrow, mark= at position 0.24 with \strarrow,
  mark= at position 0.35 with \strarrow, mark= at position 0.46 with \strarrow,
  mark= at position 0.58 with \strarrow, mark= at position 0.735 with \strarrow,
  mark= at position 0.91 with \strarrow }];
  \draw [strand] plot[smooth] coordinates {(b6) (a1017) (a1011) (a811) (a813) (b2)}
[ postaction=decorate, decoration={markings,
  mark= at position 0.14 with \strarrow, mark= at position 0.355 with \strarrow,
  mark= at position 0.5 with \strarrow, mark= at position 0.65 with \strarrow,
  mark= at position 0.86 with \strarrow }];
\draw [strand] plot[smooth] coordinates {(b7) (a89) (a811) (a1112) (a1215) (a1415) (a1421) (b3)}
[ postaction=decorate, decoration={markings,
  mark= at position 0.11 with \strarrow, mark= at position 0.27 with \strarrow,
  mark= at position 0.4 with \strarrow, mark= at position 0.505 with \strarrow,
  mark= at position 0.605 with \strarrow, mark= at position 0.74 with \strarrow,
  mark= at position 0.915 with \strarrow }];


\foreach \n in {1,2,3,4,5,6,7} {\draw (b\n) circle(\bdrydotrad) [fill=gray];}


\foreach \n/\m in {1/567,2/671,3/712,4/123,5/234,6/345,7/456}
{ \draw (\qstart-\seventh*\n:\bigrad) node (q\m) {$\m$}; }

\foreach \m/\a/\r in {156/104/0.58 , 157/63/0.47, 145/160/0.25, 147/15/0.3, 245/220/0.52, 124/295/0.4 }
{ \draw (\a:\r*\bigrad) node (q\m) {$\m$}; }


\foreach \t/\h/\a in {671/712/20, 712/123/20, 345/456/20, 456/567/20, 671/567/-20, 234/123/-20,
  345/234/-20, 567/156/5,156/456/2, 156/157/18, 157/671/3, 147/157/-8, 712/147/-15, 157/145/3,
  145/156/-3, 145/147/-4, 145/245/-25, 456/145/-7, 245/345/-10, 245/124/-2, 124/145/12, 147/124/2,
  124/234/0, 234/245/0, 123/124/-8, 124/712/-3}
{ \draw [quivarrow]  (q\t) edge [bend left=\a] (q\h); }

 \end{tikzpicture}
\]
\caption{The quiver of the Postnikov diagram in Figure~\ref{f:postfree37}
(see Remark~\ref{r:plabicdual})}
\label{f:postfreequiver37}
\end{figure}
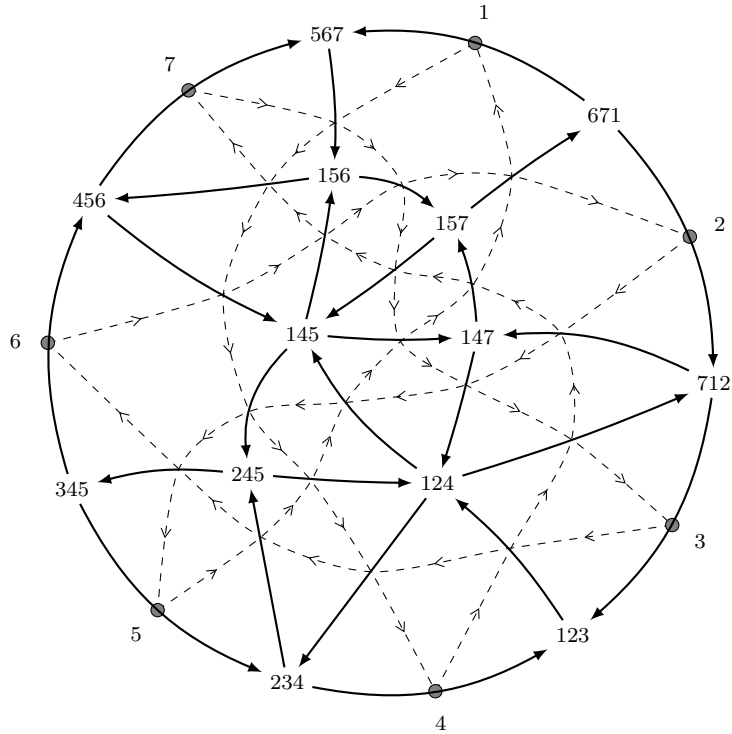

\begin{definition} \label{d:plabic}
A \emph{plabic graph}~\cite[\S11]{postnikov} is a planar
graph embedded into a disk with $n$ vertices on the boundary, each of degree $1$,
and a colouring of the internal vertices with two colours (which we take to be black and white).
In this paper, we will additionally assume that the graph is bipartite,
i.e.\ the end points of internal edges have different colours,
and that no internal vertex has degree $1$.
Note that the boundary vertices may best be considered as the mid-points of half-edges,
which we also call boundary edges.
\end{definition}

Postnikov~\cite[\S14]{postnikov} makes the following definition (see also~\cite[2.1]{GK}).

\begin{definition} \label{d:plabicgraph}
To any Postnikov diagram $D$, there is an associated plabic graph $G(D)$, defined as follows.
The boundary vertices are those of $D$, while the internal vertices correspond to the oriented regions of $D$ and are coloured black or white when the boundary of the region is oriented anticlockwise or clockwise, respectively.
The internal edges of $G(D)$ correspond to the points of intersection of pairs of oriented regions.
For each oriented boundary region of $D$, there is a boundary edge between the vertex corresponding to that region
and the boundary point that it touches.
\end{definition}

The graph $G(D)$ can be embedded in the disk, with each internal vertex mapped to a point inside its corresponding oriented region and internal edges drawn as arcs passing through the two oriented regions and their point of intersection. A boundary edge corresponding to
a boundary oriented region is drawn as an arc inside this region joining the corresponding
internal vertex to the boundary vertex.
Thus $G(D)$ is indeed a plabic graph;
Figure~\ref{f:plabic37} shows an embedded graph
for the Postnikov diagram in Figure~\ref{f:postfree37}.

\begin{figure}
\[
\begin{tikzpicture}[scale=1.1,baseline=(bb.base),
 strand/.style={black,dashed},
 bipedge/.style={black, thick}]

\newcommand{\strarrow}{\arrow{angle 60}}
\newcommand{\bstart}{125} 
\newcommand{\seventh}{51.4} 
\newcommand{\qstart}{150.7} 
\newcommand{\bigrad}{4cm} 
\newcommand{\eps}{11pt} 
\newcommand{\dotrad}{0.1cm} 
\newcommand{\bdrydotrad}{{0.8*\dotrad}} 
\newcommand{\bcolor}{black} 

\path (0,0) node (bb) {};


\draw (0,0) circle(\bigrad) [thick,gray, densely dotted];

\foreach \n in {1,...,7}
{ \coordinate (b\n) at (\bstart-\seventh*\n:\bigrad);
  \draw (\bstart-\seventh*\n:\bigrad+\eps) node {$\n$}; }


\foreach \n/\a/\r in {8/77/0.79, 10/130/0.5, 12/60/0.2, 14/350/0.5, 16/220/0.3, 18/225/0.75, 20/280/0.75,
    9/117/0.77, 11/92/0.38, 13/30/0.7, 15/290/0.05, 17/185/0.7, 19/250/0.55,  21/330/0.75}
{ \coordinate (b\n)  at (\a:\r*\bigrad); }


\foreach \h/\t in {1/8, 2/13, 3/21, 4/20, 5/18, 6/17, 7/9, 8/9, 8/11, 8/13, 9/10, 10/11, 10/17, 11/12,
 12/13, 12/15, 13/14, 14/15, 14/21, 15/16, 16/17, 17/18, 16/19, 18/19, 19/20, 20/21}
{ \draw [bipedge] (b\h)--(b\t); }


\foreach \n in {8,10,12,14,16,18,20} {\draw [\bcolor] (b\n) circle(\dotrad) [fill=\bcolor];}

\foreach \n in {9,11,13,15,17,19,21} {\draw [\bcolor] (b\n) circle(\dotrad) [fill=white];}


\foreach \e/\f/\t in {8/9/0.4, 9/10/0.5, 8/11/0.4, 10/11/0.5, 11/12/0.5, 8/13/0.5, 12/13/0.65, 13/14/0.4, 12/15/0.5,
 14/15/0.5, 15/16/0.5, 16/17/0.65, 10/17/0.6, 17/18/0.45, 18/19/0.5, 19/20/0.5, 20/21/0.5, 16/19/0.5, 14/21/0.5}
{\coordinate (a\e\f) at (${\t}*(b\e) + {1-\t}*(b\f)$); }


\draw [strand] plot[smooth]
coordinates {(b1) (a89) (a910) (a1017) (a1617) (a1619) (a1920) (b4)}
[ postaction=decorate, decoration={markings,
  mark= at position 0.11 with \strarrow, mark= at position 0.255 with \strarrow,
  mark= at position 0.37 with \strarrow, mark= at position 0.52 with \strarrow,
  mark= at position 0.655 with \strarrow, mark= at position 0.775 with \strarrow,
  mark= at position 0.92 with \strarrow }];
\draw [strand] plot[smooth] coordinates {(b2) (a1314) (a1415) (a1516) (a1617) (a1718) (b5)}
[ postaction=decorate, decoration={markings,
  mark= at position 0.13 with \strarrow, mark= at position 0.29 with \strarrow,
  mark= at position 0.47 with \strarrow, mark= at position 0.61 with \strarrow,
  mark= at position 0.75 with \strarrow, mark= at position 0.9 with \strarrow }];
\draw [strand] plot[smooth] coordinates {(b3) (a2021) (a1920) (a1819) (a1718) (b6)}
[ postaction=decorate, decoration={markings,
  mark= at position 0.125 with \strarrow, mark= at position 0.34 with \strarrow,
  mark= at position 0.52 with \strarrow, mark= at position 0.68 with \strarrow,
  mark= at position 0.86 with \strarrow }];
\draw [strand] plot[smooth] coordinates {(b4) (a2021) (a1421) (a1314) (a1213) (a1112) (a1011) (a910) (b7)}
[ postaction=decorate, decoration={markings,
  mark= at position 0.11 with \strarrow, mark= at position 0.27 with \strarrow,
  mark= at position 0.4 with \strarrow, mark= at position 0.53 with \strarrow,
  mark= at position 0.63 with \strarrow, mark= at position 0.725 with \strarrow,
  mark= at position 0.81 with \strarrow, mark= at position 0.92 with \strarrow, }];
\draw [strand] plot[smooth] coordinates {(b5) (a1819) (a1619) (a1516) (a1215) (a1213) (a813) (b1)}
[ postaction=decorate, decoration={markings,
  mark= at position 0.1 with \strarrow, mark= at position 0.24 with \strarrow,
  mark= at position 0.35 with \strarrow, mark= at position 0.46 with \strarrow,
  mark= at position 0.58 with \strarrow, mark= at position 0.735 with \strarrow,
  mark= at position 0.91 with \strarrow }];
  \draw [strand] plot[smooth] coordinates {(b6) (a1017) (a1011) (a811) (a813) (b2)}
[ postaction=decorate, decoration={markings,
  mark= at position 0.14 with \strarrow, mark= at position 0.355 with \strarrow,
  mark= at position 0.5 with \strarrow, mark= at position 0.65 with \strarrow,
  mark= at position 0.86 with \strarrow }];
\draw [strand] plot[smooth] coordinates {(b7) (a89) (a811) (a1112) (a1215) (a1415) (a1421) (b3)}
[ postaction=decorate, decoration={markings,
  mark= at position 0.11 with \strarrow, mark= at position 0.27 with \strarrow,
  mark= at position 0.4 with \strarrow, mark= at position 0.505 with \strarrow,
  mark= at position 0.605 with \strarrow, mark= at position 0.74 with \strarrow,
  mark= at position 0.915 with \strarrow }];


\foreach \n in {1,2,3,4,5,6,7} {\draw (b\n) circle(\bdrydotrad) [fill=gray];}


\foreach \n/\m/\r in {1/567/0.88, 2/671/0.87, 3/712/0.8, 4/123/0.83, 5/234/0.8, 6/345/0.85, 7/456/0.79}
{ \draw (\qstart-\seventh*\n:\r*\bigrad) node (q\m) {$\m$}; }

\foreach \m/\a/\r in {156/104/0.58 , 157/63/0.47, 145/160/0.25, 147/15/0.3, 245/220/0.52, 124/295/0.4 }
{ \draw (\a:\r*\bigrad) node (q\m) {$\m$}; }

 \end{tikzpicture}
\]
\caption{The plabic graph corresponding to the Postnikov diagram in Figure~\ref{f:postfree37}}
\label{f:plabic37}
\end{figure}
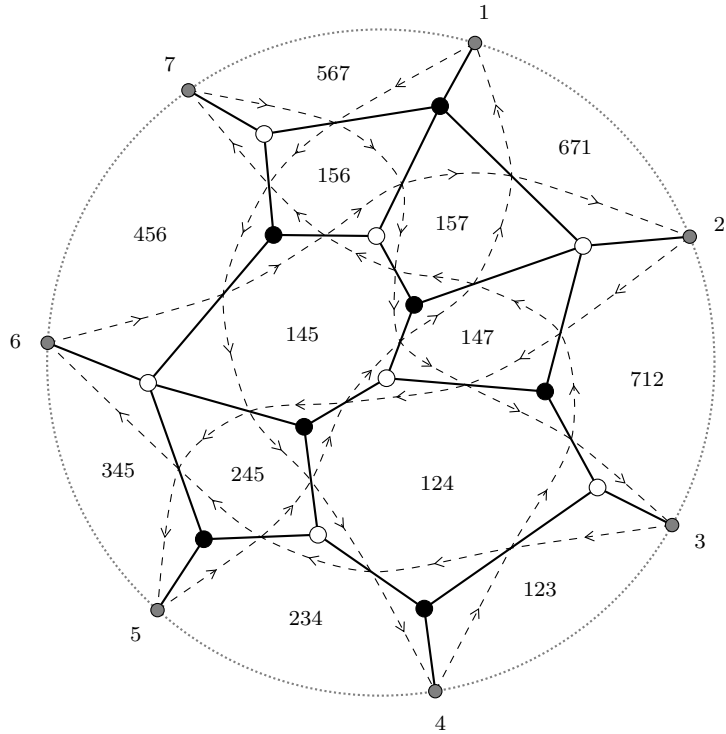

\begin{remark}
Figures~\ref{f:postfreequiver37} and~\ref{f:plabic37} may be considered as pictures of a dimer model,
or bipartite field theory (in the sense of Franco \cite{francopre12}), in a disk.
We will make this more precise in the next section, adapting the
more quiver focussed formalism of
Bocklandt \cite{bocklandt,bocklandt12}, Davison~\cite{davison11} and Broomhead~\cite{broomhead12} to the boundary case.
\end{remark}

\section{Dimer models with boundary} 
\label{s:dimermodels}

In this section, we formalise the notion of a dimer model with boundary
and show how the quiver of a Postnikov diagram can be
interpreted as a dimer model in a disk.
Given a quiver $Q$, we write $\Qcyc$ for the set of oriented cycles in $Q$
(up to cyclic equivalence). We start with a more general definition.

\begin{definition}
A \emph{quiver with faces} is a quiver $Q=(Q_0,Q_1)$, together with a set $Q_2$ of faces
and a map $\bdry\colon Q_2\to \Qcyc$,
which assigns to each $F\in Q_2$ its \emph{boundary} $\bdry F\in \Qcyc$.
\end{definition}

We shall often denote a quiver with faces by the same letter $Q$, regarded now as the
triple $(Q_0,Q_1,Q_2)$. We say that $Q$ is \emph{finite} if $Q_0,Q_1$
and $Q_2$ are all finite sets.
The number of times an arrow $\alpha\in Q_1$ appears in the boundaries of the faces in
$Q_2$ will be called the \emph{face multiplicity} of $\alpha$.
The (unoriented) \emph{incidence graph} of $Q$, at a vertex $i\in Q_0$,
has vertices given by the arrows incident with $i$. The edges between two
arrows $\alpha,\beta$ correspond to the paths of the form
$$
\xymatrix{
\ar[r]^{\alpha} & i \ar[r]^{\beta} &
}
$$
occurring in a cycle bounding a face.

\begin{definition} \label{d:dimermodel}
A (finite, oriented) \emph{dimer model with boundary} is given by a finite
quiver with faces $Q=(Q_0,Q_1,Q_2)$, where $Q_2$ is written as disjoint union $Q_2=Q_2^+\cup Q_2^-$,
satisfying the following properties:
\begin{enumerate}[(a)]
\item the quiver $Q$ has no loops, i.e.\ no 1-cycles, but 2-cycles are allowed,
\item all arrows in $Q_1$ have face multiplicity $1$ (\emph{boundary} arrows) or $2$
(\emph{internal} arrows),
\item each internal arrow lies in a cycle bounding a face
in $Q_2^+$ and in a cycle bounding a face in $Q_2^-$,
\item the incidence graph of $Q$ at each vertex is connected.
\end{enumerate}
Note that, by (b), each incidence graph in (d) must be either a line
(at a \emph{boundary} vertex) or an unoriented cycle (at an \emph{internal} vertex).
\end{definition}

\begin{remark} \label{r:dimermod}
We will only encounter oriented dimer models in this paper, but it is possible to consider
unoriented ones by not writing $Q_2$ as a disjoint union and dropping condition~(c).
We will also only encounter finite dimer models,
but infinite dimer models can also be considered,
e.g.\ the universal cover of any finite dimer model on a torus.
One should then add to (d) the condition that each incidence graph is finite,
so $Q$ is `locally finite'.
We choose not to
require that the quiver $Q$ is connected. However, note that, if it is,
then it is actually strongly connected, because every arrow is contained in a face,
whose boundary also includes a path going in the opposite direction
(cf. \cite[Def.~6.1]{bocklandt}).
Condition~(a) is included to avoid unpleasant degeneracies in Definition~\ref{d:dimeralgebra}.
\end{remark}

If we realise each face $F$ of a dimer model $Q$ as a polygon,
whose edges are labelled (cyclically) by the arrows in $\bdry F$,
then we may, in the usual way, form a topological space $|Q|$ by gluing together
the edges of the polygons labelled by the same arrows, in the manner indicated by the directions of the arrows.
Then, arguing as in~\cite[Lemma 6.4]{bocklandt}, we see that conditions (b) and (d)
ensure that $|Q|$ is a surface with boundary, while (c) means that it can be oriented
by declaring the boundary cycles of faces in $Q_2^+$ to be oriented positive (or anticlockwise)
and those of faces in $Q_2^-$ to be negative (or clockwise).
Note also that each component of the boundary of $|Q|$ is (identified with) an
unoriented cycle of boundary arrows in $Q$.
If $Q$ is a dimer model with boundary, for which $|Q|$ is homeomorphic to a disk,
then we will call $Q$ a \emph{dimer model in a disk}.

On the other hand, suppose that we are given an embedding of a finite quiver $Q=(Q_0,Q_1)$
into a compact (oriented) surface $\Sigma$ with
boundary, such that the complement of $Q$ in $\Sigma$ is a disjoint union of disks, each of which
is bounded by a cycle in $Q$. Then we may make $Q$ into an (oriented) dimer model in the above
sense, for which $|Q|\cong \Sigma$, by setting $Q_2$ to be the set of connected components of the
complement of $Q$ in $\Sigma$, which can be separated into $Q_2^+$ and $Q_2^-$
when $\Sigma$ is oriented.

\begin{remark} \label{r:plabicdual}
By Remark~\ref{r:quiverembedding}, we have precisely such an embedding
of the quiver $Q(D)$, associated to a Postnikov diagram $D$ in a disk (see Figure~\ref{f:postfreequiver37}).
Thus $Q(D)$ can be considered to be not just a quiver, but actually
a dimer model in a disk in the above sense.
As is well-known, the Postnikov diagram can be reconstructed by drawing strand segments inside each face of $Q(D)$ (in an embedding into a disk) from the mid-point of each arrow to the mid-point of the next arrow in the cycle which is
the boundary of the face, oriented in the same direction.
(In fact, the strands correspond to the zig-zag paths in the disk; see~\cite[\S5]{bocklandt12},~\cite[\S\S1.7, 2.1]{GK}).
Note that carrying out this procedure for an arbitrary dimer model in a disk will give a diagram satisfying the local axioms
(a1) -- (a3) of Definition~\ref{d:asd}, but not necessarily satisfying the
global axioms (b1) and (b2).

We may also describe $Q(D)$, as a quiver with faces, directly and more combinatorially
as the dual of the plabic graph $G(D)$,
as in \cite[\S 2.1]{francopre12} for a general bipartite field theory.
In other words, $Q_0(D)$ is in bijection with the set of faces of $G(D)$ and $Q_1(D)$ with the set of edges,
with boundary arrows corresponding to boundary edges.
An arrow joins the
two faces in $G(D)$ that share the corresponding edge and is oriented so that the black vertex is
on the left and/or the white vertex is on the right.
The faces (plaquettes in \cite{francopre12}) $F\in Q^+_2(D)$ correspond to the internal black vertices, while
those in $Q^-_2(D)$ correspond to the white vertices.
The boundary $\bdry F$ is given by the arrows corresponding
to the edges incident with the internal vertex of $G(D)$ corresponding to $F$,
ordered anticlockwise round black vertices and clockwise round white ones.
This duality is illustrated in Figure~\ref{f:plabicdual}, for $Q(D)$ as in Figure~\ref{f:postfreequiver37} and
$G(D)$ as in Figure~\ref{f:plabic37}.
\end{remark}

\begin{figure}
\[
\begin{tikzpicture}[scale=1.1,baseline=(bb.base),
 quivarrow/.style={black, -latex, thick},
 bipedge/.style={black, thick}]

\newcommand{\bstart}{125} 
\newcommand{\seventh}{51.4} 
\newcommand{\qstart}{150.7} 
\newcommand{\bigrad}{4cm} 
\newcommand{\eps}{11pt} 
\newcommand{\dotrad}{0.1cm} 
\newcommand{\bdrydotrad}{{0.8*\dotrad}} 
\newcommand{\bcolor}{black} 

\path (0,0) node (bb) {};

\foreach \n in {1,...,7}
{ \coordinate (b\n) at (\bstart-\seventh*\n:\bigrad);
  \draw (\bstart-\seventh*\n:\bigrad+\eps) node {$\n$}; }


\foreach \n/\a/\r in {8/77/0.79, 10/130/0.5, 12/60/0.2, 14/350/0.5, 16/220/0.3, 18/225/0.75, 20/280/0.75,
    9/117/0.77, 11/92/0.38, 13/30/0.7, 15/290/0.05, 17/185/0.7, 19/250/0.55,  21/330/0.75}
{ \coordinate (b\n)  at (\a:\r*\bigrad); }


\foreach \h/\t in {1/8, 2/13, 3/21, 4/20, 5/18, 6/17, 7/9, 8/9, 8/11, 8/13, 9/10, 10/11, 10/17, 11/12,
 12/13, 12/15, 13/14, 14/15, 14/21, 15/16, 16/17, 17/18, 16/19, 18/19, 19/20, 20/21}
{ \draw [bipedge] (b\h)--(b\t); }


\foreach \n in {8,10,12,14,16,18,20} {\draw [\bcolor] (b\n) circle(\dotrad) [fill=\bcolor];}

\foreach \n in {9,11,13,15,17,19,21} {\draw [\bcolor] (b\n) circle(\dotrad) [fill=white];}


\foreach \e/\f/\t in {8/9/0.4, 9/10/0.5, 8/11/0.4, 10/11/0.5, 11/12/0.5, 8/13/0.5, 12/13/0.65, 13/14/0.4, 12/15/0.5,
 14/15/0.5, 15/16/0.5, 16/17/0.65, 10/17/0.6, 17/18/0.45, 18/19/0.5, 19/20/0.5, 20/21/0.5, 16/19/0.5, 14/21/0.5}
{\coordinate (a\e\f) at (${\t}*(b\e) + {1-\t}*(b\f)$); }


\foreach \n in {1,2,3,4,5,6,7} {\draw (b\n) circle(\bdrydotrad) [fill=gray];}


\foreach \n/\m in {1/567,2/671,3/712,4/123,5/234,6/345,7/456}
{ \draw (\qstart-\seventh*\n:\bigrad) node (q\m) {$\m$}; }

\foreach \m/\a/\r in {156/104/0.58 , 157/63/0.47, 145/160/0.25, 147/15/0.3, 245/220/0.52, 124/295/0.4 }
{ \draw (\a:\r*\bigrad) node (q\m) {$\m$}; }


\foreach \t/\h/\a in {671/712/20, 712/123/20, 345/456/20, 456/567/20, 671/567/-20, 234/123/-20,
  345/234/-20, 567/156/5,156/456/2, 156/157/18, 157/671/3, 147/157/-8, 712/147/-15, 157/145/3,
  145/156/-3, 145/147/-4, 145/245/-25, 456/145/-7, 245/345/-10, 245/124/-2, 124/145/12, 147/124/2,
  124/234/0, 234/245/0, 123/124/-8, 124/712/-3}
{ \draw [quivarrow]  (q\t) edge [bend left=\a] (q\h); }

 \end{tikzpicture}
\]
\caption{The quiver and plabic graph associated to the Postnikov diagram in Figure~\ref{f:postfree37}}
\label{f:plabicdual}
\end{figure}
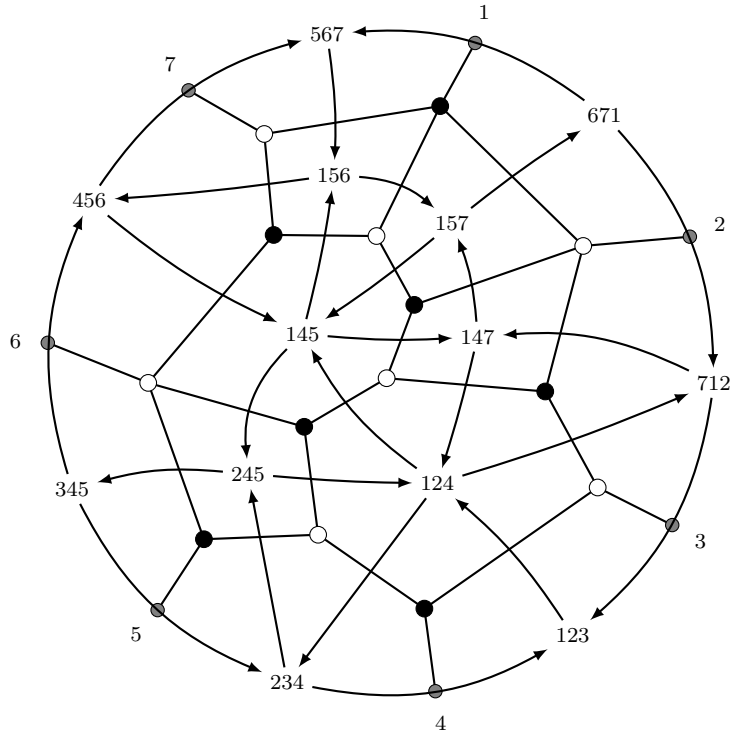

\begin{definition} \label{d:dimeralgebra}
Given a dimer model with boundary $Q$,
we define the \emph{dimer algebra} $A_Q$
as follows.
For each internal arrow $\alpha\in Q_1$,
there are (unique) faces $F^+\in Q_2^+$ and $F^-\in Q_2^-$
such that $\bdry F^{\pm}=\alpha p^{\pm}_{\alpha}$,
for paths $p^+_{\alpha}$ and $p^-_{\alpha}$ from $h\alpha$ to $t\alpha$.
Then the dimer algebra $A_Q$ is the quotient of the path algebra
$\F Q$ by the relations
\begin{equation} \label{e:definingrelations}
p^+_{\alpha}=p^-_{\alpha},
\end{equation}
for every internal arrow $\alpha\in Q_1$.
\end{definition}

\begin{remark} \label{r:potential}
Note that the orientation is not strictly necessary to define $A_Q$;
we only need to know that $F^{\pm}$ are the two faces that contain the internal arrow
$\alpha$ in their boundaries, but not which is which.
On the other hand, given the orientation, we may also define a (super)potential $W_Q$ by the usual formula
(e.g.\ \cite[\S 2]{fhkvw})
\[
  W_Q=\sum_{F\in Q_2^+}\bdry F-\sum_{F\in Q_2^-}\bdry F,
\]
defined up to cyclic equivalence.
Then $A_Q$ may also be described as the quotient of
the path algebra $\F Q$ by the so-called `F-term' relations
$$\partial_{\alpha}(W_Q)=0,$$
for each \emph{internal} arrow $\alpha$ in $Q$, where $\partial_{\alpha}$ is the usual cyclic derivative
(e.g.~\cite[\S1.3]{ginz} or \cite[\S3]{bocklandt}).
Thus, in the absence of boundary arrows in $Q$, the algebra $A_Q$
is the usual Jacobi (or superpotential) algebra
(e.g.~\cite[\S3]{bocklandt}, \cite[\S2.1.3]{broomhead12}).

In the boundary case, the idea of only considering F-term relations for internal arrows has arisen independently in work of Franco~\cite[\S 6.1]{francopre12}
and Buan-Iyama-Reiten-Smith~\cite[Defn.~1.1]{BIRS11}.
In the latter case, a slightly different approach is used, whereby any arrow joining two boundary (or frozen) vertices
is considered to be frozen and hence does not contribute an F-term relation, while
in our case, we may have internal arrows with both end-points being boundary vertices (see Figures~\ref{f:twistquiver} and~\ref{f:annulus}).
Buan-Iyama-Reiten-Smith~\cite{BIRS11} give a description~\cite[Thm.~6.6]{BIRS11} of the endomorphism algebras of some cluster-tilting objects over preprojective algebras as frozen
Jacobian algebras in the sense of~\cite[Defn~1.1]{BIRS11}.
Demonet-Luo~\cite{demonetluo} give a $2$-Calabi-Yau categorification $\mathcal{C}$ of the Grassmannian $Gr(2,n)$ using frozen Jacobian algebras in the sense of~\cite[Defn.~1.1]{BIRS11}; these algebras are the endomorphism algebras of cluster-tilting objects in $\C$ (see \cite[Thm.~1.3]{demonetluo}).
\end{remark}

\begin{definition} \label{d:centralelement}
We write $A_D$ for the dimer algebra $A_{Q(D)}$ associated to the dimer model $Q(D)$.
It follows from the defining relations that, for any vertex $I\in Q_0(D)$,
the product in $A_D$ of the arrows in any cycle that starts at $I$ and bounds
a face is the same. We denote this element by $u_I$, and write
\begin{equation} \label{e:defu}
  u=\sum_{I\in Q_0(D)} u_I.
\end{equation}
It similarly follows from the relations that $u$ commutes with every arrow
and hence is in the centre of $A_D$.
\end{definition}

\begin{remark} \label{r:commutation}
A dimer algebra is a special case of an algebra defined by a quiver $Q$ with commutation relations,
that is, it is a quotient $\F Q/I$, where the ideal $I$ is generated by $\{p_i-q_i : i\in \I\}$ for
paths $p_i$ and $q_i$ in $Q$ with the same start and end points. 
Any such algebra has a couple of elementary properties, which we note down for future reference
and provide proofs for the convenience of the reader, although these properties seem well-known and `obvious'.
Firstly, 
\begin{enumerate}
\item[(a)]  every path in $Q$ gives a non-zero element of $\F Q/I$. 
\end{enumerate}
This is an immediate corollary of a stronger property, that requires one to first observe that
commutation relations define a natural equivalence relation $\sim$ on the set of paths
in $Q$, generated by requiring that $p\sim q$ if $p$ has a subpath $p_i$ and $q$ is obtained 
from $p$ by replacing $p_i$ with $q_i$, for some $i\in\I$. 
Then, secondly,
\begin{enumerate}
\item[(b)]  the equivalence classes of $\sim$ form a basis of $\F Q/I$.
\end{enumerate}
Note that any equivalence class $\overline{p}$ of paths does determine a well-defined element $p+I$ of $\F Q/I$
and these elements evidently span. To see that they are independent, observe that there is a well-defined
algebra $\F(Q/\qrelation)$ with basis given by the set of equivalence classes of $\sim$, 
with multiplication given by concatenation, where possible, and zero otherwise, extended linearly.
The natural map 
\[ \pi\colon \F Q \to \F(Q/\qrelation) \colon p\mapsto \overline{p}
\]
has each $p_i-q_i$, for $i\in \I$, in its kernel and so induces a map
$\overline{\pi}\colon \F Q/I \to \F(Q/\qrelation)$, which is the inverse of the map $\overline{p}\mapsto p+I$.

Alternatively, (a) has a direct proof as follows.
Let $M$ be the algebra of square matrices over $\F$ of size
$|Q_0|$.
Then (numbering of the vertices $Q_0$)
there is a morphism of algebras $\theta\colon \F Q\to M$ 
taking the idempotent $e_i$ to the elementary matrix $E_{ii}$ and any arrow from $i$ to $j$ to the matrix $E_{ji}$.
The kernel of $\theta$ contains every
possible commutation relation, and hence
contains $I$, so $\theta$ induces an algebra
morphism $\overline{\theta}\colon \F Q/I \to M$. However,
$\overline{\theta}$ sends every path in $Q$ to a
non-zero element of $M$.
\end{remark}

\section{Weights} 
\label{s:weights}

In this section, we introduce a weighting on the arrows in the quiver
$Q=Q(D)$ of a Postnikov diagram $D$, via elements of $\mathbb{N}C_0$.
Our main aim is to compute the weight of the boundary of a face of $Q(D)$ and
the sum of the weights of the arrows incident with a vertex of $Q(D)$.

These results will be then used in Section~\ref{s:legalarrow} in order to find the
first step in a path between any pair of vertices in $Q_0$ whose weight does
not include every element of $C_0$; we call such a path
a \emph{minimal path}. Such a minimal path itself will be constructed in
Section~\ref{s:minimalpath}, and is a key component in the proof of
surjectivity of the morphism we shall construct in Section~\ref{s:isomorphism}
from the total algebra to the endomorphism algebra.

\begin{definition}
\label{d:weight}
For any arrow $\alpha\colon I\to J$ in $Q_1(D)$,
let $c\in C_1$ be the number of the strand crossing $\alpha$ from right to left
and $d\in C_1$ the number of the strand crossing $\alpha$ from left to right.
In other words, $J=I-c+d$. Identifying a subset of $C_0$ with the sum of
its elements in $\mathbb{N}C_0$, we give $\alpha$ the \emph{weight}
$$\weight_{\alpha}=(c,d)_0;$$
see Figure~\ref{f:arrowconvention2}. The weight of a path in $Q(D)$ is then
defined to be the sum of the weights of the arrows in the path.
\end{definition}

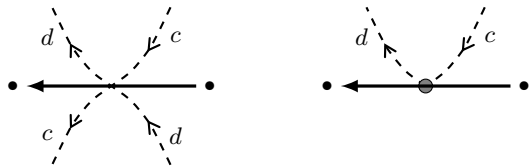
\begin{figure}[b]
\[
\begin{tikzpicture}[scale=1.3,baseline=(bb.base),
  strand/.style={black,dashed,thick},
  quivarrow/.style={black, -latex, very thick}]
\newcommand{\strarrow}{\arrow{angle 60}}
\path (0,0) node (bb) {}; 

\draw [strand] plot[smooth]
coordinates {(0.6,0.8) (0.4,0.4) (0.2,0.15) (0,0) (-0.2,-0.15) (-0.4,-0.4) (-0.6,-0.8)}
[ postaction=decorate, decoration={markings,
  mark= at position 0.25 with \strarrow, mark= at position 0.8 with \strarrow}];
\draw [strand] plot[smooth]
coordinates {(0.6,-0.8) (0.4,-0.4) (0.2,-0.15) (0,0) (-0.2,0.15) (-0.4,0.4) (-0.6,0.8)}
[ postaction=decorate, decoration={markings,
  mark= at position 0.25 with \strarrow, mark= at position 0.8 with \strarrow}];

\draw (0.65,0.5) node {\small $c$};
\draw (0.65,-0.5) node {\small $d$};
\draw (-0.65,0.5) node {\small $d$};
\draw (-0.65,-0.5) node {\small $c$};
\draw (-1,0) node (h) {$\bullet$};
\draw (1,0) node (t) {$\bullet$};
\draw [quivarrow] (t)--(h);

\end{tikzpicture}
\qquad\qquad
\begin{tikzpicture}[scale=1.3,baseline=(bb.base),
  strand/.style={black,dashed,thick},
  quivarrow/.style={black, -latex, very thick}]
\newcommand{\strarrow}{\arrow{angle 60}}
\newcommand{\bdrydotrad}{0.07cm} 
\path (0,0) node (bb) {}; 

\draw [strand] plot[smooth]
coordinates {(0.6,0.8) (0.4,0.4) (0.2,0.15) (0,0)}
[ postaction=decorate, decoration={markings,
  mark= at position 0.5 with \strarrow}];
\draw [strand] plot[smooth]
coordinates {(0,0) (-0.2,0.15) (-0.4,0.4) (-0.6,0.8)}
[ postaction=decorate, decoration={markings,
  mark= at position 0.6 with \strarrow}];

\draw (0.65,0.5) node {\small $c$};
\draw (-0.65,0.5) node {\small $d$};
\draw (0,0) circle(\bdrydotrad) [fill=gray];
\draw (-1,0) node (h) {$\bullet$};
\draw (1,0) node (t) {$\bullet$};
\draw [quivarrow] (t)--(h);
\end{tikzpicture}
\]
\caption{Internal and boundary arrows of weight $(c,d)_0$}
\label{f:arrowconvention2}
\end{figure}

\begin{remark}
Note that by definition the weight of an arrow cannot
be the whole of $C_0$. It also cannot be zero (i.e.\ the empty set) because crossing strands are distinct by Definition~\ref{d:asd}(b1).
The weights of the arrows in the quiver in
Figure~\ref{f:postfreequiver37} are shown in Figure~\ref{f:postfreelabels37}.
\end{remark}

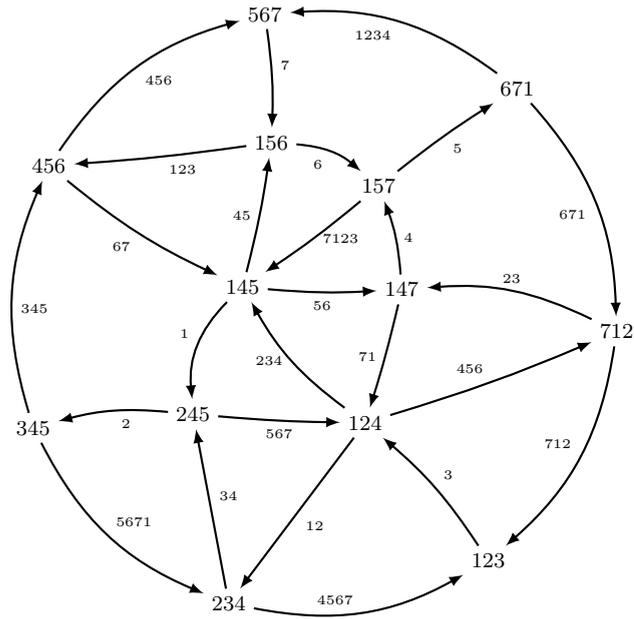
\begin{figure}
\[
\begin{tikzpicture}[scale=1.0,baseline=(bb.base), auto,
 quivarrow/.style={black, -latex, thick}]

\newcommand{\bstart}{125} 
\newcommand{\seventh}{51.4} 
\newcommand{\qstart}{150.7} 
\newcommand{\bigrad}{4cm} 
\newcommand{\eps}{11pt} 

\path (0,0) node (bb) {};


\foreach \n/\a/\r in {8/77/0.79, 10/130/0.5, 12/60/0.2, 14/350/0.5, 16/220/0.3, 18/225/0.75, 20/280/0.75,
    9/117/0.77, 11/92/0.38, 13/30/0.7, 15/290/0.05, 17/185/0.7, 19/250/0.55,  21/330/0.75}
{ \coordinate (b\n)  at (\a:\r*\bigrad); }


\foreach \n/\m in {1/567,2/671,3/712,4/123,5/234,6/345,7/456}
{ \draw (\qstart-\seventh*\n:\bigrad) node (q\m) {$\m$}; }

\foreach \m/\a/\r in {156/104/0.58 , 157/63/0.47, 145/160/0.25, 147/15/0.3, 245/220/0.52, 124/295/0.4 }
{ \draw (\a:\r*\bigrad) node (q\m) {$\m$}; }


\foreach \t/\h/\a/\lab in {671/567/-20/1234, 234/123/-20/4567, 345/234/-20/5671,
  567/156/5/7, 156/456/2/123, 245/345/-10/2, 124/234/0/12, 124/712/-3/456}
{ \draw [quivarrow]  (q\t) edge [bend left=\a] node {\tiny\lab} (q\h) ; }


\foreach \t/\h/\a/\lab in {671/712/20/671, 712/123/20/712, 345/456/20/345, 456/567/20/456,
  456/145/-7/67, 234/245/0/34, 123/124/-8/3, 157/671/3/5,
  145/245/-25/1, 245/124/-2/567, 145/147/-4/56, 156/157/18/6}
{ \draw [quivarrow]  (q\t) edge [bend left=\a] node [swap] {\tiny\lab} (q\h) ; }


\foreach \t/\h/\a/\lab/\posit in {157/145/3/7123/right, 124/145/12/234/left, 145/156/-3/45/left,
 147/124/2/71/left, 147/157/-8/4/right, 712/147/-15/23/above}
{ \draw [quivarrow]  (q\t) edge [bend left=\a] node [\posit] {\tiny\lab} (q\h) ; }

 \end{tikzpicture}
\]
\caption{Weights on the quiver of the Postnikov diagram in Figure~\ref{f:postfreequiver37}}
\label{f:postfreelabels37}
\end{figure}

\begin{lemma} \label{l:faceordering}
Let $F$ be a face in $Q_2(D)$.
The ordering of the arrows in $\bdry F$ induces a cyclic
ordering on the strands in $D$ entering $F$ on its boundary.
Then the starting points of these strands appear on the boundary
of the disk in the same order. The same result holds for the end points
of the strands exiting $F$.
\end{lemma}

\begin{proof}
We argue as in the proof of necessity in~\cite[Thm. 6.6]{bocklandt12}.
For an arrow $\alpha$ in $\bdry F$, let $\R_{\alpha}$ denote the part
of the strand entering $F$ at $\alpha$ from its starting point until its crossing
point with $\alpha$ (we refer to this as a \emph{backward ray}).
Note that if $\alpha$ is a boundary arrow then $R_{\alpha}$ is a single point.

Suppose that $\alpha$ and $\beta$ are arrows in $\bdry F$
and that $R_{\alpha}$ and $R_{\beta}$ cross outside $F$.
Let $\pi(\alpha,\beta)$ be the path in $\bdry F$ strictly between $\alpha$ and $\beta$
which is homotopic to the composition of the path from the intersection of $\alpha$
with $R_{\alpha}$, backwards along $R_{\alpha}$ to its last crossing point with
$R_{\beta}$, and the path from this crossing point to the intersection of $\beta$ and
$R_{\beta}$.

We show by induction on $l$ that there are no crossings outside $F$ between
$R_{\alpha}$ and $R_{\beta}$ for which the length of $\pi(\alpha,\beta)$ is $l$.
If $R_{\alpha}$ and $R_{\beta}$ cross outside $F$ and the length of $\pi(\alpha,\beta)$
is zero, then we have a contradiction to Definition~\ref{d:asd}(b2).
So fix $l\geq 1$ and suppose the result is true for smaller $l$.
If $R_{\alpha}$ and $R_{\beta}$ cross and $\pi(\alpha,\beta)$ has length $l$,
then, for any arrow $\gamma$ in $\pi(\alpha,\beta)$, $R_{\gamma}$ must cross
$R_{\alpha}$ or $R_{\beta}$. Since $\pi(\alpha,\beta)=\pi(\gamma,\beta)\circ\gamma\circ\pi(\alpha,\gamma)$, we see that both $\pi(\alpha,\gamma)$ and $\pi(\gamma,\beta)$
have length less than $l$, so we have a contradiction and $R_{\alpha}$ and $R_{\beta}$
cannot cross.

By induction, none of the backward rays $R_{\alpha}$ cross, and the result follows.
\end{proof}

\begin{figure}
\[
\begin{tikzpicture}[scale=0.8,baseline=(bb.base),
 strand/.style={black, dashed},
 quivarrow/.style={black, -latex, thick}]

\newcommand{\strarrow}{\arrow{angle 60}}
\newcommand{\bstart}{0} 
\newcommand{\pstart}{54} 
\newcommand{\estart}{100} 
\newcommand{\fifth}{72} 
\newcommand{\arrowlabeldisp}{-17} 
\newcommand{\raylabeldisp}{-62} 
\newcommand{\rayrad}{3.5cm} 
\newcommand{\bigrad}{4.5cm} 
\newcommand{\littlerad}{2.5cm} 
\newcommand{\middlerad}{3.5cm} 
\newcommand{\eps}{11pt} 
\newcommand{\dotrad}{0.1cm} 
\newcommand{\bdrydotrad}{{0.8*\dotrad}} 
\newcommand{\outercircle} 

\path (0,0) node (bb) {};


\draw (0,0) circle(\bigrad) [thick,gray,dotted];

\foreach \n in {1,...,5}
{ \coordinate (b\n) at (\bstart+\fifth*\n:\bigrad);
  \draw (\bstart+\fifth*\n:\bigrad+\eps) node {$c_{\n}$}; }


\foreach \n in {1,2,3,4,5} {\draw (b\n) circle(\bdrydotrad) [fill=gray];}

\foreach \n in {1,...,5}
{ \coordinate (p\n) at (\pstart+\fifth*\n:\littlerad);}

\foreach \n in {1,...,5}
{ \coordinate (t\n) at (\estart+\fifth*\n:\middlerad);}

%
\foreach \n/\m in {1/2,2/3,3/4,4/5,5/1} {\draw [quivarrow,shorten >=5pt,shorten <=5pt] (p\n) -- (p\m); }


\foreach \n in {1,2,3,4,5}
\draw (p\n) node {$\bullet$};

%
\foreach \e/\f/\t in {1/2/0.5, 2/3/0.5, 3/4/0.4, 4/5/0.5, 5/1/0.5}
{\coordinate (m\e\f) at (${\t}*(p\e) + {1-\t}*(p\f)$); }

%
\foreach \i/\j/\k in {3/2/1,4/3/2,5/4/3,1/5/4,2/1/5}
{\draw [strand] plot[smooth] coordinates {(b\j) (m\k\j) (m\j\i) (t\j)}
[ postaction=decorate, decoration={markings,
 mark= at position 0.25 with \strarrow, mark= at position 0.6 with \strarrow,
 mark= at position 0.9 with \strarrow}];}

\foreach \n in {1,2,3,4,5}
{\draw (\pstart+\fifth*\n+\arrowlabeldisp:\littlerad) node {$\alpha_{\n}$}; }

\foreach \n in {1,2,3,4,5}
{\draw (\pstart+\fifth*\n+\raylabeldisp:\rayrad) node {$\mathcal{R}_{\alpha_{\n}}$}; }

\end{tikzpicture}
\]
\caption{The ordering of strands around a face}
\label{f:faceordering}
\end{figure}
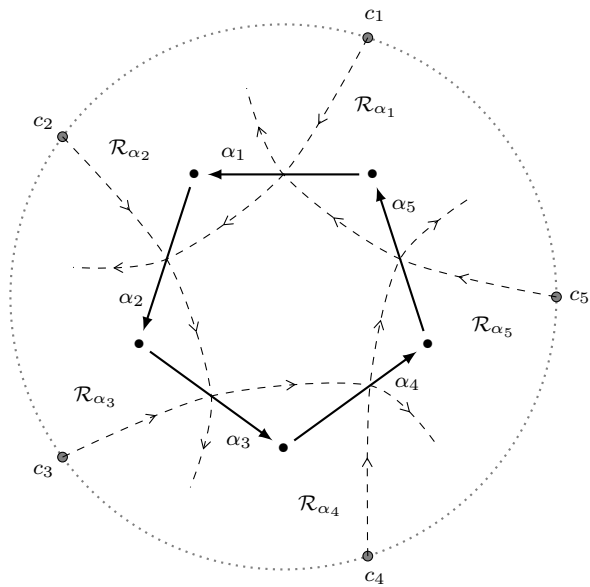

\begin{corollary} \label{c:cycleweight}
For every $F\in Q_2(D)$, the weight of $\bdry F$ is $C_0$.
\end{corollary}

\begin{proof}
Suppose first that $F\in Q_2^+(D)$. Let $\alpha_1,\ldots ,\alpha_r$ be the arrows in
$\bdry(F)$, taken in anticlockwise order.
Let $c_i$ be the number of strand $\R_{\alpha_i}$, for $i=1,\ldots ,r$.
Then the weight of $\alpha_i$ is $(c_i,c_{i-1})_0$ (with subscripts modulo $r$),
and the result follows from Lemma~\ref{l:faceordering}. A similar argument applies
to faces in $Q_2^-(D)$.
\end{proof}

\begin{remark} \label{r:ADgraded}
It follows from Corollary~\ref{c:cycleweight} that the weighting of arrows
$\alpha\mapsto w_{\alpha}$ given in Definition~\ref{d:weight} induces an
$\mathbb{N}C_0$-grading on $A_D$ (see Definition~\ref{d:dimeralgebra}).
\end{remark}

\begin{remark} \label{r:perfectmatchings}
For $i\in C_0$, let
$\mathcal{P}_i=\{\alpha\in Q_1(D) : i\in \weight_{\alpha}\}$.
By Corollary~\ref{c:cycleweight}, the cycle bounding each face
in $Q_2(D)$ contains exactly one arrow from $\mathcal{P}_i$.
Thus each such $\mathcal{P}_i$ may be considered as a perfect matching
on $Q(D)$ or, equivalently, on the dual $G(D)$
(see \cite[\S 2.2]{francopre12}).
\end{remark}

Next, we consider the strands around a vertex in a similar way. We fix
an internal vertex $I$ in $Q_0(D)$, and suppose that there are $2r$ arrows incident with
$I$ in $Q(D)$. We label them $\alpha_i$, $i=1,2,\ldots ,2r$, in order anticlockwise
around $I$, where the $\alpha_i$ for $i$ odd are outgoing arrows and the others are incoming,
and treat the subscripts modulo $2r$.
Let $I_i$ be the end-point of $\alpha_i$ not equal to $I$.
For $i$ odd, we suppose that strand $\newt_i$ crosses $\alpha_i$ from right to
left (looking along the arrow), and $\newt_{i+1}$ crosses $\alpha_i$ from left
to right, so that $\alpha_i$ has weight $(\newt_i,\newt_{i+1})_0$ (again treating
 subscripts modulo $2r$). Note that if $i$ is
even, then $\alpha_i$ also has weight $(\newt_i,\newt_{i+1})_0$.
See Figure~\ref{f:internalincidence} for the case $r=3$. Note that we have:

\begin{equation}
\label{e:Iidescription}
  I_i=\begin{cases} I-\newt_i+\newt_{i+1}, & \text{$i$ odd;} \\
  I-\newt_{i+1}+\newt_{i}, & \text{$i$ even.}
  \end{cases}
\end{equation}

\begin{figure}
\[
\begin{tikzpicture}[scale=0.8,baseline=(bb.base),
 strand/.style={black, dashed},
 quivarrow/.style={black, -latex, thick}]

\newcommand{\strarrow}{\arrow{angle 60}}
\newcommand{\bstart}{0} 
\newcommand{\sstart}{-10} 
\newcommand{\tstart}{15} 
\newcommand{\csstart}{20} 
\newcommand{\asstart}{-76.5} 
\newcommand{\sixth}{60} 
\newcommand{\asrad}{4.2cm} 
\newcommand{\arrowlabeldisp}{-23} 
\newcommand{\raylabeldisp}{-62} 
\newcommand{\rayrad}{3.5cm} 
\newcommand{\bigrad}{4.5cm} 
\newcommand{\littlerad}{1cm} 
\newcommand{\middlerad}{3.5cm} 
\newcommand{\eps}{6pt} 
\newcommand{\dotrad}{0.1cm} 
\newcommand{\bdrydotrad}{{0.8*\dotrad}} 
\newcommand{\arlength}{4.3cm} 

\path (0,0) node (bb) {};

\foreach \n in {1,...,6}
{ \coordinate (b\n) at (\bstart+\sixth*\n:\bigrad);
  \draw (\bstart+\sixth*\n:\bigrad+\eps) node {$I_{\n}$}; }


\draw (0,0) node {$I$};

\draw [quivarrow,shorten >=5pt,shorten <=9pt] (b5) -- ++(\arlength,0);

\foreach \n in {1,...,6}
{ \coordinate (s\n) at (\sstart+\sixth*\n:\bigrad+\eps);}

\foreach \n in {1,...,6}
{ \coordinate (t\n) at (\tstart+\sixth*\n:\bigrad+\eps);}

%
\foreach \n in {1,2,3,4,5,6} {\draw [quivarrow,shorten >=8pt,shorten <=8pt] (0,0) -- (b\n); }

%
\foreach \i/\j in {1/0.5, 2/0.5, 3/0.5, 4/0.5, 5/0.5, 6/0.5}
{ \coordinate (m\i) at (${\j}*(b\i)$); }


\foreach \i/\j in {2/3,6/1}
{\draw [strand] plot[smooth] coordinates {(s\i) (m\i) (m\j) (t\j)}
[ postaction=decorate, decoration={markings,
 mark= at position 0.2 with \strarrow, mark= at position 0.5 with
\strarrow,
 mark= at position 0.8 with \strarrow}];}

\coordinate (mm) at ($(b5)+(0.5*\arlength,0)$);
\coordinate (f5e) at ($(mm)+(1.2,-1.2)$);
\draw [strand] plot[smooth] coordinates {(s4) (m4) (m5) (mm) (f5e)}
[ postaction=decorate,decoration={markings,
mark= at position 0.15 with \strarrow,
mark= at position 0.36 with \strarrow,
mark= at position 0.65 with \strarrow,
mark= at position 0.93 with \strarrow}];

\draw ($(m6)+(0,-0.6)$) node {$X_5$};
\draw ($(m1)+(0.6,-0.2)$) node {$X'_5$};
\draw ($(m5)+(-0.5,-0.35)$) node {$Y'_5$};
\draw ($(mm)+(-0.1,-0.5)$) node {$Y_5$};
\draw ($(b5)+(1.2,-0.4)$) node {$\beta_5$};
\draw ($(m1)+(0.6,-0.2)$) node {$X'_5$};
\draw ($(mm)+(0,\arlength*0.35)$) node {$F_5$};

\foreach \i/\j in {2/1,4/3,6/5}
{\draw [strand] plot[smooth] coordinates {(t\i) (m\i) (m\j) (s\j)}
[ postaction=decorate, decoration={markings,
 mark= at position 0.2 with \strarrow, mark= at position 0.5 with
\strarrow,
 mark= at position 0.8 with \strarrow}];}

\foreach \n in {1,2,3,4,5,6}
{\draw (\bstart+\sixth*\n+\arrowlabeldisp:\littlerad) node
{$\alpha_{\n}$}; }

\foreach \n in {1,3,5}
{\draw (\asstart+\sixth*\n:\asrad) node {$f_{\n}$}; }

\foreach \n in {2,4,6}
{\draw (\csstart+\sixth*\n:\asrad) node {$f_{\n}$}; }

\end{tikzpicture}
\]
\caption{The arrows incident with an internal vertex $I$}
\label{f:internalincidence}
\end{figure}
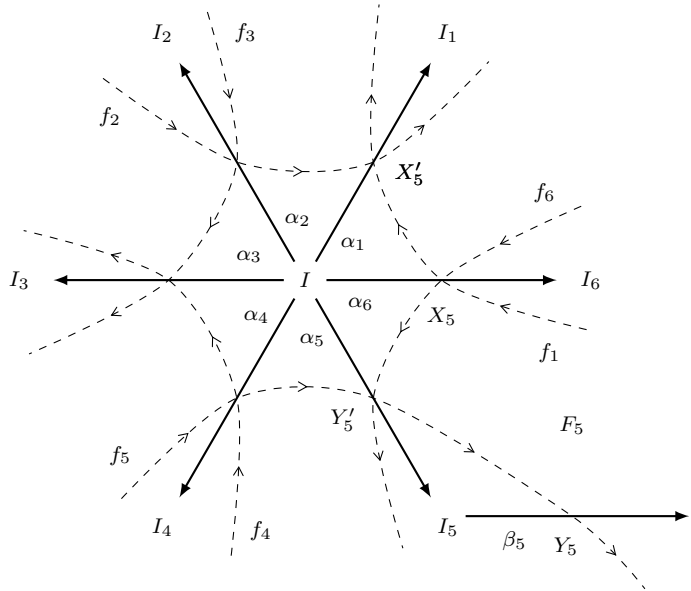

\begin{lemma} \label{l:internalordering}
Let $I$ be an internal vertex of $Q(D)$, with notation as defined above.
Then
\begin{itemize}
\item[(a)] For each $i$, the strand labels $f_i,f_{i+1},f_{i+2}$ occur in clockwise order
in $C$.
\item[(b)]
The strand labels $f_1,f_3,\ldots ,f_{2r-1}$ appear in anticlockwise order in $C$, and
\item[(c)]
The strand labels $f_2,f_4,\ldots ,f_{2r}$ appear in anticlockwise order in $C$.
\end{itemize}
\end{lemma}

\begin{proof}
We argue as in the proof of necessity in~\cite[Thm. 6.6]{bocklandt12}.
Fix $i\in [1,2r]$, odd, and let $F_i$ be the face whose boundary includes
$\alpha_{i}\alpha_{i+1}$. Let $\beta_i$ be the arrow following $\alpha_i$ in $\bdry F_i$.
The strands leaving $F_i$ at $\alpha_{i}$, $\beta_i$, and $\alpha_{i+1}$ (ordered
clockwise around $\bdry F_i$) are $f_{i},f_{i+1}$ and
$f_{i+2}$ respectively. Part (a) for $i$ odd then follows from an application of Lemma~\ref{l:faceordering}.

Let $X_i$ be the point on $\alpha_{i+1}$ where strand
$f_{i+2}$ leaves $F_i$ and let $Y_i$ be the point on
$\beta_i$ where strand $f_i$ leaves $F_i$ (note that $Y_i$ could coincide with $X_i$). By Lemma~\ref{l:faceordering},
strands $f_{i+2}$ and $f_i$ do not cross after they leave $F_i$.

Let $X'_i$ be the next crossing point of strand $f_{i+2}$
with an arrow after $X_i$ (i.e.\ on arrow $\alpha_{i+2}$), and let $Y'_i$ be the previous crossing point of strand
$f_i$ with an arrow before $Y_i$ (i.e.\ on arrow $\alpha_i$).
Since the part of strand $f_i$ between $Y'_i$ and $Y_i$ lies inside $F_i$
and the part of strand $f_{i+2}$ between $X_i$ and $X'_i$ lies outside $F_i$,
it follows that strand $f_{i+2}$, after $X'_i$, and strand $f_i$, after $Y'_i$,
do not cross.
See Figure~\ref{f:internalincidence} for an illustration in the case $i=5$.
Part (b) then follows, via an argument similar to that used in Lemma~\ref{l:faceordering}.

The proofs of part (a) for $i$ even and part (c)
are similar to the above.
\end{proof}

Properties (b) and (c) of Lemma~\ref{l:internalordering}
are similar to the notion of `proper ordering' introduced 
by Gullota~\cite[\S3]{gulotta08} for dimer models on a torus $T^2$. 
In this case, each strand (or zig-zag path) 
has a two component winding number, i.e.\ 
its homology class in $H_1(T^2)\cong \mathbb{Z}^2$.
The dimer model is said to be properly
ordered if the circular order of the strands around any vertex
of the bipartite graph is the same as the circular order of 
the directions of their winding numbers.

\begin{definition} \label{d:xiI}
Let $I$ be an internal vertex in $Q_0(D)$ and let the $\newt_i$ be defined as above.
Then consecutive intervals
$(\newt_i,\newt_{i+1})_0$ of vertices follow on from each other,
and do not overlap. Thus gluing these intervals together creates a path on $C$ which we denote $\xi(I)$.
\end{definition}

We have:

\begin{proposition} \label{p:internalformula}
Fix an internal vertex $I$ of $Q(D)$ incident with $2r$
arrows. Then the path $\xi(I)$ wraps around $C$ exactly
$r-1$ times.
\end{proposition}

\begin{proof}
By Lemma~\ref{l:internalordering}(b), we have:
\begin{equation}
\sum_{i=1}^r (f_{2i+1},f_{2i-1})_0=C_0,
\label{e:around}
\end{equation}
By Lemma~\ref{l:internalordering}(a), we have
$$(f_{2i-1},f_{2i+1})_0=(f_{2i-1},f_{2i})_0 +(f_{2i},f_{2i+1}),$$
So
$$(f_{2i+1},f_{2i-1})_0=C_0-(f_{2i-1},f_{2i})_0-(f_{2i},f_{2i+1})_0.$$
Hence, by~\eqref{e:around},
$$\sum_{i=1}^r \left( C_0-(f_{2i-1},f_{2i})_0-(f_{2i},f_{2i+1})_0 \right)=C_0.$$
So
\begin{equation}
\label{e:sumoflabels}
\sum_{i=1}^r \left( (f_{2i-1},f_{2i})_0+(f_{2i},f_{2i+1})\right)=(r-1)C_0.
\end{equation}
As mentioned above, consecutive intervals $(f_i,f_{i+1})_0$ of
vertices follow on from each other and the result follows.
\end{proof}

We next consider the case of a boundary vertex $\newL_j$ in $Q(D)$.
If there is an arrow from $\newL_j$ to $\newL_{j+1}$ (respectively, $\newL_{j-1}$)
we take $\alpha_{\out}^+$ (respectively, $\alpha_{\out}^-$) to be this arrow;
otherwise we take it to be the arrow outgoing from $\newL_j$ immediately clockwise
(respectively, anticlockwise) of this arrow. Let $W_{\out}(\newL_j)$ denote the set
of all arrows (whether outgoing or ingoing) incident with $\newL_j$ in the wedge
clockwise of $\alpha_{\out}^+$ and anticlockwise of $\alpha_{\out}^-$ (including both $\alpha_{\out}^+$ and $\alpha_{\out}^-$).
Let $r_{\out}$ be the number of arrows in $W_{\out}(\newL_j)$ starting at $\newL_j$,
so we have that $|W_{\out}|=2r_{\out}-1$. See Figure~\ref{f:externalincidence} for an
illustration.

If there is an arrow from $\newL_{j+1}$ (respectively, $\newL_{j-1}$) to
$\newL_{j}$ we take $\alpha_{\inn}^+$ (respectively,
$\alpha_{\inn}^-$) to be the this arrow; otherwise we take it to be the arrow incoming
to $\newL_j$ immediately clockwise (respectively, anticlockwise) of this arrow.
Let $W_{\inn}(\newL_j)$ denote the set of all arrows (whether outgoing or ingoing)
incident with $\newL_j$ in the wedge clockwise of $\alpha_{\inn}^+$ and anticlockwise of $\alpha_{\inn}^-$ (including both $\alpha_{\inn}^+$ and $\alpha_{\inn}^-$).
Let $r_{\inn}$ be the number of arrows in $W_{\inn}(\newL_j)$ ending at $\newL_j$,
so we have that $|W_{\inn}|=2r_{\inn}-1$.

We write
$$W_{\out}(\newL_j)=\{\alpha_1,\alpha_2,\ldots ,\alpha_{2r_{\out}-1}\},$$
numbering the arrows anticlockwise around $\newL_j$.
Note that the $\alpha_i$ for $i$ odd are outgoing from
$\newL_j$ and the others are incoming to $\newL_j$.

There are two possibilities for arrows incident with $\newL_j$ not lying in $W_{\out}(\newL_j)$.
If the arrow between $\newL_j$ and $\newL_{j-1}$ (respectively, $\newL_{j+1}$)
points towards $\newL_j$ (and so does not lie in $W_{\out}(\newL_j)$), label it $\alpha_0$
(respectively, $\alpha_{2r}$), where $r=r_{\out}$. Thus the arrows incident with $I$ are
$\alpha_0,\alpha_1, \ldots ,\alpha_{2r-1},\alpha_{2r}$, in anticlockwise order
around $i$, with the first and last arrows in the list appearing only when they are defined.

As before, we set $I_i$ be the end-point of $\alpha_i$ not equal to $I$.
And for $i$ odd, we suppose that strand $\newt_i$ crosses $\alpha_i$ from right to
left (looking along the arrow), and $\newt_{i+1}$ crosses $\alpha_i$ from left
to right, so that $\alpha_i$ has weight $(\newt_i,\newt_{i+1})_0$. Note that if $i$ is
even, then $\alpha_i$ also has weight $(\newt_i,\newt_{i+1})_0$.
Note that statement~\eqref{e:Iidescription} holds in this case also. See Figure~\ref{f:externalincidence} for an illustration of the possible cases.

Since $E_j=E_{j-1}\cup \{j\}\setminus \{j-k\}$ and strand $f_1$ crosses
the arrow between $E_{j-1}$ and $E_j$ (whichever direction it is in), with
$E_j$ on its left, we always have that $f_1=j$. Similarly, we always have that
$f_{2r}=j+1$.

We make similar definitions for the wedge $W_{\inn}(\newL_j)$.

\begin{figure}
\[
\begin{tikzpicture}[scale=0.85,baseline=(bb.base),
 strand/.style={black, dashed},
 quivarrow/.style={black, -latex, thick}]

\newcommand{\strarrow}{\arrow{angle 60}}
\newcommand{\bstart}{90} 
\newcommand{\twelfth}{30} 
\newcommand{\strandth}{10} 
\newcommand{\strandrad}{5cm} 
\newcommand{\strandlabelrad}{5.4cm} 
\newcommand{\bigrad}{4.5cm} 
\newcommand{\vertexlabelrad}{4.75cm} 
\newcommand{\dotrad}{0.1cm} 
\newcommand{\bdrydotrad}{{0.8*\dotrad}} 
\newcommand{\arrowlabelrad}{0.3*\bigrad} 
\newcommand{\alstart}{79} 
\newcommand{\shaderad}{4.275cm} 

\path (0,0) node (bb) {};

\foreach \n in {1,...,12}
{ \coordinate (b\n) at (\bstart+\twelfth*\n:\bigrad);
  \draw (\bstart+\twelfth*\n:\bigrad) node {};
  }


\coordinate (bb1) at (\bstart+\twelfth:\shaderad);

\foreach \n in {1,...,12}
{ \coordinate (v\n) at (\bstart+\twelfth*\n:\vertexlabelrad);}

\foreach \n in {1,...,36}
{ \coordinate (s\n) at (\bstart+\strandth*\n:\strandrad);}

\foreach \n in {1,...,36}
{ \coordinate (sl\n) at (\bstart+\strandth*\n:\strandlabelrad);}

\foreach \n in {1,...,12}
{ \coordinate (al\n) at (\alstart+\twelfth*\n:\arrowlabelrad);}

\foreach \n in {1,3,5} {\draw [quivarrow,shorten >=5pt,shorten <=12pt] (0,0) -- (b\n); }

\foreach \n in {2,4} {\draw [quivarrow,shorten >=12pt,shorten <=5pt] (b\n) -- (0,0); }

%
\foreach \i/\j in {1/0.5, 2/0.5, 3/0.5, 4/0.5, 5/0.5, 6/0.5, 7/0.5, 8/0.5, 9/0.5, 10/0.5, 11/0.5, 12/0.5}
{ \coordinate (m\i) at (${\j}*(b\i)$); }


\draw [strand] plot coordinates
{(m1) (s4)}[postaction=decorate, decoration={markings, mark= at position 0.5 with
\strarrow}];

\draw [strand] plot[smooth] coordinates
{(s5) (m2) (m3) (s10)}[postaction=decorate, decoration={markings,
mark= at position 0.22 with \strarrow,
mark= at position 0.5 with \strarrow,
mark= at position 0.78 with \strarrow
}];

\draw [strand] plot[smooth] coordinates
{(s7) (m2) (m1)}[postaction=decorate, decoration={markings,
mark= at position 0.4 with \strarrow,
mark= at position 0.88 with \strarrow
}];

\draw [strand] plot[smooth] coordinates
{(s13) (m4) (m3) (s8)}[postaction=decorate, decoration={markings,
mark= at position 0.22 with \strarrow,
mark= at position 0.5 with \strarrow,
mark= at position 0.78 with \strarrow
}];

\draw [strand] plot[smooth] coordinates
{(s11) (m4) (m5)}[postaction=decorate, decoration={markings,
mark= at position 0.4 with \strarrow,
mark= at position 0.88 with \strarrow
}];

\draw [strand] plot coordinates
{(m5) (s14)}[postaction=decorate, decoration={markings, mark= at position 0.5 with
\strarrow}];

\draw (sl5) node {\small $f_3$};
\draw (sl7) node[left=0pt,below=5pt] {\small $f_2=j-k$};
\draw (sl11) node[left=0pt,below=-5pt] {\small $f_5=j-k+1$};
\draw (sl13) node {\small $f_4$};
\draw ($(m1)+(0.73,0)$) node {\small $f_1=j$};
\draw ($(m5)+(0.95,-0.2)$) node {\small $f_6=j+1$};

\foreach \n in {1,2,3,4,5}
{\draw (\alstart+\twelfth*\n:
\arrowlabelrad) node
{$\alpha_{\n}$}; }

\draw (0,0) node {$E_j$};

\draw (v1) node {$E_{j-1}$};
\draw (v5) node {$E_{j+1}$};


\draw (m1) circle(\bdrydotrad) [fill=gray];
\draw (m5) circle(\bdrydotrad) [fill=gray];


\draw[fill=gray, opacity=0.1] (0,0)--(bb1)
arc (120:240:\shaderad) -- cycle;

\end{tikzpicture}
\hspace*{-3.2cm}
\begin{tikzpicture}[scale=0.85,baseline=(bb.base),
 strand/.style={black, dashed},
 quivarrow/.style={black, -latex, thick}]

\newcommand{\strarrow}{\arrow{angle 60}}
\newcommand{\bstart}{90} 
\newcommand{\twelfth}{30} 
\newcommand{\strandth}{10} 
\newcommand{\strandrad}{5cm} 
\newcommand{\strandlabelrad}{5.4cm} 
\newcommand{\bigrad}{4.5cm} 
\newcommand{\vertexlabelrad}{4.75cm} 
\newcommand{\dotrad}{0.1cm} 
\newcommand{\bdrydotrad}{{0.8*\dotrad}} 
\newcommand{\arrowlabelrad}{0.3*\bigrad} 
\newcommand{\alstart}{79} 
\newcommand{\shaderad}{4.275cm} 

\path (0,0) node (bb) {};

\foreach \n in {1,...,12}
{ \coordinate (b\n) at (\bstart+\twelfth*\n:\bigrad);
  }


\coordinate (bb1) at (\bstart+\twelfth:\shaderad);

\foreach \n in {1,...,12}
{ \coordinate (v\n) at (\bstart+\twelfth*\n:\vertexlabelrad);}

\foreach \n in {1,...,36}
{ \coordinate (s\n) at (\bstart+\strandth*\n:\strandrad);
}

\foreach \n in {1,...,36}
{ \coordinate (sl\n) at (\bstart+\strandth*\n:\strandlabelrad);}

\foreach \n in {1,...,12}
{ \coordinate (al\n) at (\alstart+\twelfth*\n:\arrowlabelrad);}

\foreach \n in {1,3,5} {\draw [quivarrow,shorten >=5pt,shorten <=12pt] (0,0) -- (b\n); }

\foreach \n in {2,4,12} {\draw [quivarrow,shorten >=12pt,shorten <=5pt] (b\n) -- (0,0); }

%
\foreach \i/\j in {1/0.5, 2/0.5, 3/0.5, 4/0.5, 5/0.5, 6/0.5, 7/0.5, 8/0.5, 9/0.5, 10/0.5, 11/0.5, 12/0.5}
{ \coordinate (m\i) at (${\j}*(b\i)$); }


\draw [strand] plot coordinates
{(s1) (m12)}[postaction=decorate, decoration={markings,
mark= at position 0.5 with \strarrow}];

\draw [strand] plot[smooth] coordinates
{(m12) (m1) (s4)}[postaction=decorate, decoration={markings,
mark= at position 0.12 with \strarrow,
mark= at position 0.6 with \strarrow}];

\draw [strand] plot[smooth] coordinates
{(s5) (m2) (m3) (s10)}[postaction=decorate, decoration={markings,
mark= at position 0.22 with \strarrow,
mark= at position 0.5 with \strarrow,
mark= at position 0.78 with \strarrow
}];

\draw [strand] plot[smooth] coordinates
{(s7) (m2) (m1) (s2)}[postaction=decorate, decoration={markings,
mark= at position 0.22 with \strarrow,
mark= at position 0.5 with \strarrow,
mark= at position 0.78 with \strarrow
}];

\draw [strand] plot[smooth] coordinates
{(s13) (m4) (m3) (s8)}[postaction=decorate, decoration={markings,
mark= at position 0.22 with \strarrow,
mark= at position 0.5 with \strarrow,
mark= at position 0.78 with \strarrow
}];

\draw [strand] plot[smooth] coordinates
{(s11) (m4) (m5)}[postaction=decorate, decoration={markings,
mark= at position 0.4 with \strarrow,
mark= at position 0.88 with \strarrow
}];

\draw [strand] plot coordinates
{(m5) (s14)}[postaction=decorate, decoration={markings, mark= at position 0.5 with
\strarrow}];

\draw (sl5) node {\small $f_3$};
\draw (sl7) node {\small $f_2$};
\draw (sl11) node[left=0pt,below=-5pt] {\small $f_5=j-k+1$};
\draw (sl13) node {\small $f_4$};
\draw ($(m12)+(0.73,0)$) node {\small $f_1=j$};
\draw ($(m5)+(0.95,-0.2)$) node {\small $f_6=j+1$};

\foreach \n in {0,1,2,3,4,5}
{\draw (\alstart+\twelfth*\n:
\arrowlabelrad) node
{$\alpha_{\n}$}; }

\draw (0,0) node {$E_j$};

\draw (v12) node {$E_{j-1}$};
\draw (v5) node {$E_{j+1}$};


\draw (m5) circle(\bdrydotrad) [fill=gray];
\draw (m12) circle(\bdrydotrad) [fill=gray];


\draw[fill=gray, opacity=0.1] (0,0)--(bb1)
arc (120:240:\shaderad) -- cycle;

\end{tikzpicture}
\]
\[
\begin{tikzpicture}[scale=0.85,baseline=(bb.base),
 strand/.style={black, dashed},
 quivarrow/.style={black, -latex, thick}]

\newcommand{\strarrow}{\arrow{angle 60}}
\newcommand{\bstart}{90} 
\newcommand{\twelfth}{30} 
\newcommand{\strandth}{10} 
\newcommand{\strandrad}{5cm} 
\newcommand{\strandlabelrad}{5.4cm} 
\newcommand{\bigrad}{4.5cm} 
\newcommand{\vertexlabelrad}{4.75cm} 
\newcommand{\dotrad}{0.1cm} 
\newcommand{\bdrydotrad}{{0.8*\dotrad}} 
\newcommand{\arrowlabelrad}{0.3*\bigrad} 
\newcommand{\alstart}{79} 
\newcommand{\shaderad}{4.275cm} 

\path (0,0) node (bb) {};

\foreach \n in {1,...,12}
{ \coordinate (b\n) at (\bstart+\twelfth*\n:\bigrad);
  }


\coordinate (bb1) at (\bstart+\twelfth:\shaderad);

\foreach \n in {1,...,12}
{ \coordinate (v\n) at (\bstart+\twelfth*\n:\vertexlabelrad);
}

\foreach \n in {1,...,36}
{ \coordinate (s\n) at (\bstart+\strandth*\n:\strandrad);
}

\foreach \n in {1,...,36}
{ \coordinate (sl\n) at (\bstart+\strandth*\n:\strandlabelrad);}

\foreach \n in {1,...,12}
{ \coordinate (al\n) at (\alstart+\twelfth*\n:\arrowlabelrad);}

\foreach \n in {1,3,5} {\draw [quivarrow,shorten >=5pt,shorten <=12pt] (0,0) -- (b\n); }

\foreach \n in {2,4,6} {\draw [quivarrow,shorten >=12pt,shorten <=5pt] (b\n) -- (0,0); }

%
\foreach \i/\j in {1/0.5, 2/0.5, 3/0.5, 4/0.5, 5/0.5, 6/0.5, 7/0.5, 8/0.5, 9/0.5, 10/0.5, 11/0.5, 12/0.5}
{ \coordinate (m\i) at (${\j}*(b\i)$); }


\draw [strand] plot coordinates
{(m1) (s4)}[postaction=decorate, decoration={markings, mark= at position 0.5 with
\strarrow}];

\draw [strand] plot[smooth] coordinates
{(s5) (m2) (m3) (s10)}[postaction=decorate, decoration={markings,
mark= at position 0.22 with \strarrow,
mark= at position 0.5 with \strarrow,
mark= at position 0.78 with \strarrow
}];

\draw [strand] plot[smooth] coordinates
{(s7) (m2) (m1)}[postaction=decorate, decoration={markings,
mark= at position 0.4 with \strarrow,
mark= at position 0.88 with \strarrow
}];

\draw [strand] plot[smooth] coordinates
{(s13) (m4) (m3) (s8)}[postaction=decorate, decoration={markings,
mark= at position 0.22 with \strarrow,
mark= at position 0.5 with \strarrow,
mark= at position 0.78 with \strarrow
}];

\draw [strand] plot[smooth] coordinates
{(s11) (m4) (m5) (s16)}[postaction=decorate, decoration={markings,
mark= at position 0.22 with \strarrow,
mark= at position 0.5 with \strarrow,
mark= at position 0.78 with \strarrow
}];

\draw [strand] plot[smooth] coordinates
{(m6) (m5) (s14)}[postaction=decorate, decoration={markings, mark= at position 0.12 with
\strarrow,
mark= at position 0.6 with \strarrow
}];

\draw [strand] plot coordinates
{(s17) (m6)}[postaction=decorate, decoration={markings, mark= at position 0.5 with
\strarrow}];

\draw (sl5) node {\small $f_3$};
\draw (sl7) node[left=0pt,below=5pt] {\small $f_2=j-k$};
\draw (sl11) node[left=0pt,below=-5pt] {\small $f_5$};
\draw (sl13) node {\small $f_4$};
\draw (sl17) node[left=-15pt,below=-10pt] {\small $f_7=j-k+1$};
\draw ($(m1)+(0.73,0)$) node {\small $f_1=j$};
\draw ($(m6)+(1,-0.2)$) node {\small $f_6=j+1$};

\foreach \n in {1,2,3,4,5,6}
{\draw (\alstart+\twelfth*\n:
\arrowlabelrad) node
{$\alpha_{\n}$}; }

\draw (0,0) node {$E_j$};

\draw (v1) node {$E_{j-1}$};
\draw (v6) node {$E_{j+1}$};


\draw (m1) circle(\bdrydotrad) [fill=gray];
\draw (m6) circle(\bdrydotrad) [fill=gray];

\draw[fill=gray, opacity=0.1] (0,0)--(bb1)
arc (120:240:\shaderad) -- cycle;

\end{tikzpicture}
\begin{tikzpicture}[scale=0.85,baseline=(bb.base),
 strand/.style={black, dashed},
 quivarrow/.style={black, -latex, thick}]

\newcommand{\strarrow}{\arrow{angle 60}}
\newcommand{\bstart}{90} 
\newcommand{\twelfth}{30} 
\newcommand{\strandth}{10} 
\newcommand{\strandrad}{5cm} 
\newcommand{\strandlabelrad}{5.4cm} 
\newcommand{\bigrad}{4.5cm} 
\newcommand{\vertexlabelrad}{4.75cm} 
\newcommand{\dotrad}{0.1cm} 
\newcommand{\bdrydotrad}{{0.8*\dotrad}} 
\newcommand{\arrowlabelrad}{0.3*\bigrad} 
\newcommand{\alstart}{79} 
\newcommand{\shaderad}{4.275cm} 

\path (0,0) node (bb) {};

\foreach \n in {1,...,12}
{ \coordinate (b\n) at (\bstart+\twelfth*\n:\bigrad);
  }


\coordinate (bb1) at (\bstart+\twelfth:\shaderad);

\foreach \n in {1,...,12}
{ \coordinate (v\n) at (\bstart+\twelfth*\n:\vertexlabelrad);}

\foreach \n in {1,...,36}
{ \coordinate (s\n) at (\bstart+\strandth*\n:\strandrad);
}

\foreach \n in {1,...,36}
{ \coordinate (sl\n) at (\bstart+\strandth*\n:\strandlabelrad);}

\foreach \n in {1,...,12}
{ \coordinate (al\n) at (\alstart+\twelfth*\n:\arrowlabelrad);}

\foreach \n in {1,3,5} {\draw [quivarrow,shorten >=5pt,shorten <=12pt] (0,0) -- (b\n); }

\foreach \n in {2,4,6,12} {\draw [quivarrow,shorten >=12pt,shorten <=5pt] (b\n) -- (0,0); }

%
\foreach \i/\j in {1/0.5, 2/0.5, 3/0.5, 4/0.5, 5/0.5, 6/0.5, 7/0.5, 8/0.5, 9/0.5, 10/0.5, 11/0.5, 12/0.5}
{ \coordinate (m\i) at (${\j}*(b\i)$); }


\draw [strand] plot coordinates
{(s1) (m12)}[postaction=decorate, decoration={markings,
mark= at position 0.5 with \strarrow}];

\draw [strand] plot[smooth] coordinates
{(m12) (m1) (s4)}[postaction=decorate, decoration={markings,
mark= at position 0.12 with \strarrow,
mark= at position 0.6 with \strarrow}];

\draw [strand] plot[smooth] coordinates
{(s5) (m2) (m3) (s10)}[postaction=decorate, decoration={markings,
mark= at position 0.22 with \strarrow,
mark= at position 0.5 with \strarrow,
mark= at position 0.78 with \strarrow
}];

\draw [strand] plot[smooth] coordinates
{(s7) (m2) (m1) (s2)}[postaction=decorate, decoration={markings,
mark= at position 0.22 with \strarrow,
mark= at position 0.5 with \strarrow,
mark= at position 0.78 with \strarrow
}];

\draw [strand] plot[smooth] coordinates
{(s13) (m4) (m3) (s8)}[postaction=decorate, decoration={markings,
mark= at position 0.22 with \strarrow,
mark= at position 0.5 with \strarrow,
mark= at position 0.78 with \strarrow
}];

\draw [strand] plot[smooth] coordinates
{(s11) (m4) (m5) (s16)}[postaction=decorate, decoration={markings,
mark= at position 0.22 with \strarrow,
mark= at position 0.5 with \strarrow,
mark= at position 0.78 with \strarrow
}];

\draw [strand] plot[smooth] coordinates
{(m6) (m5) (s14)}[postaction=decorate, decoration={markings, mark= at position 0.12 with
\strarrow,
mark= at position 0.6 with \strarrow
}];

\draw [strand] plot coordinates
{(s17) (m6)}[postaction=decorate, decoration={markings, mark= at position 0.5 with
\strarrow}];

\draw (sl5) node {\small $f_3$};
\draw (sl7) node[left=0pt,below=5pt] {\small $f_2=j-k$};
\draw (sl11) node[left=0pt,below=-5pt] {\small $f_5=j-k+1$};
\draw (sl13) node {\small $f_4$};
\draw ($(m12)+(0.73,0)$) node {\small $f_1=j$};
\draw ($(m6)+(1,-0.2)$) node {\small $f_6=j+1$};

\foreach \n in {0,1,2,3,4,5,6}
{\draw (\alstart+\twelfth*\n:
\arrowlabelrad) node
{$\alpha_{\n}$}; }

\draw (0,0) node {$E_j$};

\draw (v12) node {$E_{j-1}$};
\draw (v6) node {$E_{j+1}$};


\draw (m6) circle(\bdrydotrad) [fill=gray];
\draw (m12) circle(\bdrydotrad) [fill=gray];


\draw[fill=gray, opacity=0.1] (0,0)--(bb1)
arc (120:240:\shaderad) -- cycle;

\end{tikzpicture}
\]
\caption{The arrows incident with a boundary vertex $\newL_j$. The shaded
region indicates the wedge $W_{\out}$}
\label{f:externalincidence}
\end{figure}
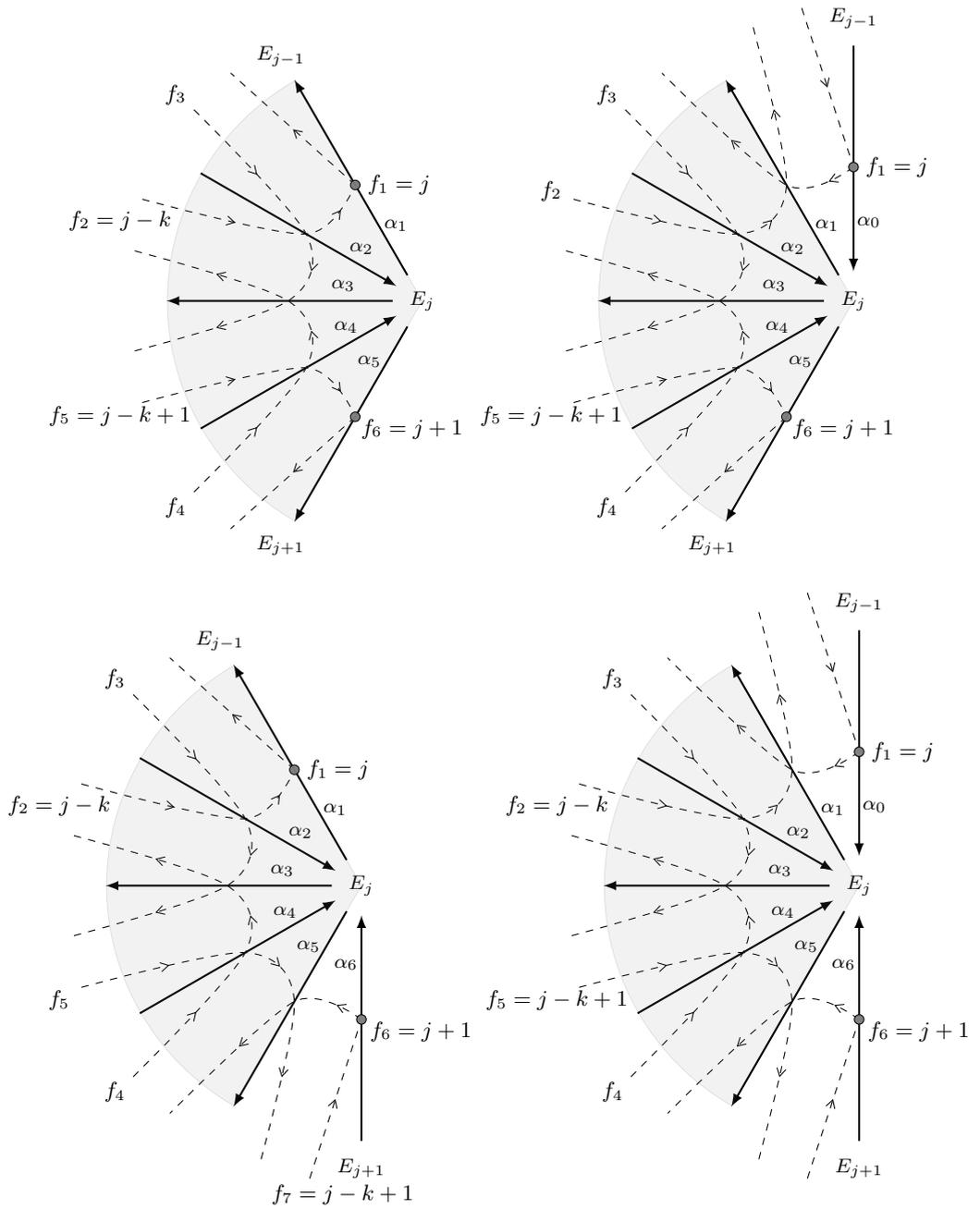

\begin{definition} \label{d:xiIboundary}Let $\newL_j$ be an external vertex in $Q_0(D)$ and let the $\newt_i$ be defined as above.
Then consider the intervals $(\newt_i,\newt_{i+1})_0$ of vertices, for $i=1,2,\ldots ,2r_{\out}-1$.
Consecutive intervals in this set follow on from
each other exactly and do not overlap. We denote the path
obtained by gluing these intervals together by $\xi_{\out}
(\newL_j)$. Similarly, we have a path $\xi_{\inn}(\newL_j)$.
\end{definition}

\begin{proposition} \label{p:externalformula}
Let $\newL_j$ be a boundary vertex of $Q$.
Then we have the following.
\begin{itemize}
\item[(a)] The path $\xi_{\out}(E_j)$ starts at vertex $j$,
wraps $r_{\out}-1$ times around $C_0$ times, then ends by revisiting the vertex $j$.
\item[(b)] The path $\xi_{\inn}(E_j)$ starts at vertex $j-k$,
wraps $r_{\inn}-1$ times around $C_0$ times, then ends by revisiting the vertex $j-k$.
\end{itemize}
\end{proposition}

\begin{proof}
We prove part (a) only; the proof of (b) is similar. We set $r=r_{\out}$.
The result will follow if we can show that:
\begin{equation}
\label{e:desiredresult}
\sum_{\alpha\in W_{\out}(\newL_j)} w_{\alpha}=(r-1)C_0+j,
\end{equation}
since the weights of the $\alpha_i$ are intervals forming the path $\xi_{\out}(\newL_j)$.
Arguing as in Lemma~\ref{l:internalordering}, we see that:
\begin{itemize}
\item[(a)] $f_1,f_3,\ldots ,f_{2r-1}$ appear in anticlockwise order in $C_0$, and
\item[(b)] for $1\leq i\leq 2r-3$, $f_i,f_{i+1},f_{i+2}$ occur in clockwise order in $C_0$.
\end{itemize}
Hence
$$\sum_{i=1}^{r-1} (f_{2i+1},f_{2i-1})_0=C_0-(f_1,f_{2r-1})_0,$$
and, for $1\leq i\leq r-1$,
$$(f_{2i-1},f_{2i+1})_0=(f_{2i-1},f_{2i})_0+(f_{2i},f_{2i+1})_0=w_{\alpha_{2i-1}}+w_{\alpha_{2i}}.$$
So, for $1\leq i\leq r-1$,
$$(f_{2i+1},f_{2i-1})_0=C_0-w_{\alpha_{2i-1}}-w_{\alpha_{2i}},$$
and we have:
\begin{equation}
\label{e:sumequation}
\sum_{i=1}^{r-1} \left( C_0-w_{\alpha_{2i-1}}-w_{\alpha_{2i}} \right)=C_0-(f_1,f_{2r-1})_0.
\end{equation}
If the arrow between $E_j$ and $E_{j+1}$ points towards $E_{j+1}$
then it is labelled $\alpha_{2r-1}$.
Since $E_j=E_{j+1}\cup \{j-k+1\}\setminus \{j+1\}$ and strand
$f_{2r-1}$ cuts this arrow with $E_j$ on its left, we have that
$f_{2r-1}=j-k+1$. The other strand crossing this arrow is $f_{2r}=j+1$.
Recall also that $f_1=j$.
Hence, in this case,~\eqref{e:sumequation} can be rewritten as
$$\sum_{i=1}^{2r-2} w_{\alpha_i}=(r-2)C_0+(j,j-k+1)_0$$
We then have, using the fact that $w_{\alpha_{2r-1}}=(f_{2r-1},f_{2r})_0=(j-k+1,j+1)_0$, that
$$\sum_{i=1}^{2r-1} w_{\alpha_i}=(r-2)C_0+(j,j-k+1)_0+(j-k+1,j+1)_0,$$
implying~\eqref{e:desiredresult} as required.

If the arrow between $E_j$ and $E_{j+1}$ points towards $E_j$ then it is labelled
$\alpha_{2r}$. Then strand $f_{2r}=j+1$ crosses this arrow from right to left
and $f_{2r+1}=j-k+1$ crosses it from left to right.
Arguing as above, we obtain the following variation of~\eqref{e:sumequation}:
\begin{equation*}
\sum_{i=1}^{r} \left( C_0-w_{\alpha_{2i-1}}-w_{\alpha_{2i}} \right)=C_0-(f_1,f_{2r+1})_0,
\end{equation*}
which can be rewritten as:
$$\sum_{i=1}^{2r} w_{\alpha_i}=(r-1)C_0+(j,j-k+1)_0.$$
We then have, using the fact that $w_{\alpha_{2r}}=(j+1,j-k+1)_0$, that
$$\sum_{i=1}^{2r-1} w_{\alpha_i}=(r-1)C_0+(j,j-k+1)_0-(j+1,j-k+1)_0,$$
implying~\eqref{e:desiredresult} as required.
\end{proof}

\section{Angles and isoradial embedding} 
\label{s:angles}

Fix a Postnikov diagram $D$.
In this section we investigate the implications of the
results in the previous section for embedding $Q(D)$
isoradially into a planar polygon.
This leads us to a notion of consistent boundary $R$-charge.
We compare this embedding with results of~\cite{ops}. The results in this section are not needed for the main result of the paper, but seem to be interesting from the point of view of understanding $Q(D)$ as a dimer model.

For each $i\in C_0$, let $\theta_i\in (0,2\pi)$ be an angle, with the property
that $\sum_{i\in C_0} \theta_i = 2\pi$.
For each arrow $\alpha$ in $Q_1(D)$, let
$$\theta_{\alpha}=\sum_{i\in w_{\alpha}} \theta_i.$$
We then have the following result. It has a
geometric interpretation which will be explained in
Corollary~\ref{c:embedding}.

\begin{lemma} \label{l:angles}
The $\theta_{\alpha}$ satisfy the following conditions:
\begin{enumerate}[(a)]
\item For all $F\in Q_2(D)$,
$$\sum_{\alpha\in \bdry F} \theta_{\alpha}=2\pi.$$
\item Let $I\in Q_0(D)$ be an internal vertex. Then we have:
\begin{equation} \label{e:internalangle1}
\sum_{\alpha,\beta} \pi-\tfrac{1}{2}(\theta_{\alpha}+\theta_{\beta})=2\pi,
\end{equation}
where the sum is over all pairs of arrows $\alpha,\beta$ incident with $I$ such that
$\alpha,\beta$ are consecutive arrows in a face of $Q(D)$.
Equivalently,
\begin{equation} \label{e:internalangle2}
\sum_{\alpha}(\pi-\theta_{\alpha})=2\pi,
\end{equation}
where the sum is over all arrows $\alpha$ incident with
$I$.
\item Let $I=\newL_j\in Q_0(D)$ be an external vertex. Then we have:
\begin{equation} \label{e:boundaryangle1}
\sum_{\alpha,\beta} \pi-\tfrac{1}{2}(\theta_{\alpha}+\theta_{\beta})=
\pi-\tfrac{1}{2}(\theta_j+\theta_{j-k}).
\end{equation}
where the sum is over all pairs of arrows $\alpha,\beta$ incident with $I$ such that
$\alpha,\beta$ are consecutive arrows in a face of $Q(D)$.
Equivalently,
\begin{equation} \label{e:boundaryangle2}
\sum_{\alpha\in W_{\out}(\newL_j)} \pi-\theta_{\alpha}=\pi-\theta_j.
\end{equation}

\end{enumerate}
\end{lemma}
\begin{proof}
Part (a) is equivalent to Corollary~\ref{c:cycleweight}.
In parts (b) and (c), we use the same notation as above for the arrows incident with $I$ in $Q(D)$.
For part (b), note that $\sum_{i=1}^{2r} \theta_{\alpha_i}=2(r-1)\pi$, by equation~\eqref{e:sumoflabels} in the proof of Proposition~\ref{p:externalformula}.
This is easily seen to be equivalent to the two formulas in (b).

We now consider part (c).
Let $\gamma^-$ be the arrow between $\newL_j$ and $\newL_{j-1}$, and let
$\gamma^+$ be the arrow between $\newL_j$ and $\newL_{j+1}$.
Then we have
\begin{equation}
\label{e:gamma}
\begin{aligned}
w_{\gamma^-}&=\begin{cases}
[j,j-k-1] & \text{if $\gamma^-$ points away from $\newL_j$}; \\
C_0-[j,j-k-1] & \text{if $\gamma^-$ points towards $\newL_j$};
\end{cases} \\
w_{\gamma^+}&=
\begin{cases}
[j-k+1,j] & \text{if $\gamma^+$ points away from $\newL_j$}; \\
C_0-[j-k+1,j] & \text{if $\gamma^+$ points towards $\newL_j$}.
\end{cases}
\end{aligned}
\end{equation}
The sum in part (c) can be rewritten as follows:
\begin{align*}
\sum_{\alpha,\beta} \pi&-\tfrac{1}{2}(\theta_{\alpha}+\theta_{\beta})
=
(\pi-\tfrac{1}{2}(\theta_{\alpha_0}+\theta_{\alpha_1}))+
(\pi-\tfrac{1}{2}(\theta_{\alpha_{2r-1}}+\theta_{\alpha_{2r}}))+
\sum_{i=1}^{2r-2}\pi-\tfrac{1}{2}(\theta_{\alpha_i}+\theta_{\alpha_{i+1}})
\\
&=
(\pi-\tfrac{1}{2}(\theta_{\alpha_0}+\theta_{\alpha_1}))+
(\pi-\tfrac{1}{2}(\theta_{\alpha_{2r-1}}+\theta_{\alpha_{2r}}))+
\tfrac{1}{2}(\theta_{\alpha_1}+\theta_{\alpha_{2r-1}})-\pi+
\sum_{i=1}^{2r-1} (\pi-\theta_{\alpha_i}).
\end{align*}
where the first (respectively, second) term on the right hand side appears
if the arrow between $\newL_j$  and $\newL_{j-1}$ (respectively, $\newL_{j+1}$)
points towards $\newL_j$.
Hence,
\begin{equation} \label{e:relationeqn}
\sum_{\alpha,\beta} \pi-\tfrac{1}{2}(\theta_{\alpha}+\theta_{\beta})
-
\sum_{i=1}^{2r-1} (\pi-\theta_{\alpha_i})
+\pi-\theta_j=\tfrac{1}{2}(\psi_1+\psi_2-2\theta_j),
\end{equation}
where
$$
\psi_1=\begin{cases}
\theta_{\alpha_1} & \text{if $\gamma^-$ points away from $E_j$;} \\
2\pi-\theta_{\alpha_0} & \text{if $\gamma^-$ points towards $E_j$;}
\end{cases}
$$
and
$$
\psi_2=\begin{cases}
\theta_{\alpha_{2r-1}} & \text{if $\gamma^+$ points away from $E_j$;} \\
2\pi-\theta_{\alpha_{2r}} & \text{if $\gamma^+$ points towards $E_j$.}
\end{cases}$$
Note that $\gamma^-=\alpha_1$ if $\gamma^-$ points away from $E_j$ and $\gamma^-=\alpha_0$ otherwise;
see Figure~\ref{f:internalincidence}.
Similarly, $\gamma^+=\alpha_{2r-1}$ if $\gamma^+$ points
away from $E_j$ and $\gamma^+=\alpha_{2r}$ otherwise.
Applying~\eqref{e:gamma}, we see that the right hand side of~\eqref{e:relationeqn} simplifies to
$$\frac{1}{2}\left( -2\theta_j+\sum_{i\in [j,j-k-1]} \theta_i+\sum_{i\in [j-k+1,j]} \theta_i \right)=\pi-\tfrac{1}{2}(\theta_j+\theta_{j-k}).$$
Therefore, recalling that
$W_{\out}(\newL_j)=\{\alpha_1,\ldots ,\alpha_{2r-1}\}$,
equation~\eqref{e:boundaryangle2} is equivalent
to equation~\eqref{e:boundaryangle1}, which holds
by equation~\eqref{e:desiredresult}
in the proof of Proposition~\ref{p:externalformula}.
\end{proof}

\begin{remark}
For each arrow $\alpha\in Q_1(D)$, let $R_{\alpha}=\theta_{\alpha}/\pi$. Then Lemma~\ref{l:angles}(a) states
that for all $F\in Q_2(D)$,
\begin{equation} \label{e:Rcharge1}
\sum_{\alpha\in \bdry F} R_{\alpha}=2.
\end{equation}
Equation~\eqref{e:internalangle2} becomes:
\begin{equation} \label{e:Rcharge2}
\sum_{\alpha}(1- R_{\alpha})=2,
\end{equation}
for every internal vertex $I\in Q_0(D)$, where the sum
is over all arrows $\alpha$ incident with $I$.
Equation~\eqref{e:boundaryangle2} becomes:
\begin{equation} \label{e:Rcharge3}
\sum_{\alpha\in W_{\out}(\newL_j)} 1-R_{\alpha}=
1-\theta_j/\pi.
\end{equation}
We can regard equations~\eqref{e:Rcharge1}--\eqref{e:Rcharge3} as the definition of a \emph{consistent boundary $R$-charge}.
\end{remark}

Following~\cite[\S6]{bocklandt12}, we make the following definition (noting that Bocklandt uses the term `embedding with isoradial cycles').

\begin{definition} \label{d:isoradialembedding}
Let $Q$ be a quiver with faces. We will say that a map
$v\colon Q_0\to \mathbb{R}^2$ is an \emph{isoradial embedding} of $Q$
if the following hold:
\begin{enumerate}[(i)]
\item The map $v$ induces an embedding of $(Q_0,Q_1)$ into $\mathbb{R}^2$, taking an arrow to the line segment between the images of its endpoints.
\item For each $F\in Q_2$ with $$\bdry F =I_1\to  I_2\to \cdots \to I_r\to I_1,$$
the points $v(I_1),v(I_2),\ldots ,v(I_r)$, taken in order, form a polygon $\mathbb{T}_F$
inscribed on a unit circle.
\item Two polygons of form $\mathbb{T}_F$ which have non-empty intersection can intersect only in a single common edge or point.
\end{enumerate}
\end{definition}

Note that in an isoradial embedding, two polygons $\mathbb{T}_F$ intersect in an edge (respectively, a point, the empty set) if and only if the corresponding faces have
boundaries containing a unique common arrow (respectively, a unique common vertex, no common vertex) in $Q$.
The strictly convex polygons $\mathbb{T}_F$ tile a subset of $\mathbb{R}^2$, which we shall refer to as the \emph{image} of $Q$ under $v$.

\begin{corollary} \label{c:embedding}
If $D$ is of reduced type (see Definition~
\ref{d:reduced}),
then the quiver with faces $Q(D)$ can be
isoradially embedded into an $n$-sided polygon with
vertices $\newL_j$, $j=1,\ldots ,n$, and internal
angle $\pi-\frac{1}{2}(\theta_j+\theta_{j-k})$ at
vertex $\newL_j$.
\end{corollary}
\begin{proof}
We modify the discussion after~\cite[Rk~6.2]{bocklandt12} for the boundary case
(see also~\cite[\S3]{hv} for the dual case).
Part (a) of Lemma~\ref{l:angles} means that each cycle in $Q_2(D)$ can be embedded into
the plane as a polygon inscribed on a unit circle such that each arrow $\alpha$ stands on
an arc of angle $\theta_{\alpha}$. The angle between successive arrows $\alpha,\beta$ in a
cycle in $Q_2(D)$ is $\pi-\frac{1}{2}(\theta_{\alpha}+\theta_{\beta})$ (see Figure~\ref{f:successivearrows}),
so it follows from Lemma~\ref{l:angles}(b) and (c) that these polygons (which all have at least three sides since $D$ is of reduced type) tile the polygon mentioned above and hence give an isoradial embedding.
\end{proof}

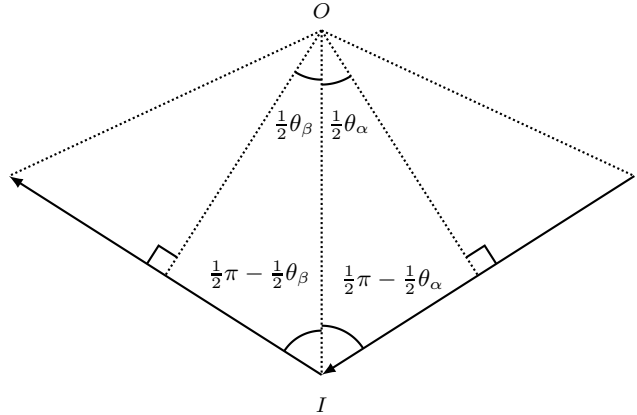
\begin{figure}
\[
\begin{tikzpicture}[scale=1.3,
 quivarrow/.style={black, -latex, thick},
 perpline/.style={thick,densely dotted,black}]

\newcommand{\radius}{3.5} 
\newcommand{\thetaalpha}{65} 
\newcommand{\thetabeta}{65} 
\newcommand{\rt}{0.22} 


\draw [quivarrow] (-90+\thetaalpha:\radius) -- (-90:\radius);
\draw [quivarrow] (-90:\radius) -- (-90-\thetabeta:\radius);


\coordinate (ma) at ($0.5*(-90+\thetaalpha:\radius)+0.5*(-90:\radius)$);
\coordinate (mb) at ($0.5*(-90:\radius)+0.5*(-90-\thetabeta:\radius)$);


\draw [perpline] (0,0) -- (ma);
\draw [perpline] (0,0) -- (mb);
\draw [perpline] (0,0) -- (0,-\radius);
\draw [perpline] (0,0) -- (-90+\thetaalpha:\radius);
\draw [perpline] (0,0) -- (-90-\thetabeta:\radius);


\draw (0,0.2) node {$O$};
\draw ($(0,-\radius)+(0,-0.3)$) node {$I$};


\draw [black,thick,domain=-90+0.5*\thetaalpha:-90] plot
({0.55*cos(\x)}, {0.55*sin(\x)});
\draw (-90+0.25*\thetaalpha:1) node {\small $\frac{1}{2}\theta_{\alpha}$};

\draw [black,thick,domain=-90-0.5*\thetabeta:-90] plot
({0.50*cos(\x)}, {0.50*sin(\x)});
\draw (-90-0.25*\thetabeta:1) node {\small $\frac{1}{2}\theta_{\beta}$};

\draw [black,thick,domain=90:0.5*\thetaalpha] plot
({0.50*cos(\x)}, {-\radius+0.50*sin(\x)});
\draw ($(0,-\radius)+(37+0.25*\thetaalpha:1.2)$) node
{\small $\frac{1}{2}\pi-\frac{1}{2}\theta_{\alpha}$};

\draw [black,thick,domain=90:180-0.5*\thetabeta] plot
({0.45*cos(\x)}, {-\radius+0.45*sin(\x)});
\draw ($(0,-\radius)+(138-0.25*\thetabeta:1.2)$) node
{\small $\frac{1}{2}\pi-\frac{1}{2}\theta_{\beta}$};


\draw [black,thick] (ma)++(0.5*\thetaalpha:\rt)
-- ++(90+0.5*\thetaalpha:\rt) -- ++(180+0.5*\thetaalpha:\rt);

\draw [black,thick] (mb)++(180-0.5*\thetabeta:\rt)
-- ++(90-0.5*\thetabeta:\rt) -- ++(-0.5*\thetabeta:\rt);

\end{tikzpicture}
\]
\caption{The angle between successive arrows incident with
a vertex is $\pi-\frac{1}{2}(\theta_{\alpha}+\theta_{\beta})$}
\label{f:successivearrows}
\end{figure}

Note that if $\theta_i=2\pi/n$, the polygon in Corollary~\ref{c:embedding} is regular.
If $D$ is not of reduced type, we still obtain a tiling of the same polygon using the above procedure, if we allow degenerate 2-sided tiles. But is not an isoradial embedding since the arrows in a two-cycle map onto the same line in the plane.

We remark that an isoradial embedding of $Q(D)$ in the above sense has been constructed explicitly in~\cite{ops}
in the case where $D$ is of reduced
type.
We recall this construction here in order to compare with the above.

Recall that a pair $I,J$ of $k$-subsets of $C_1$ is said to
be \emph{noncrossing}, or \emph{weakly separated}~
\cite{lz98} (see also~\cite[Defn.~3]{scott06}) if there are
no elements $a,b,c,d$, cyclically ordered around $C_1$,
such that $a,c\in I-J$ and $b,d\in J-I$.
Note that a pair $I,I$ is always noncrossing.
A collection of $k$-subsets is said to be noncrossing if it
is pairwise noncrossing.

By~\cite[Cor.~1]{scott06}, the set $\C=\C(D)$ of $k$-subsets of $C_1$
labelling a Postnikov diagram $D$ is a maximal noncrossing collection. By~\cite[Thm.~7.1]{ops}, every such
collection arises in this way.

\begin{definition} \label{d:GammaC}
Let $\C$ be any maximal noncrossing collection of $k$-subsets of $C_1$. We can define a quiver with faces $\Gamma=\Gamma(\C)$, with vertex set $\Gamma_0=\C$.
The arrows $\Gamma_1$ and faces $\Gamma_2$ are determined as follows, mimicking the definition of the CW-complex $\Sigma(\C)$ in~\cite[\S9]{ops}.

If $K$ is any $(k-1)$-subset of $C_1$, then the \emph{white clique} $\W(K)$ of $K$ is the set
$$\{I\in \C\,:\,K\subseteq I\},$$
which is given a cyclic order
$$K+a_1,K+a_2,\ldots ,K+a_r,K+a_1,$$
where $a_1,a_2,\ldots ,a_r,a_1$ are cyclically ordered clockwise in $C_1$.

If $L$ is any $(k+1)$-subset of $C_1$, then the \emph{black clique} $\B(L)$ of $L$ is the set
$$\{I\in \C\,:\,I\subseteq L\},$$
which is given a cyclic order
$$L-b_s,L-b_{s-1},\ldots ,L-b_1,L-b_{s},$$
where $b_1,b_2,\ldots ,b_s,b_1$ are cyclically ordered clockwise in $C_1$.

A clique is said to be \emph{nontrivial} when it has at least $3$ elements.
We let $\Gamma_2$ be the set of nontrivial cliques, or more precisely, we set
\[
  \Gamma_2^- =\{K : |\W(K)|\geq 3\}, \qquad
  \Gamma_2^+ =\{L \,:\, |\B(L)|\geq 3\}.
\]
For $I,J\in\C$, there is an arrow $\alpha\colon I\to J$ in $\Gamma_1$ if $J$ follows $I$ in the cyclic ordering of some nontrivial clique. Note: even if this occurs in more than one clique, there is only one arrow.
Then we may also define $\bdry\colon \Gamma_2\to \Gcyc$ in the obvious way.
\end{definition}

Note that an arrow $\alpha\colon I\to J$ in $\Gamma(\C)$ can occur in the boundary of at most 2 cliques, namely $\W(I\cap J)$ and $\B(I\cup J)$, and that it
will occur with the same orientation in both.
Indeed, if both these cliques are nontrivial, then $\alpha\colon I\to J$
must be an arrow in both boundaries, by \cite[Lemma~9.2]{ops}.

In~\cite[\S 9]{ops}, an isoradial embedding of $\Gamma(\C)$ is constructed as follows.
For all $i\in C_1$, choose unit vectors $v_i$ in $\mathbb{R}^2$, in the same clockwise order around the unit circle as in $C_1$ (see Figure~\ref{f:vicircle} for an example the case $n=7$).
Note that these points form a strictly convex polygon, which is the condition required by~\cite{ops} (see~\cite[Lemma 9.3]{ops}).
Here we assume the stronger property that they form a polygon inscribed in a unit circle.

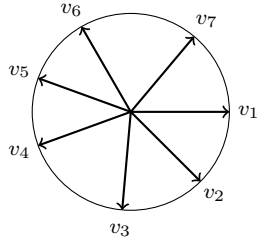
\begin{figure}
\[
\begin{tikzpicture}[scale=1.3]
 \draw (0,0) circle(1);
\foreach \j/\ang in {1/0,2/-45,3/-95,4/-160,5/-200,6/-240,7/-310} 
{ \draw [->,thick] (0,0)--(\ang:1);
 \path (\ang:1.2) node (v\j) {$v_\j$}; }
\end{tikzpicture}
\]
\caption{Typical points $v_1,\ldots ,v_n$ in the case $n=7$}
\label{f:vicircle}
\end{figure}

This is extended to a map $v$ on all subsets $J$ of $C_1$ by setting
$$v(J)=\sum_{i\in J}v_i.$$
Denote the point with position vector $v(J)$ by $J$.
Any $i\in C_0$ is common to the two adjacent edges
$i,i+1\in C_1$. Thus it makes sense to set $\theta_i$ to be the angle between $v_i$ and $v_{i+1}$ for $i\in C_0$.

\begin{lemma} \label{l:arrowangle}
Let $\alpha{\colon}I\to J$ be an arrow in $Q$. Set $K=I\cap J$ and $L=I\cup J$. Then
$$
 \directedangle JKI = \theta_{\alpha}= \directedangle ILJ,
$$
\end{lemma}

\begin{proof}
Firstly, note that $\theta_{\alpha}$ cannot be equal to $0$ or $2\pi$, by the remarks immediately following Definition~\ref{d:weight}. Now we know that
$$
 \overrightarrow{KI}=v_I-v_K=v_{I-J},
 \qquad  \overrightarrow{KJ}=v_J-v_K=v_{J-I}.
$$
The formula for the angle $\directedangle JKI$ then follows from the
definition of the $\theta_i$ and the fact that
$\weight_{\alpha}=(c,d)_0$, where $I-J=\{c\}$ and $J-I=\{d\}$.
The formula for the angle $\directedangle ILJ$ is proved similarly.
\end{proof}

Suppose $Q$ is a quiver with faces such that the boundary of every face is a cycle
of length at least $3$ and that $Q$ has an isoradial embedding $v$ whose image
is a subset of the plane bounded by a polygon, so the map $v$ induces an embedding of
the quiver $(Q_0,Q_1)$ into a disk. Then the given structure of quiver with faces on $(Q_0,Q_1)$ coincides with the structure inherited from the embedding and it follows that $Q$ is a dimer model in a disk.

\begin{theorem} \label{t:ops}
Let $D$ be a $(k,n)$-Postnikov diagram of reduced type and $
\C=\C(D)$ the set of labels of the
alternating regions. Let $\Gamma=\Gamma(\C)$ be the
associated quiver with faces,
as in Definition~\ref{d:GammaC},
and let $v_{\C}:\Gamma_0\to \mathbb{R}^2$ be the
restriction of $v$ to $\Gamma_0=\C$.
\begin{enumerate}[(a)]
\item The map $v_{\C}$ is an isoradial embedding of $\Gamma$. The image of
$\Gamma$ under $v_{\C}$ is the subset of the plane bounded by a convex polygon
with vertices $v(E_1),v(E_2),\ldots ,v(E_n)$ arranged clockwise around the boundary.
\item As quivers with faces, we have $\Gamma(\C)\cong Q(D)$.
\end{enumerate}
\end{theorem}

\begin{proof}
We first prove part (a).
To check part (ii) of Definition~\ref{d:isoradialembedding},
notice that, if $K\in \Gamma_2^-$, then $v(K+a)-v(K)=v_{a}$, for any $a\in C_1$.
Hence, the points
\[
  v(K+a_1),v(K+a_2),\ldots ,v(K+a_r)
\]
lie in order clockwise around a unit circle
centred at $v(K)$.
A similar statement holds for $L\in \Gamma_2^+$.
For parts (i) and (iii) of Definition~\ref{d:isoradialembedding}
we use~\cite[Prop.~9.4]{ops}, noting that by~\cite[Thm.~9.12]{ops}
there cannot be a pair of vertices $I,J\in \C$ with $\W(I\cap J)=\B(I\cup J)=\{I,J\}$,
so $\Gamma(\C)$ corresponds to $\Sigma(\C)$ in~\cite[\S9]{ops}.
The claim concerning the image follows
from~\cite[Prop.~9.8, Rk.~9.9 and Thm.~9.12]{ops}.

For part (b), we note that
by~\cite[Thm.~9.12]{ops}, $\Gamma(\C)$ is the dual of $G(D)$.
By Remark~\ref{r:plabicdual},
the dual of $G(D)$ is $Q(D)$ and the result follows.
\end{proof}

In particular, we see, using Lemma~\ref{l:arrowangle} and
Theorem~\ref{t:ops}, that $v_{\C}$ has the property
mentioned in the proof of Corollary~\ref{c:embedding}, i.e.\ that each arrow $\alpha$ stands on an arc of angle $\theta_{\alpha}$.

We colour a tile $\mathbb{T}_F$ in an isoradial embedding black
(respectively, white) if the anticlockwise (respectively, clockwise) ordering of its vertices corresponds to the ordering of the cycle $\bdry F$. Note that the tiling is bipartite in the sense that two tiles sharing an edge must be of opposite colour.
The bipartite tiling corresponding to the isoradial embedding in Theorem~\ref{t:ops}(a)
is referred to as the \emph{plabic tiling} corresponding to~$\C$ in~\cite{ops}.
For the Postnikov diagram $D$ in Figure~\ref{f:postfree37} and the vectors in Figure~\ref{f:vicircle}, the plabic tiling and the image of $\Gamma(\C(D))$ under the isoradial embedding in Theorem~\ref{t:ops}(a) are shown in Figure~\ref{f:postembedding}.
Note that the quiver is the same as the one in Figure~\ref{f:postfreequiver37}, but a little staightened.

\begin{figure}
\[
\begin{tikzpicture}[scale=1.2]
\newcommand{\tile}{[fill=gray]}
\foreach \j/\ang in {1/0,2/-45,3/-95,4/-160,5/-200,6/-240,7/-310} 
{ \path (\ang:1) coordinate (v\j); }
\foreach \x/\y/\z in
{1/2/3,2/3/4,3/4/5,4/5/6,5/6/7,1/6/7,1/2/7,1/2/4,2/4/5,1/4/7,1/4/5,1/5/7,1/5/6}
{\path (v\x)++(v\y)++(v\z) coordinate (q\x\y\z); }
\foreach \t/\h in {123/234,234/345,456/345,567/456,567/167,127/167,123/127,
234/124,124/123,124/245,245/234,345/245,245/145,145/456,456/156,156/145,156/567,
145/157,157/156,157/147,147/127,167/157,147/145,145/124,124/147,127/124}
 {\draw (q\t)--(q\h);}
\draw\tile (q124)--(q145)--(q245)--cycle;
\draw\tile (q234)--(q345)--(q245)--cycle;
\draw\tile (q123)--(q234)--(q124)--cycle;
\draw\tile (q124)--(q127)--(q147)--cycle;
\draw\tile (q145)--(q157)--(q147)--cycle;
\draw\tile (q145)--(q456)--(q156)--cycle;
\draw\tile (q167)--(q157)--(q156)--(q567)--cycle;
\end{tikzpicture}
\qquad
\begin{tikzpicture}[scale=1.2]
\foreach \j/\ang in {1/0,2/-45,3/-95,4/-160,5/-200,6/-240,7/-310} 
{ \node (v\j) at (\ang:1) {}; }
\foreach \x/\y/\z in
{1/2/4,2/4/5,1/4/7,1/4/5,1/5/7,1/5/6}
{\path {(v\x)++(v\y)++(v\z)} node (q\x\y\z) {$\scriptstyle \x\y\z$}; }
\foreach \x/\y/\z in
{1/2/3,2/3/4,3/4/5,4/5/6,5/6/7,1/6/7,1/2/7}
{\path {(v\x)++(v\y)++(v\z)} node (q\x\y\z) {$\scriptstyle \x\y\z$}; }
\foreach \h/\t in {123/234,234/345,456/345,567/456,567/167,127/167,123/127,
234/124,124/123,124/245,245/234,345/245,245/145,145/456,456/156,156/145,156/567,
145/157,157/156,157/147,147/127,167/157,147/145,145/124,124/147,127/124}
 {\draw [->] (q\t)--(q\h);}
\end{tikzpicture}
\]
\caption{The plabic tiling and the image of $Q(D)$ under $v_{\C(D)}$
 for the Postnikov diagram $D$ in Figure~\ref{f:postfree37} and $v_i$ in Figure~\ref{f:vicircle}}
\label{f:postembedding}
\end{figure}
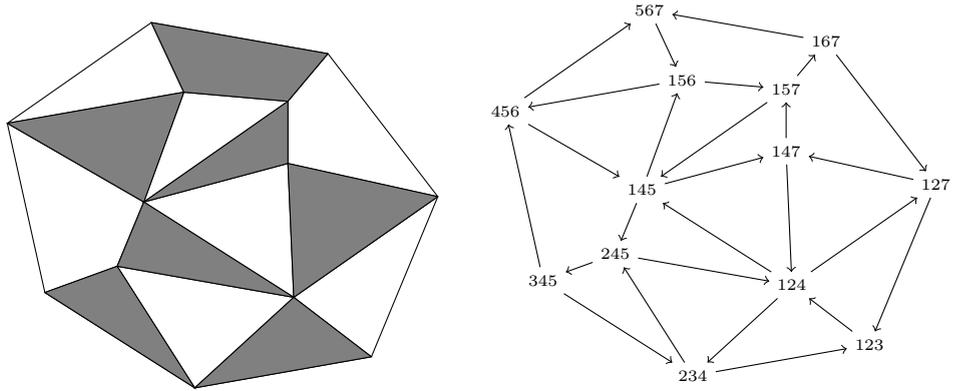

\section{Construction of a minimal path: first arrow} 
\label{s:legalarrow}

Our goal in this section is to show that, given any distinct pair of vertices
$I,J\in Q_0(D)$, there is an arrow, starting at $I$,
whose weight is constrained
in a way that makes it a candidate for the first arrow in a minimal path from $I$ to $J$.
We will then construct this path in Section~\ref{s:minimalpath}.

\begin{definition}
\label{d:aIJbIJ}
Suppose that $I,J$ are distinct noncrossing $k$-subsets of $C_1$.
Then $I-J$ and $J-I$ are nonempty and contained in non-overlapping
cyclic intervals in $C_1$.
Let $[i_1,i_2]$ and $[j_1,j_2]$ be the smallest
cyclic intervals containing $I-J$ and $J-I$, respectively.
We define two elements of $C_1$ associated to the ordered pair $(I,J)$ by setting
\begin{equation*}
  a=a(I,J)=j_2 \in J-I ,\qquad b=b(I,J)=i_1 \in I-J .
\end{equation*}
Observe that, since $a\not=b$, the interval $(a,b)_0$ in $C_0$ is nonempty.
See Figure~\ref{f:abdiagram} for a picture;
the vertices in $(a,b)_0$ are indicated by black dots.
\end{definition}

\begin{figure}
\[
\begin{tikzpicture}[scale=0.8,baseline=(bb.base)]
\newcommand{\fourteenth}{25.714} 
\newcommand{\polyradius}{2cm} 
\newcommand{\dotrad}{0.1cm} 
\newcommand{\bdrydotrad}{{0.8*\dotrad}} 
\path (0,0) node (bb) {}; 

\foreach \i in {1,...,14}
{\coordinate (b\i) at (90-\fourteenth*\i:\polyradius);}


\foreach \i/\j in
{1/2,2/3,3/4,4/5,5/6,6/7,7/8,8/9,9/10,10/11,11/12,
12/13,13/14,14/1}
\coordinate (m\i) at (${0.5}*(b\i)+{0.5}*(b\j)$);

\foreach \i/\j in
{1/2,2/3,3/4,4/5,5/6,6/7,7/8,8/9,9/10,10/11,11/12,
12/13,13/14,14/1}
\draw (b\i) -- (b\j);

\draw ($(m1)+(0.2,0.25)$) node {$j_1$};
\draw ($(m12)+(-0.2,0.25)$) node {$i_2$};
\draw ($(m5)+(0.5,-0.2)$) node {$j_2=a$};
\draw ($(m9)+(-0.65,-0.18)$) node {$i_1=b$};


\draw ($(m7)+(0,-1)$) node {$(a,b)_0$};
\draw ($(m11)+(-1.2,0)$) node {$I-J$};
\draw ($(m2)+(1.2,0.2)$) node {$J-I$};


\draw (m1) -- (m5);
\draw (m9) -- (m12);
\draw (b6) circle(\bdrydotrad) [fill=black];
\draw (b7) circle(\bdrydotrad) [fill=black];
\draw (b8) circle(\bdrydotrad) [fill=black];
\draw (b9) circle(\bdrydotrad) [fill=black];

\end{tikzpicture}
\]
\caption{A noncrossing pair}
\label{f:abdiagram}
\end{figure}
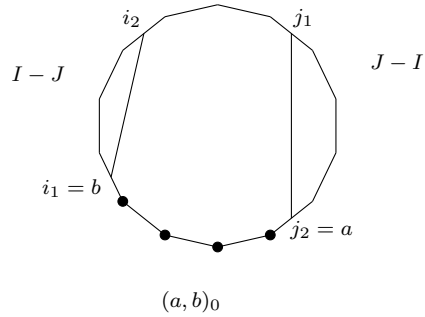

Note that, if $I,J$ are distinct $k$-subsets labelling regions of the same
Postnikov diagram $D$, then they are noncrossing, by \cite[Cor.~1]{scott06}.
Our goal may now be stated more precisely as follows.

\begin{proposition} \label{p:legalarrow}
Let $D$ be a Postnikov diagram
and let $I,J$ be distinct
vertices in $Q_0(D)$.
Let $a=a(I,J)$ and $b=b(I,J)$.
Then there is an arrow $\alpha$, starting at $I$,
whose weight $\weight_{\alpha}$ (as a subset of $C_0$) is disjoint from $(a,b)_0$.
\end{proposition}

\begin{proof}
We divide the proof into two cases, when $I$ is internal or external.
The required statements are then proved as
Lemmas~\ref{l:internallegalarrow} and~\ref{l:externallegalarrow} below.
\end{proof}

Thus, we start by assuming that $I$ is an internal vertex of $Q(D)$ and $J$ is an
arbitrary vertex. We use the notation introduced after Remark~\ref{r:perfectmatchings}.

\begin{lemma} \label{l:tilemma}
Fix $1\leq i\leq r$. Then $[\newt_{2i},\newt_{2i+1}]\not\subseteq (a,b)$.
\end{lemma}

\begin{proof}
Suppose, for a contradiction, that $[\newt_{2i},\newt_{2i+1}]\subseteq (a,b)$.
Then $a,\newt_{2i},\newt_{2i+1},b$ appear in cyclic order clockwise in $C_1$.
Since $\newt_{2i}\not\in I$ and $\newt_{2i}\in (a,b)$, we must have
$\newt_{2i}\not\in J$, so $\newt_{2i}\in I_{2i}-J$ by~\eqref{e:Iidescription}.
Similarly, $\newt_{2i+1}\in J-I_{2i}$.
We also have that $a\not\in I_{2i}$ since $a\not\in I$ and $a\not=\newt_{2i}$,
so $a\in J-I_{2i}$. Similarly, $b\in I_{2i}-J$.
This implies that $I_{2i}$ and $J$ are crossing,
which is the required contradiction, as $I_{2i},J\in Q_0(D)$.
\end{proof}

Let $\xi(I)$ be the path associated to $I$ in Definition~
\ref{d:xiI}. By Proposition~\ref{p:internalformula}, it
wraps around $C$ exactly $r-1$ times.
Let $C(r-1)$ be a connected $(r-1)$-fold cover of the graph $C$.
By choosing appropriate consecutive lifts
$(\tilde{\newt}_{i},\tilde{\newt}_{i+1})_0$, we can lift $\xi$ to
$C(r-1)$, obtaining a path $\tilde{\xi}$ which encircles $C(r-1)$ exactly once.

\begin{lemma} \label{l:internallegalarrow}
There is an $i\in \{1,2,\ldots ,r\}$
such that the weight $(\newt_{2i-1},\newt_{2i})_0$ of the
arrow $\alpha_{2i-1}{\colon}I\to I_{2i-1}$
is disjoint from $(a,b)_0$.
\end{lemma}

\begin{proof}
The preimage of $(a,b)_0$ in $C(r-1)$ consists of a disjoint union of
$r-1$ lifts of $(a,b)_0$.
If one of these, say $(\tilde{a},\tilde{b})_0$,
had a nonempty intersection with more than one of the lifts
$(\tilde{\newt}_{2i-1},\tilde{\newt}_{2i})_0$ in $\tilde{\xi}$
it would have non empty intersection with two consecutive odd intervals
$(\tilde{\newt}_{2i-1},\tilde{\newt}_{2i})_0$ and
$(\tilde{\newt}_{2i+1},\tilde{\newt}_{2i+2})_0$.
Then we'd have
$[\tilde{\newt}_{2i},\tilde{\newt}_{2i+1}]\subseteq (\tilde{a},\tilde{b})$,
and thus that $[\newt_{2i},\newt_{2i+1}]\subseteq (a,b)$,
a contradiction to Lemma~\ref{l:tilemma}.

Hence each of the $r-1$ lifts of $(a,b)_0$ can have nonempty intersection
with at most one of the $r$ intervals $(\tilde{\newt}_{2i-1},\tilde{\newt}_{2i})_0$.
Thus there must be one outgoing arrow $\alpha_{2i-1}$
whose weight $(\newt_{2i-1},\newt_{2i})_0$
does not intersect $(a,b)_0$, as required.
\end{proof}

We next consider the case where $I=\newL_j$ is a boundary vertex.
We use the notation for the incident arrows etc., introduced just before
Definition~\ref{d:xiIboundary}.
The statement of Lemma~\ref{l:tilemma} also still holds in this case,
but we need an extra lemma too.

\begin{lemma} \label{l:joinedintervals}
We have that $j\not\in (a,b)_0$.
\end{lemma}

\begin{proof}
Since $\newL_j=[j-k+1,j]$, we have $\newL_j-J\subseteq [j-k+1,j]$, while
$J-\newL_j\subseteq [j+1,j-k]$. It follows from the definition of $a$ and $b$
that $j,a,b,j$ appear in clockwise order around $C$, with $j\not=a$,
$a\not=b$ (but possibly $b=j$). The result follows.
\end{proof}

We consider the path $\xi_{\out}(\newL_j)$ associated to $\newL_j$ in Definition~\ref{d:xiIboundary}.
By Proposition~\ref{p:externalformula}(a), $\xi_{\out}(\newL_j)$ wraps around $C$ exactly
$r-1$ times but the vertex $j$ appears at the start and the end of $\xi_{\out}(\newL_j)$ (so that there are two edges in the overlap of $\xi$ at its start
and end): we have $\newt_1=j$ and $\newt_{2r}=j+1$.

By choosing appropriate consecutive lifts $(\tilde{\newt}_{i},\tilde{\newt}_{i+1})_0$),
we can lift $\xi_{\out}(\newL_j)$ to $C(r-1)$, obtaining a path $\tilde{\xi}$ which
encircles $C(r-1)$ exactly once plus the single vertex (a lift of $j$) in the
overlap at the start and the end.

\begin{lemma} \label{l:externallegalarrow}
There is an $i\in \{1,2,\ldots ,r\}$ such that
the weight $(\newt_{2i-1},\newt_{2i})_0$ of the outgoing arrow
$\alpha_{2i-1}\colon \newL_j\to I_{2i-1}$
is disjoint from $(a,b)_0$.
\end{lemma}

\begin{proof}
The interval $(a,b)_0$ has $r-1$ distinct lifts in $C(r-1)$.
By Lemma~\ref{l:joinedintervals}, no lift can have nonempty intersection with
both intervals $(\tilde{\newt}_1,\tilde{\newt}_2)_0$ and
$(\tilde{\newt}_{2r-1},\tilde{\newt}_{2r})_0$, since then it would contain a lift of $j$.
Arguing as in the proof of Lemma~\ref{l:internallegalarrow}, we obtain that
a lift of $(a,b)_0$ can only have nonempty intersection with at most one of
the intervals $(\tilde{\newt}_{2i-1},\tilde{\newt}_{2i})_0$.
The $r-1$ copies in total have nonempty intersection with at most $r-1$
such intervals.
But there are $r$ such intervals in total, so there is at least one
such interval that has empty
intersection with $(a,b)_0$, as required.
\end{proof}

\section{The rank one modules $\module_I$.} 
\label{s:rankone}

Consider the quiver with vertices $C_0$
and, for each edge $i\in C_1$, a pair of arrows $x_i{\colon}i-1\to i$
and $y_i{\colon}i\to i-1$. Then let $B$ be the quotient of the
path algebra (over $\F$) of this quiver by the ideal generated by the
$2n$ relations $xy=yx$ and $x^k=y^{n-k}$, interpreting $x$ and $y$
as arrows of the form $x_i,y_i$ appropriately and starting at any vertex.
See Figure~\ref{f:JKSquiver} for an example when $n=5$.

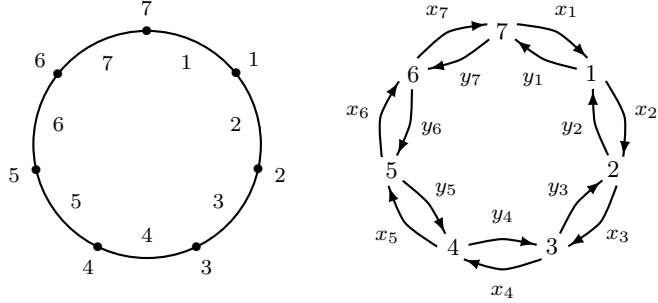
\begin{figure}
\[
\begin{tikzpicture}[scale=1,baseline=(bb.base)]
\newcommand{\seventh}{51.4} 
\newcommand{\circradius}{1.5cm} %
\newcommand{\inradius}{1.2cm} %
\newcommand{\outradius}{1.8cm} %
\path (0,0) node (bb) {}; 

\draw[black,thick] (0,0) circle(\circradius);
\foreach \j in {1,...,7}
{\draw (90-\seventh*\j:\circradius) node[black] {$\bullet$};
 \draw (90-\seventh*\j:\outradius) node[black] {$\j$};
 \draw (90+\seventh/2-\seventh*\j:\inradius) node[black] {$\j$}; }
\end{tikzpicture}
\quad\quad
\begin{tikzpicture}[scale=1,baseline=(bb.base),
quivarrow/.style={black, -latex, thick}]
\newcommand{\seventh}{51.4} 
\newcommand{\circradius}{1.5cm} %
\newcommand{\innermidradius}{1.25cm} 
\newcommand{\outermidradius}{1.65cm} 
\newcommand{\innerradius}{1.4cm} 
\newcommand{\outerradius}{1.6cm} 
\newcommand{\xradius}{1.95cm} 
\newcommand{\yradius}{0.95cm} 
\path (0,0) node (bb) {}; 


\foreach \i in {1,...,7}
{ \coordinate (b\i) at (90-\seventh*\i:\circradius);
\coordinate (outerplus\i) at (97-\seventh*\i:\outerradius);
\coordinate (outerminus\i) at (83-\seventh*\i:\outerradius);
\coordinate (outermid\i) at (90-\seventh*\i-0.5*\seventh:\outermidradius);
\coordinate (innermid\i) at (90-\seventh*\i+0.5*\seventh:\innermidradius);
\coordinate (innerplus\i) at (97-\seventh*\i:\innerradius);
\coordinate (innerminus\i) at (83-\seventh*\i:\innerradius);
\coordinate (xpos\i) at (90-\seventh*\i+0.5*\seventh:\xradius);
\coordinate (ypos\i) at (90-\seventh*\i+0.5*\seventh:\yradius);
\draw (b\i) node {\small $\i$};
}


\foreach \i/\j in {1/2,2/3,3/4,4/5,5/6,6/7,7/1}
\draw [black,thick] plot[smooth]
coordinates {(outerminus\i) (outermid\i)
(outerplus\j)}
[postaction=decorate, decoration={markings,
  mark= at position 0.98 with \arrow{latex}}];

\foreach \i/\j in {2/1,3/2,4/3,5/4,6/5,7/6,1/7}
\draw [black,thick] plot[smooth]
coordinates {(innerplus\i) (innermid\i)
(innerminus\j)}
[postaction=decorate, decoration={markings,
  mark= at position 0.97 with \arrow{latex}}];


\foreach \i in {1,...,7}
\draw (xpos\i) node {$x_{\i}$};

\foreach \i in {1,...,7}
\draw (ypos\i) node {$y_{\i}$};

\end{tikzpicture}
\]
\caption{The graph $C$ and corresponding quiver}
\label{f:JKSquiver}
\end{figure}

The completion $\widehat{B}$ of $B$ coincides with the quotient of the completed path
algebra of the graph $C$, i.e.\ the doubled quiver as above,
by the closure of the ideal generated by the relations above.
The algebras $B$ and $\widehat{B}$ were introduced
in~\cite[\S3]{jks}.

The centre $Z$ of $B$ is the polynomial ring ${\F}[t]$,
where $t=\sum_{i=1}^n x_iy_i$.
The (maximal) Cohen-Macaulay $B$-modules are precisely those which are
free as $Z$-modules. Indeed, such a module $M$ is given by a representation
$\{M_i\,:\,i\in C_0\}$ of
the quiver with each $M_i$ a free $Z$-module of the same rank,
which is the rank of $M$ (cf. \cite[\S 3]{jks}).

\begin{definition} \label{d:moduleMI}
For $I$ any $k$-subset of $C_1$, we define a rank one $B$-module
\[
  \module_I = (U_i,\ i\in C_0 \,;\, x_i,y_i,\, i\in C_1)
\]
as follows (cf. \cite[\S5]{jks}).
For each vertex $i\in C_0$, set $U_i=\F[t]$ and,
for each edge $i\in C_1$, set
\begin{itemize}
\item[] $x_i\colon U_{i-1}\to U_{i}$ to be multiplication by $1$ if $i\in I$ and by $t$ if $i\not\in I$,
\item[] $y_i\colon U_{i}\to U_{i-1}$ to be multiplication by $t$ if $i\in I$ and by $1$ if $i\not\in I$.
\end{itemize}
\end{definition}

The module $\module_I$ can be represented by a lattice diagram
$\LL_I$ in which $U_0,U_1,U_2,\ldots U_n$ are represented by columns from
left to right (with $U_0$ and $U_n$ to be identified).
The vertices in each column correspond to the natural monomial basis of $\F [t]$.
The column corresponding to $U_{i+1}$ is displaced half a step vertically
downwards (respectively, upwards) in relation to $U_i$ if $i\in I$
(respectively, $i\not \in I$), and the actions of $x_i$ and $y_i$ are
shown as diagonal arrows. Note that the $k$-subset $I$ can then be read off as
the set of labels on the arrows pointing down to the right which are exposed
to the top of the diagram. For example, the lattice picture $\LL_{\{1,4,5\}}$
in the case $k=3$, $n=8$, is shown in Figure~\ref{f:LIexample}.

\begin{remark} \label{r:projectives}
Fix $j\in C_0$, and consider the module $\module_{\newL_j}$.
The corresponding lattice diagram is bounded at the top by paths going down from $j-k$ to $j$. It is easy to see that
there is a bijection between the vertices of the diagram and a
basis of $Be_{j-k}$, where $e_{j-k}$ is the idempotent corresponding
to the vertex $j-k$. We see that
the module $\module_{\newL_j}$ is the projective indecomposable $B$-module corresponding to the vertex $j-k$.
\end{remark}

\begin{figure}
\[
\begin{tikzpicture}[scale=0.8,baseline=(bb.base),
quivarrow/.style={black, -latex, thick}]
\newcommand{\seventh}{51.4} 
\newcommand{\circradius}{1.5cm}
\newcommand{\inradius}{1.2cm}
\newcommand{\outradius}{1.8cm}
\newcommand{\dotrad}{0.1cm} 
\newcommand{\bdrydotrad}{{0.8*\dotrad}} 
\path (0,0) node (bb) {}; 


\draw (0,0) circle(\bdrydotrad) [fill=black];
\draw (0,2) circle(\bdrydotrad) [fill=black];
\draw (1,1) circle(\bdrydotrad) [fill=black];
\draw (2,0) circle(\bdrydotrad) [fill=black];
\draw (2,2) circle(\bdrydotrad) [fill=black];
\draw (3,1) circle(\bdrydotrad) [fill=black];
\draw (3,3) circle(\bdrydotrad) [fill=black];
\draw (4,0) circle(\bdrydotrad) [fill=black];
\draw (4,2) circle(\bdrydotrad) [fill=black];
\draw (5,1) circle(\bdrydotrad) [fill=black];
\draw (6,0) circle(\bdrydotrad) [fill=black];
\draw (6,2) circle(\bdrydotrad) [fill=black];
\draw (7,1) circle(\bdrydotrad) [fill=black];
\draw (7,3) circle(\bdrydotrad) [fill=black];
\draw (8,2) circle(\bdrydotrad) [fill=black];
\draw (8,4) circle(\bdrydotrad) [fill=black];


\draw [quivarrow,shorten <=5pt, shorten >=5pt] (0,2)
-- node[above]{$1$} (1,1);
\draw [quivarrow,shorten <=5pt, shorten >=5pt] (1,1) -- node[above]{$1$} (0,0);
\draw [quivarrow,shorten <=5pt, shorten >=5pt] (2,2) -- node[above]{$2$} (1,1);
\draw [quivarrow,shorten <=5pt, shorten >=5pt] (1,1) -- node[above]{$2$} (2,0);
\draw [quivarrow,shorten <=5pt, shorten >=5pt] (3,3) -- node[above]{$3$} (2,2);
\draw [quivarrow,shorten <=5pt, shorten >=5pt] (2,2) -- node[above]{$3$} (3,1);
\draw [quivarrow,shorten <=5pt, shorten >=5pt] (3,1) -- node[above]{$3$} (2,0);
\draw [quivarrow,shorten <=5pt, shorten >=5pt] (3,3) -- node[above]{$4$} (4,2);
\draw [quivarrow,shorten <=5pt, shorten >=5pt] (4,2) -- node[above]{$4$} (3,1);
\draw [quivarrow,shorten <=5pt, shorten >=5pt] (3,1) -- node[above]{$4$} (4,0);
\draw [quivarrow,shorten <=5pt, shorten >=5pt] (4,2) -- node[above]{$5$} (5,1);
\draw [quivarrow,shorten <=5pt, shorten >=5pt] (5,1) -- node[above]{$5$} (4,0);
\draw [quivarrow,shorten <=5pt, shorten >=5pt] (6,2) -- node[above]{$6$} (5,1);
\draw [quivarrow,shorten <=5pt, shorten >=5pt] (5,1) -- node[above]{$6$} (6,0);
\draw [quivarrow,shorten <=5pt, shorten >=5pt] (6,2) -- node[above]{$7$} (7,1);
\draw [quivarrow,shorten <=5pt, shorten >=5pt] (7,1) -- node[above]{$7$} (6,0);
\draw [quivarrow,shorten <=5pt, shorten >=5pt] (7,3) -- node[above]{$7$} (6,2);
\draw [quivarrow,shorten <=5pt, shorten >=5pt] (7,3) -- node[above]{$8$} (8,2);
\draw [quivarrow,shorten <=5pt, shorten >=5pt] (8,2) -- node[above]{$8$} (7,1);
\draw [quivarrow,shorten <=5pt, shorten >=5pt] (8,4) -- node[above]{$8$} (7,3);

\draw [dotted] (0,-2) -- (0,2);
\draw [dotted] (8,-2) -- (8,4);

\draw [dashed] (4,-2) -- (4,-1);

\end{tikzpicture}
\]
\caption{The module $\module_{\{1,4,5\}}$ in the case $k=3$, $n=8$}
\label{f:LIexample}
\end{figure}
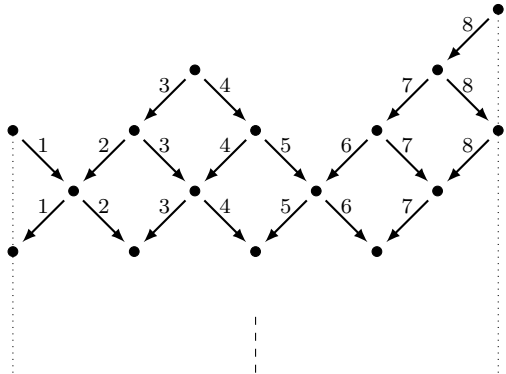

Every $B$-module has a canonical endomorphism $t$, that is, multiplication by $t\in Z$.
For $\module_J$ this corresponds to shifting $\LL_J$ one step downwards.
Since $Z$ is central, $\Hom_B(M,N)$ is
a $Z$-module for arbitrary $B$-modules $M$ and $N$.
If $M,N$ are free $Z$-modules, then so is $\Hom_B(M,N)$.

\begin{definition} \label{d:degree}
A \emph{monomial} morphism $\varphi\colon \module_I\to \module_J$ is
one determined by an embedding of $\LL_I$ in $\LL_J$.
Since the cokernel of any such $\varphi$ is clearly finite dimensional,
we may define the \emph{degree} of $\varphi$ to be the element of $\mathbb{N}C_0$
that counts the multiplicities of the simple modules in a
composition series for its cokernel.
As for the weights on arrows, we identify a
subset of $C_0$ with the sum of its elements, and
refer to a monomial morphism as being \emph{sincere}
if the support of its degree is $C_0$,
and \emph{insincere} otherwise.
\end{definition}

Note that the composition of monomial morphisms is monomial and
that degree is additive under composition.
The degree of the endomorphism $t$ is precisely $C_0$, without multiplicity.
Given $k$-subsets $I,J$ of $C_1$, let
\begin{equation} \label{e:embedJI}
  \embed_{JI}\colon \module_I \to \module_J
\end{equation}
denote the monomial morphism with minimal codimension.
This can be represented by embedding $\LL_I$ into $\LL_J$ as high as possible,
i.e.\ so that at least one vertex on the top boundary of $\LL_J$ lies in image of $\LL_I$.
Thus $\embed_{JI}$ is insincere.
Note that $g_{II}$ is the identity map on $\module_I$.
We have the following, which explains the terminology
\emph{monomial morphism}.

\begin{lemma} \label{l:homLILJ}
Let $I,J$ be $k$-subsets of $C_1$. Then
$\Hom_B(\module_I,\module_J)$ is a free rank one $\F[t]$-module
generated by $\embed_{JI}$. Furthermore, $t^m\embed_{JI}$ may be
characterised as the unique monomial morphism of its degree
$\module_I\to \module_J$. In particular, $\embed_{JI}$ is
the unique insincere monomial morphism $\module_I\to \module_J$.
\end{lemma}

\begin{proof}
For the first part, observe that there is a basis of $\Hom_B(\module_I,\module_J)$ consisting of monomial morphisms, given by taking all embeddings of
$\LL_I$ into $\LL_J$, which are clearly just downward shifts of the highest possible one.
In other words, they are the maps $t^m\embed_{JI}$, for all $m\geq0$, as required.
For the second part, note that the degree of $t^m\embed_{JI}$ is equal to to the degree of
$\embed_{JI}$ plus $mC_0$. The last part follows from this and the fact that
$g_{JI}$ is insincere, by the remarks immediately following Definition~\ref{d:degree}.
\end{proof}

For a nonempty subset $V$ of $C_0$ and $k$-subsets $I,J$ of $C_1$, we write $I\leq_V J$
if the support of the degree of $\embed_{JI}$ does not contain any element of $V$.

\begin{lemma}
Let $V$ be a subset of $C_0$. Then the relation $\leq_V$ is a partial
order on the collection of all $k$-subsets of $C_1$.
\end{lemma}

\begin{proof}
Reflexivity is clear. If $I \leq_V J$ and $J \leq_V K$, then
$\coker \embed_{JI}$ and $\coker \embed_{KJ}$ have composition
series containing no simple $S_i$ with $i\in V$. It follows that the
composition of these two morphisms has the same property. It thus
coincides with $\embed_{JI}$, and it follows that $I\leq_V J$, so $\leq_V$
is transitive. Suppose that $I\leq_V J$ and $J \leq_V I$.
Then the composition $\embed_{IJ}\embed_{JI}{\colon}\module_I \to \module_I$
has cokernel containing no simple $S_i$ with $i\in V$, and thus must be
the identity map and we have $I=J$ as required.
\end{proof}

\section{Existence of minimal path} 
\label{s:minimalpath}

In this section, we show that there is a minimal path between any pair of
vertices in the quiver $Q(D)$ of a Postnikov diagram $D$.

\begin{proposition} \label{p:degree}
Let $I,J$ be distinct noncrossing $k$-subsets of $C_1$.
Let $a=a(I,J)$ and $b=b(I,J)$, as in Definition~\ref{d:aIJbIJ}.
Let $I-J=\{b_1,\ldots ,b_s\}$, writing the elements in clockwise
order starting at $b=b_1$.
Let $J-I=\{a_1,\ldots ,a_s\}$, writing the elements in anticlockwise
order starting at $a=a_1$.
Then the degree of $\embed_{JI}\colon \module_I\to \module_J$
is equal to $\sum_{j=1}^s (b_j,a_j)_0$.
\end{proposition}

To prove this, we first need the following technical lemma:

\begin{lemma} \label{l:addinglayer}
Let $I$ be a $k$-subset of $C_1$, and suppose that $J=I-c+d$ for $c\not=d$
in $C_1$. Then $\LL_J$ can be obtained from $\LL_I$ by adding one vertex
to the top of each of the columns corresponding to the vertices in
$(c,d)_0$, together with the required additional arrows, i.e.\
arrows pointing downwards between any new vertex $v$ and the vertices
in adjacent columns.
\end{lemma}

\begin{proof}
Let $c_1,c_2,\ldots ,c_{N}$ be the elements of $I$ which are clockwise
of $c$ and anticlockwise of $d$. Let $c_{N+1}$ be the first element
of $I$ clockwise past $d$ (possibly equal to $c$).
The added vertices and arrows cover the arrows $c,c_1,c_2,\ldots ,c_N$
exposed to the top of the diagram. The new exposed arrows are
$c_1,c_2, \ldots ,c_N,d$. Thus we see that
the diagram constructed as indicated corresponds to the $k$-subset $I-c+d$
as required. This is illustrated in Figure~\ref{f:addinglayer},
which shows the top part of $\LL_I$ between the columns of arrows $c_1$ and
$c_{N+1}$ (which may or not be identified). The new arrows are shown in gray.
\end{proof}

\begin{figure}
\[
\begin{tikzpicture}[scale=0.6,baseline=(bb.base),
normalarrow/.style={black, -latex, thick},
grayarrow/.style={-latex, gray}]
\newcommand{\seventh}{51.4} 
\newcommand{\circradius}{1.5cm}
\newcommand{\inradius}{1.2cm}
\newcommand{\outradius}{1.8cm}
\newcommand{\dotrad}{0.1cm} 
\newcommand{\bdrydotrad}{{0.8*\dotrad}} 
\path (0,0) node (bb) {}; 


\draw (0,0) circle(\bdrydotrad) [fill=black];
\draw (1,-1) circle(\bdrydotrad) [fill=black];
\draw (2,0) circle(\bdrydotrad) [fill=black];
\draw (3,1) circle(\bdrydotrad) [fill=black];
\draw (4,2) circle(\bdrydotrad) [fill=black];
\draw (5,1) circle(\bdrydotrad) [fill=black];
\draw (6,2) circle(\bdrydotrad) [fill=black];
\draw (7,3) circle(\bdrydotrad) [fill=black];
\draw (8,4) circle(\bdrydotrad) [fill=black];
\draw (9,5) circle(\bdrydotrad) [fill=black];
\draw (10,4) circle(\bdrydotrad) [fill=black];
\draw (11,5) circle(\bdrydotrad) [fill=black];
\draw (12,6) circle(\bdrydotrad) [fill=black];
\draw (14,8.5) circle(\bdrydotrad) [fill=black];
\draw (15,7.5) circle(\bdrydotrad) [fill=black];
\draw (16,8.5) circle(\bdrydotrad) [fill=black];
\draw (17,9.5) circle(\bdrydotrad) [fill=black];
\draw (18,10.5) circle(\bdrydotrad) [fill=black];
\draw (19,11.5) circle(\bdrydotrad) [fill=black];
\draw (20,10.5) circle(\bdrydotrad) [fill=black];


\draw (1,1) circle(\bdrydotrad);
\draw (2,2) circle(\bdrydotrad);
\draw (3,3) circle(\bdrydotrad);
\draw (4,4) circle(\bdrydotrad);
\draw (5,3) circle(\bdrydotrad);
\draw (6,4) circle(\bdrydotrad);
\draw (7,5) circle(\bdrydotrad);
\draw (8,6) circle(\bdrydotrad);
\draw (9,7) circle(\bdrydotrad);
\draw (10,6) circle(\bdrydotrad);
\draw (11,7) circle(\bdrydotrad);
\draw (15,9.5) circle(\bdrydotrad);
\draw (16,10.5) circle(\bdrydotrad);
\draw (17,11.5) circle(\bdrydotrad);


\draw [normalarrow,shorten <=3pt, shorten >=3pt] (0,0)
-- node[pos=0.2,below,inner sep=8pt]{\small $c$} (1,-1);
\draw [normalarrow,shorten <=3pt, shorten >=3pt] (2,0)
-- (1,-1);
\draw [normalarrow,shorten <=3pt, shorten >=3pt] (3,1)
-- (2,0);
\draw [normalarrow,shorten <=3pt, shorten >=3pt] (4,2)
-- (3,1);
\draw [normalarrow,shorten <=3pt, shorten >=3pt] (4,2)
-- node[pos=0.2,below,inner sep=8pt]{\small $c_1$} (5,1);
\draw [normalarrow,shorten <=3pt, shorten >=3pt] (6,2)
-- (5,1);
\draw [normalarrow,shorten <=3pt, shorten >=3pt] (7,3)
-- (6,2);
\draw [normalarrow,shorten <=3pt, shorten >=3pt] (8,4)
-- (7,3);
\draw [normalarrow,shorten <=3pt, shorten >=3pt] (9,5)
-- (8,4);
\draw [normalarrow,shorten <=3pt, shorten >=3pt] (9,5)
-- node[pos=0.2,below,inner sep=8pt]{\small $c_2$} (10,4);
\draw [normalarrow,shorten <=3pt, shorten >=3pt] (11,5)
-- (10,4);
\draw [normalarrow,shorten <=3pt, shorten >=3pt] (12,6)
-- (11,5);
\draw [normalarrow,shorten <=3pt, shorten >=3pt] (14,8.5) -- node[pos=0.2,below,inner sep=8pt]{\small $c_N$} (15,7.5);
\draw [normalarrow,shorten <=3pt, shorten >=3pt] (16,8.5) -- (15,7.5);
\draw [normalarrow,shorten <=3pt, shorten >=3pt] (17,9.5) -- (16,8.5);
\draw [normalarrow,shorten <=3pt, shorten >=3pt] (18,10.5) -- (17,9.5);
\draw [normalarrow,shorten <=3pt, shorten >=3pt] (19,11.5) -- (18,10.5);
\draw [normalarrow,shorten <=3pt, shorten >=3pt] (19,11.5) -- node[pos=0.98,above,outer sep=7pt]{\small $c_{N+1}$} (20,10.5);


\draw [grayarrow,shorten <=3pt, shorten >=3pt] (1,1) -- (0,0);
\draw [grayarrow,shorten <=3pt, shorten >=3pt] (1,1) -- (2,0);
\draw [grayarrow,shorten <=3pt, shorten >=3pt] (2,2) -- (1,1);
\draw [grayarrow,shorten <=3pt, shorten >=3pt] (2,2) -- (3,1);
\draw [grayarrow,shorten <=3pt, shorten >=3pt] (3,3) -- (2,2);
\draw [grayarrow,shorten <=3pt, shorten >=3pt] (3,3) -- (4,2);
\draw [grayarrow,shorten <=3pt, shorten >=3pt] (4,4)
-- node[pos=0.2,below,inner sep=8pt,black]{\small $c_1$} (5,3);
\draw [grayarrow,shorten <=3pt, shorten >=3pt] (4,4) -- (3,3);
\draw [grayarrow,shorten <=3pt, shorten >=3pt] (5,3) -- (4,2);
\draw [grayarrow,shorten <=3pt, shorten >=3pt] (5,3) -- (6,2);
\draw [grayarrow,shorten <=3pt, shorten >=3pt] (6,4) -- (5,3);
\draw [grayarrow,shorten <=3pt, shorten >=3pt] (6,4) -- (7,3);
\draw [grayarrow,shorten <=3pt, shorten >=3pt] (7,5) -- (6,4);
\draw [grayarrow,shorten <=3pt, shorten >=3pt] (7,5) -- (8,4);
\draw [grayarrow,shorten <=3pt, shorten >=3pt] (8,6) -- (7,5);
\draw [grayarrow,shorten <=3pt, shorten >=3pt] (8,6) -- (9,5);
\draw [grayarrow,shorten <=3pt, shorten >=3pt] (9,7) -- (8,6);
\draw [grayarrow,shorten <=3pt, shorten >=3pt] (9,7)
-- node[pos=0.2,below,inner sep=8pt,black]{\small $c_2$} (10,6);
\draw [grayarrow,shorten <=3pt, shorten >=3pt] (10,6) -- (9,5);
\draw [grayarrow,shorten <=3pt, shorten >=3pt] (10,6) -- (11,5);
\draw [grayarrow,shorten <=3pt, shorten >=3pt] (11,7) -- (10,6);
\draw [grayarrow,shorten <=3pt, shorten >=3pt] (11,7) -- (12,6);
\draw [grayarrow,shorten <=3pt, shorten >=3pt] (15,9.5) -- (14,8.5);
\draw [grayarrow,shorten <=3pt, shorten >=3pt] (15,9.5) -- (16,8.5);
\draw [grayarrow,shorten <=3pt, shorten >=3pt] (14,10.5) -- node[pos=0.2,below,inner sep=8pt,black]{\small $c_N$} (15,9.5);
\draw [grayarrow,shorten <=3pt, shorten >=3pt] (16,10.5) -- (15,9.5);
\draw [grayarrow,shorten <=3pt, shorten >=3pt] (16,10.5) -- (17,9.5);
\draw [grayarrow,shorten <=3pt, shorten >=3pt] (17,11.5) -- (16,10.5);
\draw [grayarrow,shorten <=3pt, shorten >=3pt] (17,11.5) -- node[pos=0.2,below,inner sep=8pt,black]{\small $d$} (18,10.5);


\draw [dotted,very thick] (12.5,6.7) -- (13.25,6.7);

\end{tikzpicture}
\]
\caption{Proof of Lemma~\ref{l:addinglayer}}
\label{f:addinglayer}
\end{figure}
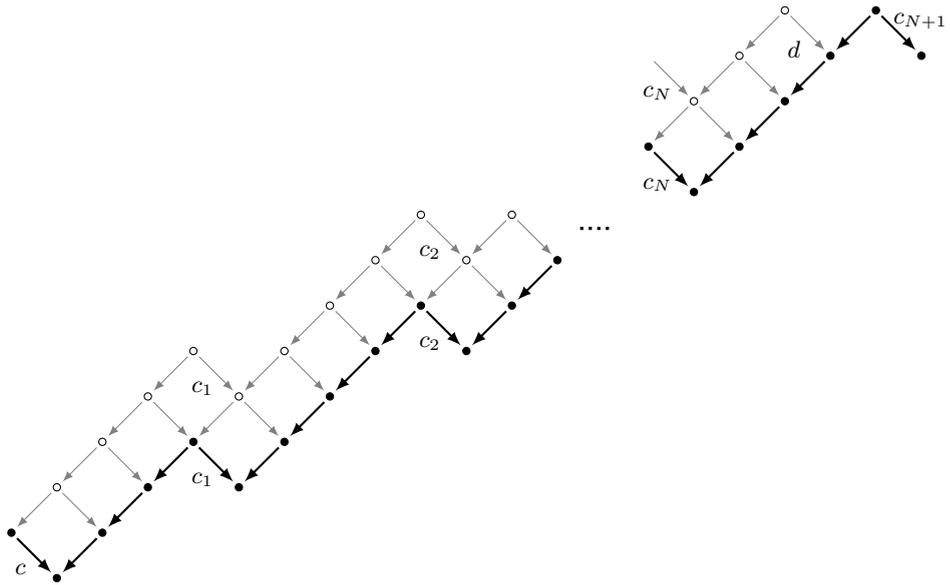

\begin{proof}[Proof of Proposition~\ref{p:degree}]
We proceed by induction on $s$.
The base case is $s=1$. Then $I-J=\{b_1\}$
and $J=I=\{a_1\}$ for elements $a_1,b_1$.
We apply Lemma~\ref{l:addinglayer} to the pair $I$ and $J=I-b+a=I-b_1+a_1$,
and see that $\LL_J$ can be obtained from $\LL_I$ by adding a layer of
vertices to the columns $b_1,\ldots ,a_1-1$ to the lattice picture, giving the result
in this case.

Suppose that $s>1$ and the result holds
for smaller $s$. Let $I_1=I-b+a=I-b_1+a_1$.
Then, by again applying Lemma~\ref{l:addinglayer}, $\LL_{I_1}$ can be obtained from $\LL_I$ by adding a
layer of vertices to the columns
$b_1,\ldots ,a_1-1$.
We have $I_1-J=(I-J)-b_1=\{b_2,\ldots ,b_s\}$ and
$J-I_1=(J-I)-a_1=\{a_2,\ldots ,a_s\}$.
The result now follows from applying the inductive hypothesis to the pair
$I_1,J$.
\end{proof}

\begin{corollary} \label{c:weightequalsdegree}
Let $\alpha{\colon}I\to J$ be an arrow in $Q_1(D)$.
Then the degree of $\embed_{JI}$ is equal to the weight $\weight_{\alpha}$
(see Definition~\ref{d:weight}).
\end{corollary}

\begin{proof}
We apply Lemma~\ref{l:addinglayer} to the pair $I,J$, noting that
$J=I-b+a$, where $a,b$ are the strands crossing $\alpha$,
in such a way that the weight of $\alpha$ is $(b,a)_0$.
\end{proof}

We obtain the following corollary:

\begin{corollary} \label{c:simpletest}
Let $I,J$ be distinct noncrossing $k$-subsets of $C_1$.
For any $c,d\in C_1$, we have $I \leq_{(c,d)_0}J$
if and only if $(c,d)_0\subseteq (a(I,J),b(I,J))_0$.
\end{corollary}

\begin{proof}
By Proposition~\ref{p:degree}, we have $I\leq_{(c,d)_0}J$ if and only if
no element of $(c,d)_0$ lies in the support of
$\sum_{j=1}^s (b_j,a_j)_0$. This support is $(b_1,a_1)_0=(b(I,J),a(I,J))_0,$
so $I\leq_{(c,d)_0}J$ holds if and only if
$(c,d)_0\subseteq (a(I,J),b(I,J))_0$, as required.
\end{proof}

Note that Corollary~\ref{c:simpletest} means, in particular, that
$I \leq_{(a(I,J),b(I,J))_0} J$ for any pair of (distinct) vertices $I,J\in Q_0(D)$.

\begin{lemma} \label{l:arrowless}
Let $I\neq J\in Q_0(D)$ and set $a=a(I,J)$, $b=b(I,J)$.
Suppose that there is an arrow $\alpha\colon I \to I_1$ in $Q_1(D)$
whose weight $w_\alpha$ is disjoint from $(a,b)_0$.
Then $I_1 \leq_{(a,b)_0} J$.
\end{lemma}

\begin{proof}
If $I_1=I-c+d$, then, by assumption, the weight $w_\alpha=(c,d)_0$ is a subset of $(b,a)_0$.
Applying Lemma~\ref{l:addinglayer} to the pair $I,I_1$,
we see that $\LL_{I_1}$ is obtained from $\LL_{I}$ by adding vertices
at the top of the columns corresponding to entries in the set
$(c,d)_0$.

This additional layer is part of the first layer added to $\module_I$ in
the proof of Proposition~\ref{p:degree}.
Additional layers are added to this to eventually reach $\module_J$.
Since $I \leq_{(a,b)_0} J$, all of these layers involve only columns
corresponding to elements of $(b,a)_0$, and we see that
$I_1 \leq_{(a,b)_0} J$, as required.
\end{proof}

\begin{definition} \label{d:sincere}
We call a path $p$ \emph{sincere} if the
support of its weight is equal to $C_0$, and \emph{insincere} otherwise.
\end{definition}

\begin{proposition} \label{p:path}
Let $I\neq J\in Q_0(D)$ and set $a=a(I,J)$, $b=b(I,J)$.
Then there is an insincere path from $I$ to $J$
in $Q(D)$.
\end{proposition}

\begin{proof}
By Corollary~\ref{c:simpletest}, we have $I\leq_{(a,b)_0} J$.
By Proposition~\ref{p:legalarrow}, there is a vertex $I_1\in Q_0(D)$
and an arrow $\alpha{\colon}I\to I_1$ such that $\weight_{\alpha}$ does not
contain any element of $(a,b)_0$.
By Lemma~\ref{l:arrowless} we have $I_1\leq_{(a,b)_0} J$.
Let $a_1=a(I_1,J)$ and $b_1=b(I_1,J)$. Then, by Corollary~\ref{c:simpletest},
we have $(a,b)_0\subseteq (a_1,b_1)_0$.
From the proof of Lemma~\ref{l:arrowless} we also see that
$\coker \embed_{JI_1}$ has strictly smaller dimension than
$\coker \embed_{JI}$.

We can now repeat this argument for $I_1$, since $I_1\leq_{(a_1,b_1)_0}J$
again by Corollary~\ref{c:simpletest}. We get
$I_2\leq_{(a_2,b_2)_0} J$, where $a_2=a(I_2,J)$ and $b_2=b(I_2,J)$.
Continuing in this way, we obtain a path
$$I\to I_1\to I_2\to \cdots $$
which satisfies
$$(a,b)_0\subseteq (a_1,b_1)_0\subseteq (a_2,b_2)_0\subseteq \cdots ,$$
and thus that the weight of each arrow in the path avoids $(a,b)_0$.
The path must eventually reach $J$ because of the decreasing dimension of
$\coker \embed_{JI_i}$ as $i$ increases, and we are done.
\end{proof}

\section{Description of paths} 
\label{s:paths}
Before we can prove the main result, we need some more information about $A_D$. By Proposition~\ref{p:path}, there is an insincere
path between any two vertices $I,J$ of $Q(D)$.
In this section we will show that there is a unique
element of $A_D$ which can be written as such a path,
and furthermore that any path from $I$ to $J$ is
equal in $A_D$ to a power of $u$ multiplied by this
path, adapting arguments from~\cite[\S5]{bocklandt12}. We shall use the fact that
$Q(D)$ is a dimer model in a disk
(see Remark~\ref{r:plabicdual}).

We recall the following from~\cite[\S5]{bocklandt12}.
Let $p$ be a path in $Q(D)$.
A subpath $p$ of maximal length with respect to being contained in the boundary of a face
in $Q_2^+$ (respectively, $Q_2^-$) is called a \emph{positive arc} (respectively,
\emph{negative arc}) of $p$.
Then $p$ can be written as a concatenation of its positive arcs or a concatenation
of its negative arcs.
The path $p$ is defined to be \emph{positively reducible}
(respectively, \emph{negatively reducible}) if (at least) one of its positive
(respectively, negative) arcs contains all of the arrows
but one in lying in the boundary of a face in $Q_2^+$ (respectively, in $Q_2^-$).

If $p$ is any path in $Q(D)$, we denote by
$p^{-1}$ the path in the underlying unoriented graph of $Q(D)$ obtained by reading $p$ in reverse.
We say that a cycle in the underlying
unoriented graph \emph{does not self-intersect}
if it does not visit the same vertex twice
(except the start and end points).
Note that this includes the case of a cycle
which goes along an arrow in $Q(D)$
and then returns along the same arrow.

Suppose that $p$ and $q$ are distinct paths in $Q(D)$ with the same start and end points, such that $q^{-1}p$ (regarded as path in the
underlying unoriented graph) is a clockwise cycle which does not self-intersect.
Then we say that the pair $p,q$
is \emph{reducible} if either $p$ is negatively reducible or $q$ is positively reducible. We make a similar definition if $q^{-1}p$ is an anticlockwise cycle.

\begin{proposition} \label{p:bocklandtreducibility}
Let $p,q$ be distinct paths in $Q(D)$,
with the same start and end points, 
such that $q^{-1}p$ does not self-intersect. Then the pair $p,q$ is reducible.
\end{proposition}

\begin{proof}
Let $D_0=R_{pq}$ be the interior of the region enclosed by the cycle $q^{-1}p$, which is a disk
by the assumption on $p$ and $q$.
We first show that either $p$ or $q$ has a \emph{backwards arrow} in the interior of
$R_{pq}$, that is, an arrow starting at some vertex on the path which ends at an
earlier vertex of the path
(e.g.\ the thick curved arrows in Figure~\ref{f:reducibility1}).

Consider the strand $S_1$ which enters $D_0$
through the last arrow $a_0$ of $q$. Then $S_1$
leaves $D_0$ through an arrow $a_1$ of $p$ or $q$.
We assume that $a_1$ lies in $p$, as illustrated in the left hand diagram in Figure~\ref{f:reducibility1};
the argument when $a_1$ lies in $q$ is similar and is illustrated in the right hand diagram
in Figure~\ref{f:reducibility1}.
Since $S_1$ does not intersect itself (Definition~\ref{d:asd}(b1)), it divides $D_0$
into two disks (and also $a_1\neq a_0$). 
Let $D_1$ be the disk on the right of $S_1$ and let $S_2$ be the strand 
that crosses $S_1$ immediately before $a_1$.
Note that $S_2\not=S_1$, again by Definition~\ref{d:asd}(b1),
and that $S_2$ will be leaving $D_1$ at this crossing, by the alternating crossing property (Definition~\ref{d:asd}(a3)).

By Definition~\ref{d:asd}(b2), $S_2$ cannot enter
$D_1$ through the part of $S_1$ before its crossing point
with $S_2$, so it must enter at an arrow
$a_2$ of $p$ after or equal to $a_1$.
The strand $S_2$ divides $D_1$ into two disks and
we take $D_2$ to be the disk on the right of $S_2$. 

In the case $a_2=a_1$, we claim that
no strand can cross $S_2$ between its
two crossing points with $S_1$.
It then follows that the arrow corresponding to the crossing not at $a_1$
is a backwards arrow, with head/tail equal to the tail/head of $a_1$,
and we are done in this case.
To prove the claim, consider the first strand 
crossing $S_2$ after $a_1$. 
If it were not $S_1$, then it would be entering $D_2$.
But it can not leave $D_2$ through $S_2$, by Definition~\ref{d:asd}(b2),
and it can not leave through $S_1$, as the only two available crossing points are already
taken by $S_2$. This contradiction proves the claim.

In the case $a_2\neq a_1$, let $S_3$ be the strand that crosses $S_2$ immediately after $a_2$.
If $S_3=S_1$, then the arrow corresponding to the crossing of $S_1$ and $S_2$
has tail at the head of $a_2$ and head at the tail of $a_1$, so is the required backwards arrow for 
$p$.
If $S_3\neq S_1$, then, as in the previous case, $S_3$ cannot leave $D_2$
by crossing $S_2$ or $S_1$, so it must leave through an arrow
$a_3$ of $p$ lying strictly after $a_1$ and before or equal to $a_2$.
If $a_3=a_2$, then we get a backwards arrow as in the case $a_2=a_1$.
Otherwise $a_3$ lies strictly between $a_2$ and $a_1$.

We may continue in this way to define successive strands $S_i$ entering or leaving disks $D_{i-1}$
at arrows $a_i$ until, by the finiteness of $p$, we necessarily reach a point when $S_i=S_{i-2}$
or $a_i=a_{i-1}$
and we obtain a backwards arrow as above.

Thus $p$ or $q$ has a backwards arrow.
For the second part of the proof, we show
that, if $p$ has a backwards arrow $\beta$, then $p$ is negatively reducible.
A similar argument shows that if $q$ has a backwards arrow, then $q$ is positively reducible.
As there are only finitely many backwards arrows for $p$, we may further
assume that $\beta$ is `minimal' in the sense that there is no other backwards arrow 
in the disk bounded by $\beta$ and the part $p'$ of $p$ between the endpoints of $\beta$. 
Let $q'=p_{\beta}^-$ (as in Definition~\ref{d:dimeralgebra}) and observe that our goal is achieved by showing that the paths $p'$ and $q'$ coincide.

So suppose, for contradiction, that $p'$ and $q'$ do not coincide.
If $p'$ and $q'$ intersect, we replace
them with subpaths which start and
end at the same vertices but don't intersect.
As in the first part of the proof, let $S_1$ be the strand entering 
the region $R_{p'q'}$ at the last arrow $a_0$ of $q'$. 
If $S_1$ were to leave $R_{p'q'}$ at an arrow of $q'$,
then, since $q'$ is a part of $p_{\beta}^-$, it must 
re-enter $R_{p'q'}$ at the following arrow of $q'$
and, after doing this possibly several times (see Figure~\ref{f:reducibility2}),
it would have to cross itself (or be a loop), contradicting Definition~\ref{d:asd}(b1).
On the other hand, if $S_1$ were to leave $R_{p'q'}$ at an arrow of $p'$,
then we can argue as in the first part to show that $p'$ has a backwards arrow in $R_{p'q'}$,
contradicting the minimality of $\beta$ and completing the proof.
\end{proof}

\begin{figure}
\[
\begin{tikzpicture}[scale=1.35,baseline=(bb.base),
  strand/.style={black,dashed,thick},
  quivarrow/.style={black, -latex}]
 \newcommand{\strarrow}{\arrow{angle 60}}
\path (0,0) node (bb) {}; 

\foreach \n/\x/\y in {1/0/0, 2/1.2/1, 3/2/2, 4/2.2/3, 5/1.75/4, 6/1/4.75, 7/0/5}
{ \draw (\x,\y) node (p\n) {\tiny$\bullet$}; }
\foreach \t/\h/\a in {1/2/0, 2/3/0, 3/4/0, 4/5/0, 5/6/0, 6/7/0}
{\draw [quivarrow]  (p\t) edge [bend left=\a] (p\h); }

\foreach \n/\x/\y in {1/0/0, 2/-1/0.5, 3/-2/1.5, 4/-2/3, 5/-1/4.5, 6/0/5}
{ \draw (\x,\y) node (q\n) {\tiny$\bullet$}; }
\foreach \t/\h/\a in {1/2/0, 2/3/0, 3/4/0, 4/5/0, 5/6/0}
{\draw [quivarrow]  (q\t) edge [bend left=\a] (q\h); }
 
\draw (0.6,5.1) node {\small $a_0$};
\draw (-1.7,0.9) node {\small $a_1$};
\draw (-1.75,3.85) node {\small $a_2$};
\draw (-2.25,2.25) node {\small $a_3$};

\draw (0.0,3.4) node {\small $S_1$};
\draw (-1.0,2.0) node {\small $S_2$};
\draw (-1.6,2.65) node {\small $S_3$};

\foreach \n/\t/\h in {1/6/7, 2/1/2, 3/4/5, 4/2/3}
\coordinate (s\n) at ($0.5*(p\t) + 0.5*(p\h)$);
 
\draw [strand] plot[smooth] 
  coordinates {(s1) (0.4,4.7) (0,4) (-0.5,3) (-0.25,1.75) (-1.3,1.2) (-1.5,1)}
 [postaction=decorate, decoration={markings, 
  mark= at position 0.2 with \strarrow, 
  mark= at position 0.55 with \strarrow,
  mark= at position 0.85 with \strarrow}];
  
\draw [strand] plot[smooth] 
  coordinates {(-1.5,3.75) (-1.3,3.5) (-1.3,2.8) (-0.4,1.7) (-0.3,1.45)}
 [postaction=decorate, decoration={markings, 
  mark= at position 0.25 with \strarrow, 
  mark= at position 0.65 with \strarrow}];

\draw [strand] plot[smooth] 
  coordinates {(-1.1,2.8) (-1.4,2.5) (-2,2.25)}
 [postaction=decorate, decoration={markings, 
  mark= at position 0.65 with \strarrow}];

\draw [quivarrow, very thick]  (q5) edge [bend left=21] (q3); 

\draw (-1.5,0.3) node {$p$};
\draw (1.5,0.3) node {$q$};
\end{tikzpicture}
\qquad
\begin{tikzpicture}[scale=1.35,baseline=(bb.base),
  strand/.style={black,dashed, thick},
  quivarrow/.style={black, -latex}]
 \newcommand{\strarrow}{\arrow{angle 60}}
\path (0,0) node (bb) {}; 

\foreach \n/\x/\y in {1/0/0, 2/1.2/1, 3/2/2, 4/2.2/3, 5/1.75/4, 6/1/4.75, 7/0/5}
{ \draw (\x,\y) node (p\n) {\tiny$\bullet$}; }
\foreach \t/\h/\a in {1/2/0, 2/3/0, 3/4/0, 4/5/0, 5/6/0, 6/7/0}
{\draw [quivarrow]  (p\t) edge [bend left=\a] (p\h); }

\foreach \n/\x/\y in {1/0/0, 2/-1/0.5, 3/-2/1.5, 4/-2/3, 5/-1/4.5, 6/0/5}
{ \draw (\x,\y) node (q\n) {\tiny$\bullet$}; }
\foreach \t/\h/\a in {1/2/0, 2/3/0, 3/4/0, 4/5/0, 5/6/0}
{\draw [quivarrow]  (q\t) edge [bend left=\a] (q\h); }

\draw (0.6,5.1) node {\small $a_0$};
\draw (0.8,0.3) node {\small $a_1$};
\draw (2.3,3.5) node {\small $a_2$};
\draw (1.85,1.4) node {\small $a_3$};

\draw (-0.75,3.0) node {\small $S_1$};
\draw (0.5,2.52) node {\small $S_2$};
\draw (1.6,2.2) node {\small $S_3$};

\foreach \n/\t/\h in {1/6/7, 2/1/2, 3/4/5, 4/2/3}
\coordinate (s\n) at ($0.5*(p\t) + 0.5*(p\h)$);
 
\draw [strand] plot[smooth] 
  coordinates {(s1) (0.4,4.7) (0,4) (-0.5,3) (0.25,1.75) (0.4,0.8) (s2)}
 [postaction=decorate, decoration={markings, 
  mark= at position 0.2 with \strarrow, 
  mark= at position 0.41 with \strarrow,
  mark= at position 0.67 with \strarrow,
  mark= at position 0.85 with \strarrow}];
  
\draw [strand] plot[smooth] 
  coordinates {(2,3.4) (1.8,3.3) (1.1,2.8) (0.3,1.7) (0.0,1.5)}
 [postaction=decorate, decoration={markings, 
  mark= at position 0.25 with \strarrow, 
  mark= at position 0.65 with \strarrow}];

\draw [strand] plot[smooth] 
  coordinates {(0.95,2.95) (1.2,2.5) (1.4,1.7) (1.6,1.5)}
 [postaction=decorate, decoration={markings, 
  mark= at position 0.6 with \strarrow}];

\draw [quivarrow, very thick]  (p5) edge [bend right=27] (p2); 

\draw (-1.5,0.3) node {$p$};
\draw (1.5,0.3) node {$q$};

\end{tikzpicture}
\]
\caption{Proof of Proposition~\ref{p:bocklandtreducibility}, first part}
\label{f:reducibility1}
\end{figure}
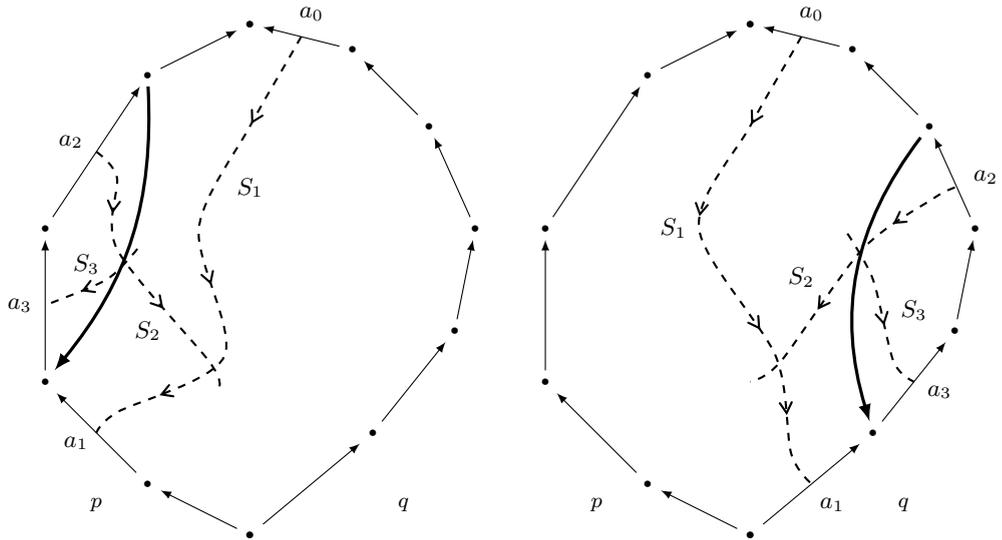

\begin{figure}
\[
\begin{tikzpicture}[scale=0.25,baseline=(bb.base),
  strand/.style={black,dashed, thick},
  quivarrow/.style={black, -latex}]
 \newcommand{\strarrow}{\arrow{angle 60}}
\path (0,0) node (bb) {}; 

\foreach \n/\x/\y in {0/0/0, 1/4/2, 2/7/4, 3/10/8, 4/11/11, 5/10/15, 6/8/18, 7/4/22, 8/0/24}
{ \draw (\x,\y) node (p\n) {\tiny$\bullet$}; }

\foreach \n/\x/\y in {0/0/0, 1/-4/1, 2/-7/5, 4/-8/14, 5/-6/19, 6/-3/22, 7/0/24}
{ \draw (\x,\y) node (q\n) {\tiny$\bullet$}; }

\foreach \t/\h/\a in {0/1/2, 1/2/0, 2/3/0, 3/4/0, 4/5/0, 5/6/0, 6/7/0, 7/8/0}
{\draw [quivarrow] (p\t) edge [bend left=\a] (p\h); }

\foreach \t/\h/\a in {0/1/2, 1/2/0, 2/4/0, 4/5/0, 5/6/0, 6/7/0}
{\draw [quivarrow]  (q\t) edge [bend left=\a] (q\h); }

\foreach \t/\h in {0/1, 1/2, 2/3, 3/4, 4/5, 5/6, 6/7, 7/8}
\coordinate (pm\t) at ($0.5*(p\t) + 0.5*(p\h)$);

\draw [strand] plot[smooth] 
  coordinates {(pm7) (0.5,20.5) (-1,16) (-1,12) (1,7) 
  (pm1) (8,2.5) (9.2,3.7) (pm2) (6.5,8) (6,9.5) (6.5,11)
  (pm4) (12,14) (12.3,15.5) (11.5,16.5) (pm5) (2,15.5) (-3,15.5)}
 [postaction=decorate, decoration={markings, 
 mark= at position 0.09 with \strarrow,
 mark= at position 0.24 with \strarrow, 
 mark= at position 0.35 with \strarrow,
 mark= at position 0.45 with \strarrow,
 mark= at position 0.58 with \strarrow,  
 mark= at position 0.73 with \strarrow,
 mark= at position 0.9 with \strarrow}];  

\draw (-10,7) node {$p'$};
\draw (12,7) node {$q'$};
\draw (3,23.7) node {\small $a_0$};
\draw (1.2,10) node {\small $S_1$};

\end{tikzpicture}
\]
\caption{Proof of Proposition~\ref{p:bocklandtreducibility}, second part}
\label{f:reducibility2}
\end{figure}
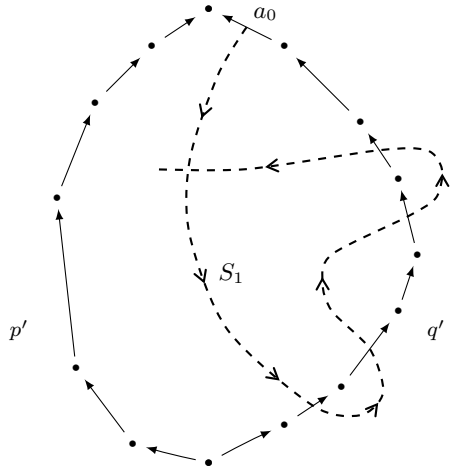

Suppose that $I$ and $J$ are vertices in
$Q_0(D)$, that $p,q$ are paths from $I$ to $J$
in $Q(D)$ and that $q^{-1}p$ does not self-intersect (except at its endpoints). We will denote by $\overline{R_{pq}}$ the region between $p$ and $q$, including $p$ and $q$. If $p=q=\gamma$ for some arrow $\gamma$ then $\overline{R_{pq}}=\gamma$;
otherwise it is a disk.
We shall call a non-trivial path $s$
in $\overline{R_{pq}}$ which
intersects $q^{-1}p$ only at its end points
a \emph{chord} in $\overline{R_{pq}}$.

If, furthermore, the end points of $s$ both lie on $p$ and $s$ is oriented in the same
direction as (respectively, in the opposite
direction to) $p$, we shall call $s$ a forwards
(respectively, backwards) $p$-chord.
In particular, we allow an arrow of $p$ to be a forwards $p$-chord.
We make a similar definition for $q$.

We also define the \emph{area} of a region
bounded by a cycle in the underlying unoriented graph of $Q(D)$ which does not self-intersect to be the number of faces in $Q_2(D)$ contained in it.

\begin{lemma} \label{l:nointersection}
Let $I,J$ be vertices in $Q_0(D)$ and
$p,q$ be paths from $I$ to $J$ in
$Q(D)$ such that $q^{-1}p$ does not
self-intersect.
Then there is a path 
$\rr$ in $\overline{R_{pq}}$ such that, in $A_D$, $p=u^{N_p}\rr$ and
$q=u^{N_q}\rr$ for some nonnegative integers $N_p$ and $N_q$.
\end{lemma}
\begin{proof}
Note that we include the case where $p=q=\gamma$ for some arrow $\gamma$.
We first observe that if the result holds in a case where $q=e_I$ is trivial, we must have $r=e_I$ and so $p=u^{N_p}e_I$; similarly if $p$ is trivial. To prove the lemma, we argue by induction on the area of $\overline{R_{pq}}$. If this is zero, then $p=q=\gamma$ for some arrow $\gamma$, and the result is trivial.
We suppose that $\overline{R_{pq}}$ has non-zero area, and assume that the result holds in the case where it has smaller area.

Without loss of generality, we
may assume $q^{-1}p$ is a clockwise cycle.
By Proposition~\ref{p:bocklandtreducibility},
either $p$ is negatively reducible or
$q$ is positively reducible. We suppose
that $p$ is negatively reducible (a
similar argument holds in the case where
$q$ is positively reducible).
Then $p$ contains a subpath of the form $p_{\alpha}^-$ for some backwards arrow $\alpha:L\rightarrow K$ for $p$ in
$\overline{R_{pq}}$.
We can apply the relation
$p_{\alpha}^-=p_{\alpha}^+$ in $A_D$
to $p$ to produce a new path $p'$ from $I$ to $J$ in $\overline{R_{pq}}$. Note that we have $p=p'$ in $A_D$.

Since $p_{\alpha}^-\alpha$ and $p_{\alpha}^+\alpha$ are faces of $Q(D)$, and $Q(D)$ is a dimer model, there cannot be any arrows inside the disks they bound. It follows that $p_{\alpha}^+$ (and hence also $p'$) is contained in $\overline{R_{pq}}$.
Note that $p_{\alpha}^+$ can be written uniquely as
a composition of chords. We will apply
the inductive hypothesis to the parts of
$\overline{R_{pq}}$ between each chord and
$q^{-1}p$, which we can do since they each have
area smaller than that of $\overline{R_{pq}}$.
Several types of chord 
can appear in the composition and
we need to deal with each type in a slightly different way,
which we now detail.

Let $p[K,I]$ be the subpath of $p$ from
$I$ to $K$ and $p[J,L]$ the subpath of $p$
from $L$ to $J$,
so that $p'=p[J,L]p_{\alpha}^+p[K,I]$.
It is helpful to distinguish
two cases for the path
$p_{\alpha}^+$ from $K$ to $L$:

\begin{enumerate}
\item[(a)] The path $p_{\alpha}^+$ is
of the form $p_5p_4p_3p_2p_1$,
where $p_1$ and $p_5$ are compositions of backwards $p$-chords,
$p_2$ is either an idempotent
(if $I$ lies on $p_{\alpha}^+$)
or a chord from a vertex
on $p[K,I]$ to a vertex on $q$, $p_3$ is a composition of forwards $q$-chords and
$p_4$ is either an idempotent 
(if $J$ lies on $p_{\alpha}^+$)
or a chord from a vertex on $q$ to a vertex on $p[J,L]$.
\item[(b)] The path $p_{\alpha}^+$ is
of the form $p_3p_2p_1$, where
$p_1$ and $p_3$ are compositions of backwards
$p$-chords and $p_2$ is a forwards $p$-chord.
\end{enumerate}
We also allow $p_1,p_3$ and $p_5$ in (a) and
$p_1$ and $p_3$ in (b) to be idempotents, considered as empty compositions.

See Figure~\ref{f:pairs} for examples.
We reiterate that it is possible for forwards
$q$-chords in $p_3$ to be arrows in
$q$, as in case (a) in
Figure~\ref{f:pairs}.

In case (a), each backwards $p$-chord
in $p_1$ and $p_5$ corresponds to a loop in $p'$ whose area is less than that of $\overline{R_{pq}}$: by the inductive hypothesis and the remark at the start of the proof, each such loop is equal to a power of $u$ times an idempotent.
We see that $p'$ is the product of
a power of $u$ and $s_2p_4p_3p_2s_1$, where $s_1$ is an initial
subpath of $p[K,I]$ and $s_2$ is a final subpath of $p[J,L]$. But each forwards $q$-chord
in $p_3$, together with the subpath of $q$ with the same start and end vertices, form
a pair of paths satisfying the assumptions
in the lemma. The same applies to $p_2s_1$ together with an initial subpath of $q$, and
to $s_2p_4$ together with a final subpath
of $q$. Applying the inductive hypothesis
to all of these pairs gives the result
for $p',q$, and hence for $p,q$.

Similarly, in case (b), each backwards $p$-chord in $p_1$ and $p_3$ corresponds to a loop in $p'$ which, by induction, is equal to a power of
$u$ times an idempotent. We see that $p'$ is
the product of a power of $u$ and $s_2p_2s_1$,
where $s_1$ is an initial subpath of $p[K,I]$
and $s_2$ is a final subpath of $p[J,L]$. 
We may then apply the inductive hypothesis
to the pair $s_2p_2s_1,q$ to get the
result for $p,q$.
\end{proof}

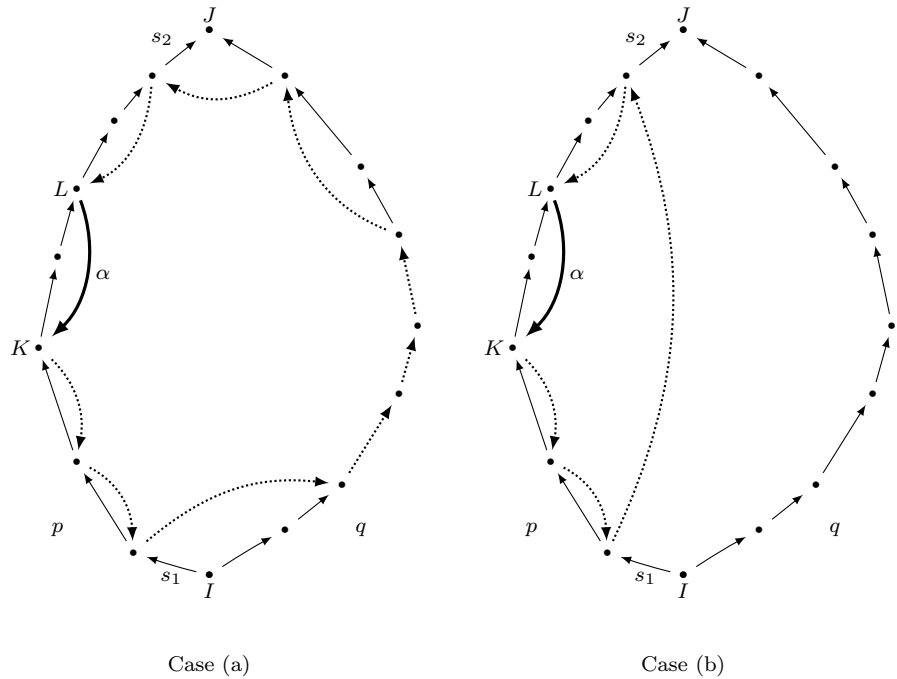
\begin{figure}
\[
\begin{tikzpicture}[xscale=0.25,yscale=0.3,
  quivarrow/.style={black, -latex},
  chord/.style={black,densely dotted, thick,-latex} ]

\foreach \n/\x/\y in {0/0/0, 1/4/2, 2/7/4, 3/10/8, 4/11/11, 5/10/15, 6/8/18, 7/4/22, 8/0/24}
{ \draw (\x,\y) node (p\n) {\tiny$\bullet$}; }

\foreach \n/\x/\y in {0/0/0, 1/-4/1, 2/-7/5, 3/-9/10, 4/-8/14, 5/-7/17, 6/-5/20, 7/-3/22, 8/0/24}
{ \draw (\x,\y) node (q\n) {\tiny$\bullet$}; }

\foreach \t/\h/\a in {0/1/2, 1/2/0, 5/6/0, 6/7/0, 7/8/0}
{\draw [quivarrow] (p\t) edge [bend left=\a] (p\h); }

\foreach \t/\h/\a in {2/3/0, 3/4/0, 4/5/0}
{\draw [chord] (p\t) edge [bend left=\a] (p\h); }

\foreach \t/\h/\a in {0/1/2, 1/2/0, 2/3/0, 3/4/0, 4/5/0,
5/6/0, 6/7/0, 7/8/0}
{\draw [quivarrow]  (q\t) edge [bend left=\a] (q\h); }

\foreach \t/\h in {0/1, 1/2, 2/3, 3/4, 4/5, 5/6, 6/7, 7/8}
\coordinate (pm\t) at ($0.5*(p\t) + 0.5*(p\h)$);

\foreach \t/\h in {0/1, 1/2, 2/3, 3/4, 4/5, 5/6, 6/7, 7/8}
\coordinate (qm\t) at ($0.5*(q\t) + 0.5*(q\h)$);

\draw (p0) node [below] {$I$};
\draw (p8) node [above] {$J$};
\draw (q3) node [left] {$K$};
\draw (q5) node [left] {$L$};

\draw (qm0) node [below] {$s_1$};
\draw (qm7) node [above left] {$s_2$};
\draw (-8,2) node {$p$};
\draw (8,2) node {$q$};
\draw (p0)++(0,-4) node {Case (a)};

\draw [quivarrow, very thick]  (q5) edge [bend left=38] node [right] {$\alpha$} (q3); 

\draw [chord]  (q3) edge [bend left=32] (q2); 
\draw [chord]  (q2) edge [bend left=32] (q1); 
\draw [chord]  (q1) edge [bend left=20] (p2); 
\draw [chord]  (p5) edge [bend left=32] (p7); 
\draw [chord]  (p7) edge [bend left=25] (q7); 
\draw [chord]  (q7) edge [bend left=32] (q5); 


\begin{scope}[shift={(25,0)}]

\foreach \n/\x/\y in {0/0/0, 1/4/2, 2/7/4, 3/10/8, 4/11/11, 5/10/15, 6/8/18, 7/4/22, 8/0/24}
{ \draw (\x,\y) node (p\n) {\tiny$\bullet$}; }

\foreach \n/\x/\y in {0/0/0, 1/-4/1, 2/-7/5, 3/-9/10, 4/-8/14, 5/-7/17, 6/-5/20, 7/-3/22, 8/0/24}
{ \draw (\x,\y) node (q\n) {\tiny$\bullet$}; }

\foreach \t/\h/\a in {0/1/2, 1/2/0, 5/6/0, 6/7/0, 7/8/0}
{\draw [quivarrow] (p\t) edge [bend left=\a] (p\h); }

\foreach \t/\h/\a in {2/3/0, 3/4/0, 4/5/0}
{\draw [quivarrow] (p\t) edge [bend left=\a] (p\h); }

\foreach \t/\h/\a in {0/1/2, 1/2/0, 2/3/0, 3/4/0, 4/5/0,
5/6/0, 6/7/0, 7/8/0}
{\draw [quivarrow]  (q\t) edge [bend left=\a] (q\h); }

\foreach \t/\h in {0/1, 1/2, 2/3, 3/4, 4/5, 5/6, 6/7, 7/8}
\coordinate (pm\t) at ($0.5*(p\t) + 0.5*(p\h)$);

\foreach \t/\h in {0/1, 1/2, 2/3, 3/4, 4/5, 5/6, 6/7, 7/8}
\coordinate (qm\t) at ($0.5*(q\t) + 0.5*(q\h)$);

\draw (p0) node [below] {$I$};
\draw (p8) node [above] {$J$};
\draw (q3) node [left] {$K$};
\draw (q5) node [left] {$L$};

\draw (qm0) node [below] {$s_1$};
\draw (qm7) node [above left] {$s_2$};
\draw (-8,2) node {$p$};
\draw (8,2) node {$q$};
\draw (p0)++(0,-4) node {Case (b)};

\draw [quivarrow, very thick]  (q5) edge [bend left=38] node [right] {$\alpha$} (q3); 

\draw [chord]  (q3) edge [bend left=32] (q2); 
\draw [chord]  (q2) edge [bend left=32] (q1); 
\draw [chord]  (q1) edge [bend right=27] (q7); 
\draw [chord]  (q7) edge [bend left=32] (q5); 
\end{scope}
\end{tikzpicture}
\]
\caption{Examples for the proof of Lemma~\ref{l:nointersection}. 
Dotted arrows in $\overline{R_{pq}}$ represent 
chords
and $p_{\alpha}^+$ is the composition of all of them}
\label{f:pairs}
\end{figure}

\begin{proposition} \label{p:transverse}
Let $p,q$ be arbitrary paths from $I$ to $J$ in $Q(D)$.
Then there is a path $\rr$ in $Q(D)$ such that, in
$A_D$,
$p=u^{N_p}\rr$ and $q=u^{N_q}\rr$ for some nonnegative integers $N_p$ and $N_q$.
\end{proposition}
\begin{proof}
We will prove the result by induction on the 
length of $q^{-1}p$. If the length of 
$q^{-1}p$ is zero, then $p=q=e_I$ and the result
holds with $r=e_I$. We assume that the length
of $q^{-1}p$ is positive and that the result
holds when it has shorter length.

We write $p$ as:
$$I_0\rightarrow I_1\rightarrow \cdots \rightarrow I_i$$
and $q$ as:
$$I_j\rightarrow I_{j-1}\rightarrow \cdots \rightarrow I_i,$$
where $j\geq i$.
Thus, $I_0,\ldots ,I_j$ are the vertices of $Q(D)$
visited by $q^{-1}p$, in order, and $I_j=I_0$.
Let $m$ be minimal such that
$I_m=I_l$ for some $l<m$. Suppose first
that $m\leq i$.
Then, by Lemma~\ref{l:nointersection}, the part of $p$
between $I_l$ and $I_m$ is equal
to a power of $u$ multiplied by $e_{I_l}$.
The result follows in this case by
applying the inductive hypothesis to the
pair $p,q$ with the part of $p$
between $I_l$ and $I_m$ removed.
A similar argument applies in the case where
$l\geq i$.

The remaining possibility is that $l<i$ and $m>i$.
But then we can write $p=p_2p_1$ and
$q=q_2q_1$ as compositions of paths,
such that
$(q_2)^{-1}p_2$ is the part of $q^{-1}p$
between $I_l$ and $I_m$ and meets itself
only at its starting and ending points.
Note that the case $p_2=q_2=\gamma$ for some arrow $\gamma$ may occur here.
The result for $(p,q)$ then follows by applying
Lemma~\ref{l:nointersection} to the pair $(p_2,q_2)$ and the induction hypothesis to the pair $(p_1,q_1)$.
\end{proof}

\begin{corollary} \label{c:describepaths}
Let $I,J$ be vertices in $Q(D)$. Then there is a unique
element $p_{JI}$ of $A_D$ which can be written as an
insincere path in $Q(D)$ from $I$ to $J$. 
Furthermore, the elements
$$\{u^Np_{JI}\,:\,N\geq 0\}$$
form a basis of $e_JA_De_I$.
\end{corollary}
\begin{proof}
By Proposition~\ref{p:path}, there is an insincere
path $p$ in $Q(D)$ from $I$ to $J$ (taking $p=e_I$
if $I=J$). Suppose that $q$ is another such path and
that $p=q$ in $A_D$. Since $p,q$ are insincere
and the support of $u$ is $C_0$, it follows from
Proposition~\ref{p:transverse} that $p=q$ in $A_D$.
We denote the common element of $A_D$ arising from
an insincere path in $Q(D)$ from $I$ to $J$ by
$p_{JI}$.

If $p$ is any path in $Q(D)$ from $I$ to $J$, then by
Proposition~\ref{p:transverse} we must have $p=u^Np_{JI}$ in $A_D$ for some nonnegative integer $N$, since $p_{JI}$ is insincere.
Furthermore, the elements $u^Np_{JI}$ are
non-zero in $A_D$ by Remark~\ref{r:commutation}(a) and independent (using Remark~\ref{r:ADgraded})
since they have distinct weights.
The result follows.
\end{proof}

\section{Isomorphism of algebras} 
\label{s:isomorphism}

Recall that $B$ denotes the (polynomial case) algebra introduced
in~\cite{jks}, as defined in Section~\ref{s:rankone}, and that we have associated a $B$-module $\module_I$ to each $k$-subset $I$ of $C_1$
(see Definition~\ref{d:moduleMI}). To any Postnikov diagram $D$, we may associate a $B$-module
\begin{equation} \label{e:TD}
  T_D=\bigoplus_{I\in Q_0(D)} \module_I,
\end{equation}
where $Q_0(D)=\C(D)$ is the set of labels on the alternating regions of
$D$ and thus the set of vertices of the dimer model $Q(D)$.

\begin{remark} \label{l:uniquemorphism}
Since the degree of a product of
monomial morphisms is the sum of their
degrees, the degree map induces an $\mathbb{N}C_0$-grading on
$\End_B(T_D)$; see Remark~\ref{l:homLILJ}.
\end{remark}

Our goal in this section is to show that the dimer algebra $A_D=A_{Q(D)}$, as in
Definition~\ref{d:dimeralgebra},
is isomorphic to $\End_B(T_D)$ as an $\mathbb{N}C_0$-graded algebra,
and, as a corollary, that the idempotent subalgebra
of $A_D$ corresponding to the boundary vertices of $Q(D)$ is isomorphic to
$B^{\opp}$.
We start by defining a homomorphism from $A_D$ to $\End_B(T_D)$.

\begin{lemma} \label{l:defembed}
There is a homomorphism $\embed{\colon}A_D\to \End_B(T_D)$ of
$\mathbb{N}C_0$-graded algebras
determined (uniquely) by the following properties:
\begin{enumerate}[(a)]
\item If $I\in Q_0(D)$, then $g(e_I)=\id_{\module_I}$.
\item If $\alpha{\colon}I\to J$ is an arrow in $Q_1(D)$, then $\embed(\alpha)=\embed_{JI}$.
\end{enumerate}
\end{lemma}

\begin{proof}
As $A_D$ is a quotient of the path algebra of $Q(D)$, such a homomorphism
is certainly uniquely determined by (a) and (b) and what we must
check is that the morphisms $\embed(\alpha)$,
for all $\alpha\in Q_1(D)$, satisfy the defining relations of $A_D$,
namely \eqref{e:definingrelations}.
So, suppose that $\alpha\colon I\to J$ is an internal arrow. Then
$\alpha$ lies in the boundary of a face $F^+\in Q_2^+$. This boundary is a cycle
$$I\to J\to I_1 \to \cdots \to I_m \to I.$$
Then $g(p_\alpha^+)=\embed_{I I_{m}}\cdots\embed_{I_1J}$,
which is insincere by
Corollaries~\ref{c:weightequalsdegree} and~\ref{c:cycleweight}.
Since this morphism is monomial, we must have $g(p_\alpha^+)=\embed_{IJ}$, by
Lemma~\ref{l:homLILJ}.
The same argument applies to $g(p_\alpha^-)$ and so it is equal
to $\embed (p_\alpha^+)$, as required.
The fact that $\embed$ is a homomorphism
of $\mathbb{N}C_0$-graded algebras follows from Corollary~\ref{c:weightequalsdegree}.
\end{proof}

We can now prove our main result.

\begin{theorem} \label{t:isomorphism}
Let $D$ be an arbitrary $(k,n)$-Postnikov diagram.
Let $A_D$ be the associated dimer algebra and $T_D$ the associated $B$-module,
as in \eqref{e:TD}.
Then the map $\embed{\colon}A_D\to \End_B(T_D)$, as in Lemma~\ref{l:defembed},
is an isomorphism of graded algebras.
\end{theorem}
\begin{proof}
Fix $I,J\in Q_0(D)$ and consider the minimal codimension map
$\embed_{JI}{\colon}\module_I\to \module_J$ defined in~\eqref{e:embedJI}.
Let $p_{JI}$ be the element of $e_JA_De_I$ from
Corollary~\ref{c:describepaths}.
Its image under $g$ is a monomial
morphism, since the composition of
monomial morphisms is monomial.
As $g$ is a morphism of graded algebras
(Lemma~\ref{l:defembed}), the degree
of $g(u^mp_{JI})$ coincides with the
degree of $t^mg_{JI}$.
Hence, by Lemma~\ref{l:homLILJ}, we have:
\begin{equation} \label{e:surjectivity}
g(u^Np_{JI})=t^Ng_{JI}
\end{equation}
for all $N\geq 0$.

By Corollary~\ref{c:describepaths}, the 
set
$\{ u^Np_{JI} : N\geq 0 \}$ is a basis of $e_JA_De_I$.
On the other hand, by Lemma~\ref{l:homLILJ}, the set
$\{ t^Ng_{JI} : N\geq 0 \}$ is a basis for
\[
 \Hom_B(\module_I,\module_J) = g(e_J) \End_B(T_D) g(e_I).
\]
Hence, by~\eqref{e:surjectivity},
$\embed$ maps a basis of $A_D$ to a basis of $\End_B(T_D)$, so $\embed$ is an isomorphism.
\end{proof}

Let $e$ be the sum, in $A_D$, of the idempotents $e_I$
corresponding to the boundary vertices $I$ of the Postnikov diagram $D$.
We call the algebra $eA_De$ the \emph{boundary algebra} of $D$.

\begin{corollary} \label{c:obtainJKS}
The boundary algebra $eA_De$ is isomorphic to $B^{\opp}$,
i.e.\ the opposite of the algebra $B$ in Section~\ref{s:rankone}.
In particular, it is independent of the choice of Postnikov diagram $D$, up to isomorphism.
\end{corollary}

\begin{proof}
Using the isomorphism in Theorem~\ref{t:isomorphism}, we see that
\[
  eA_De\cong \embed(e)\End_B(T_D)\embed(e) = \End_B(P),
\]
where $P$ is the direct sum of the modules $\module_I$ corresponding to
the labels $I$ of the boundary regions of $D$.
These are exactly the projective $B$-modules (see Remarks~\ref{r:boundarylabels} and~\ref{r:projectives}),
so $eAe$ is isomorphic to $B^{\opp}$.
\end{proof}

\begin{remark}
\label{r:idempotents}
If $D,D'$ are any two $(k,n)$-Postnikov diagrams, 
then, by Corollary~\ref{c:obtainJKS}, 
we have isomorphisms $eA_De\cong B^{\opp}$ and $B^{\opp}\cong eA_{D'}e$. 
It follows from the proof that the isomorphism obtained by composing
these two isomorphisms takes $e_I\in eA_De$ to $e_I\in eA_{D'}e$,
for every boundary label $I$.
\end{remark}

\begin{remark}
Using the map from $\mathbb{N}C_0$ to
$\mathbb{N}$ mapping an element of
$\mathbb{N}C_0$ to the sum of its coefficients, we get an $\mathbb{N}$-grading on $A_D$ and hence on $eA_De$.
The algebra $B$ (and hence also $B^{\opp}$)
has a natural grading in which $x$ has degree $n-k$ and $y$ has degree $k$.

Furthermore, since $g(p_{\newL_{j+1}\newL_j})=g_{\newL_{j+1}\newL_j}$,
the isomorphism in
Corollary~\ref{c:obtainJKS} takes
$p_{\newL_{j+1}\newL_j}$ to $y_{j+1-k}$,
regarded as an element of $B^{\opp}$.
If $p_{\newL_{j+1}\newL_j}$ is an arrow,
it has weight $\newL_j$ (since
strand $j+1$ starts and
strand $j+1-k$ ends on this arrow).
If it is a path, then there is an arrow
from $\newL_j$ to $\newL_{j+1}$ of weight
$C_0\setminus \newL_j$ and it again
follows that the weight of $p_{\newL_{j+1}\newL_j}$
is $\newL_j$. Hence in the $\mathbb{N}$-grading, $p_{\newL_{j+1}\newL_j}$ has
degree $|\newL_j|=k$, which is the same
as the degree of $y_{j+1-k}$ in $B^{\opp}$.
A similar argument shows that
$p_{\newL_j\newL_{j+1}}$ maps to $x_{j+1-k}$,
both elements of degree $n-k$.
We see that the isomorphism in Corollary~\ref{c:obtainJKS} preserves the
$\mathbb{N}$-grading.
\end{remark}

\begin{remark} \label{r:obtainMI}
Consider a Postnikov diagram $D$ with a face labelled $I$.
Then we can regard $e_I A_D e$ as
a right $eA_De$ module and hence, by
Corollary~\ref{c:obtainJKS}, as a left $B$-module. The isomorphism $g$ in Theorem~\ref{t:isomorphism} induces an isomorphism
between $e_I A_D e$ and $g(e_I)\End_B(T_D)g(e)$ as left $B$-modules.
But
$$g(e_I) \End_B(T_D) g(e) = \id_{\module_I} \End_B(T_D) \id_B = \Hom_B(B,\module_I),$$
which is isomorphic to $\module_I$ as a left
$B$-module.
We thus obtain that $e_I A_D e$
is isomorphic to $\module_I$ as a left $B$-module.
It follows that $T_D$ is isomorphic to
$A_De$ as a left $B$-module.
\end{remark}

\begin{remark}
We can also obtain that the algebra $A_D$ is a cancellation algebra.
To see this, we must show that if
$\alpha p=\alpha q$ (or $p\alpha=q\alpha$) in $A_D$ for paths $p,q$ in
$Q(D)$ and an arrow $\alpha$ in $Q_1(D)$, then $p=q$.
If $\alpha p=\alpha q$ then let $F$ be a face whose boundary contains $\alpha$.
Let $r$ be a path be such that
$\bdry F=r\alpha$.
We get $r\alpha p=r\alpha q$, so $up=uq$.
Hence it is enough to show that $up=uq$ (or $pu=qu$) implies $p=q$. We focus on the case $up=uq$; the other case is similar.

Note that we can assume that $p,q$ start at the same vertex, say $I$, and end at the same vertex, say $J$. By Corollary~\ref{c:describepaths}, there are nonnegative integers $N_p,N_q$
such that $p=u^{N_p}p_{JI}$ and $q=u^{N_q}p_{JI}$.
Since $up=uq$, we have $u^{N_p+1}p_{JI}=u^{N_q+1}p_{JI}$.
The elements on each side of this equation
are non-zero in $A_D$ by Remark~\ref{r:commutation}(a). 
Comparing their weights,
we see that $N_p=N_q$ and hence $p=q=u^N p_{JI}$ as required.
\end{remark}

\section{Completion} 
\label{s:completion} 

The definition of the Jacobian algebra associated to a quiver with potential involves
taking a quotient of the completed path algebra (completed with respect to the arrow ideal) by the closure of the ideal generated by
the relations determined by the potential; see~\cite{dwz08}. If the potential lies in the
path algebra, this coincides with the completion of the quotient of the path algebra by
the relations. Thus is is natural to consider the completion $\widehat{A}$ of the total
algebra with respect to the arrow ideal. We also want to relate our results to those in~\cite{jks}. So, in this section, we obtain analogues of
Theorem~\ref{t:isomorphism} and Corollary~\ref{c:obtainJKS} for the completed total algebra.

\begin{lemma} \label{l:idealscontained}
Let $m_A,m_B$ denote the arrow ideals of $A_D$
and $B$, respectively.
Then there is a nonnegative integer $N_A$ such that
$m_A^{N_A}\subseteq (u) \subseteq m_A$
and a nonnegative integer $N_B$ such that
$m_B^{N_B}\subseteq (t)\subseteq m_B$.
\end{lemma}
\begin{proof}
By Corollary~\ref{c:describepaths}, any
insincere path $p$ from $I$ to $J$ in $Q(D)$ must be equal (in $A_D$) to $p_{JI}$.
Hence, in particular, its weight must be equal to
the weight of $p_{JI}$.
The length of $p$ is less than or equal to the
sum of the entries in its weight, so it is
bounded. Allowing $I$ and $J$ to vary, we see
that there is a positive integer $N_A$ such that any insincere path in $Q(D)$ has length
at most $N_A-1$.

Hence any path $p$ of length at least $N_A$ must be sincere and so, by Corollary~\ref{c:describepaths},
we have $p=u^s p_{JI}$ in $A_D$ for some positive integer $s$. This proves the first part, since clearly also $(u) \subseteq
m_A$.

For the second part, note that, in $B$,
$x^{k+1}=y^{n-k}x$ and $y^{n-k+1}=x^ky$, so any path in the quiver of $B$ with at least
$N_B:=\max(k+1,n-k+1)$ steps is equal in $B$ to an element of $(t)$.
\end{proof}

Let $\widehat{A}$ and $\widehat{B}$ be the completions
of $A$ and $B$ with respect to $(u)$ and $(t)$ respectively. By Lemma~\ref{l:idealscontained}, these
completions are isomorphic to the completions with
respect to $m_A$ and $m_B$, respectively.
Similarly, we denote the completion
of a $B$-module $M$ with respect to $(t)$
by $\widehat{M}$; this completion is isomorphic to the completion with respect to $m_B$.

Recall that, as we observed in the introduction,
$\widehat{T}_D$ is a cluster-tilting object in the 
category of Cohen-Macaulay $\widehat{B}$-modules, by~\cite[Rk.~5.5]{jks} and~\cite[Cor.~1]{scott06} 

\begin{theorem} \label{t:completedisomorphism}
Let $D$ be a Postnikov diagram.
Then the isomorphism $g$ in Theorem~\ref{t:isomorphism} induces an isomorphism $\widehat{\embed} \colon \widehat{A} \to \End_{\widehat{B}}(\widehat{T}_D)$.
Let $e$ be the sum in $\widehat{A}_D$ of the idempotents $e_I$ for $I$ a boundary vertex in $D$.
Then $e\widehat{A}_De\cong \widehat{B}^{\opp}$.
\end{theorem}
\begin{proof}
By Lemma~\ref{l:homLILJ} and Theorem~\ref{t:isomorphism}, $A$ is finitely generated as a ${\F}[u]$-module and by~\cite[Cor. 3.4]{jks}, $B$ is finitely generated as a ${\F}[t]$-module.
Since polynomial rings are Noetherian,
the natural maps induce morphisms
of $\F[[u]]$ and $\F[[t]]$-modules as follows:
\begin{equation}
\label{e:completiontensor}
\begin{aligned}
\widehat{A} &\cong A \otimes_{\F[u]} \F[[u]]; \\
\widehat{B} &\cong B \otimes_{\F[t]} \F[[t]]; \\
\widehat{\module}_I &\cong \module_I \otimes_{\F[t]} \F[[t]].
\end{aligned}
\end{equation}
Since $u$ is central
in $A$ the completion of $A$ as a $\mathbb{C}[u]$-module with respect to the ideal $(u)$ in
$\mathbb{C}[u]$ coincides with the completion
of $A$ with respect to the ideal $(u)$ in $A$,
and the first isomorphism above is an isomorphism
of algebras. Similarly for $B$.

Since $\F[[t]]$ is a flat $\F[t]$-module we have, using a base change argument analogous to~\cite[Thm.\  7.11]{matsumura86}, that for any finitely presented $B$-modules $M,N$,
\begin{equation}
\label{e:changeofbasis}
\Hom_B(M,N)\otimes_{\F[t]} \F[[t]] =
\Hom_{B\otimes_{\F[t]} \F[[t]]} (M\otimes_{\F[t]} \F[[t]], N\otimes_{\F[t]} \F[[t]]).
\end{equation}
As $B$ is Noetherian, any finitely generated
$B$-module is finitely presented.
By~\eqref{e:completiontensor}
and~\eqref{e:changeofbasis}, we obtain $\widehat{\embed}$ as the composition:
\begin{align*}
\begin{split}
\widehat{A} &\cong A\otimes_{\F[u]} \F[[u]] \\
&\cong \End_B(T)\otimes_{\F[t]} \F[[t]] \\
&\cong \End_{\widehat{B}}(\widehat{T}),
\end{split}
\end{align*}
where the second isomorphism is induced
by $\embed$. The proof of the second part then goes through in the same way as for Corollary~\ref{c:obtainJKS}.
\end{proof}

\section{Geometric exchange} 
\label{s:exchange}

By Corollary~\ref{c:obtainJKS},
the algebra $eA_De$ does not depend on the choice of $(k,n)$-Postnikov diagram $D$. In this section, we give
an alternative proof of this fact by showing
directly that $eA_De$ is invariant under the
untwisting and twisting moves and
the boundary untwisting and twisting moves
(see Figure~\ref{f:twisting})
and the geometric exchange move (see Figure~\ref{f:geomexquiver}).

The following lemma shows that $A_D$ is invariant under equivalence of Postnikov diagrams.

\begin{lemma} \label{l:twistinginvariance}
The algebra $A_D$ is invariant (up to 
isomorphism) under the untwisting and 
twisting moves and the boundary untwisting
and twisting moves (see Figure~\ref{f:twisting}).
\end{lemma}

\begin{proof}
It is enough to consider the untwisting moves in Figure~\ref{f:twisting}. Invariance
under the twisting moves follows, and
the arguments for the moves obtained by reflecting those in Figure~\ref{f:twisting}
in a horizontal line of symmetry are similar.
Suppose first that $D'$ is obtained from $D$ 
by applying a untwisting move, as in the
top diagram in Figure~\ref{f:twisting}.
Then, locally, the quivers $Q(D)$ and $Q(D')$
are as shown in
the top diagram of Figure~\ref{f:twistquiver} 
(with part of $D$ and $Q(D)$ shown
on the left).
In general we will denote the path in $Q(D')$ corresponding to a path
$\pi$ in $Q(D)$ by $\pi'$; this is well defined for any path not passing
along the arrows $\alpha$ or $\beta$.

Let $F_1$ be the face in $Q_2(D)$ on the left hand side of
$\alpha$, $F_2$ the face with boundary $\beta\alpha$ and $F_3$ the
face on the right hand side of $\beta$. Then $\bdry F_1=q\alpha$
for some path $q$ and $\bdry F_2 =p\beta$ for some path $\beta$.
The corresponding part of $Q(D')$ has only one face, $F'$, with
$\bdry F'=q'p'$. The other faces of $Q(D)$ and $Q(D')$
are in a natural one-to-one correspondence.

Recall that each arrow of $Q(D)$ determines a defining relation
$$
p^+_{\alpha}=p^-_{\alpha},
$$
of $A_D$ (see equation~\eqref{e:definingrelations}).
The defining relation corresponding to the arrow
$\alpha$ (respectively, $\beta$) is $q=\beta$ (respectively, $p=\alpha$).
If $\gamma$ is an arrow in the path $p$, so that $p=p_2\gamma p_1$ for
some paths $p_1$ and $p_2$, then $p^-_{\gamma}$ is the path $p_1\beta p_2$.
Similarly, if $\delta$ is an arrow in the path $q$,
so that $q=q_2\delta q_1$ for some paths $q_1,q_2$,
then $q^-_{\delta}$ is the path $q_1\alpha q_2$.

Making the substitutions $\alpha=p$ and $\beta=q$,
we obtain $p^-_{\gamma}=p_1qp_2$ and $q^-_{\delta}=q_1pq_2$ in $A_D$.
Since we have, in $\F Q(D')$, that
$(p^-_{\gamma'})'=p'_1q'p'_2$ and $(q^-_{\delta'})=q'_1p'q'_2$,
it follows that $A_D$ is isomorphic to $A_{D'}$, since the
defining relations for $A_{D'}$ and $A_D$ correspond precisely
away from the local area affected by the untwisting move.

For the boundary \emph{untwisting move} case (the lower pair of diagrams in
Figure~\ref{f:twisting}), the corresponding change in the quiver is displayed
in the lower pair of diagrams in Figure~\ref{f:twistquiver}.
The defining relation of $A_D$ corresponding to the arrow $\beta$ is
$p=\alpha$. If $\gamma$ is an arrow in $p$, so that $p=p_2\gamma p_1$, then
$p^-_{\gamma}=p_1\beta p_2$, corresponding to $(p^-_{\gamma'})'=p'_1\beta'p'_2$
in $\F Q(D')$. Noting that $\beta'$ is a boundary arrow in $Q(D')$ (so has
no corresponding defining relation), we see that the substitution
$p=\alpha$ gives an isomorphism between $A_D$ and $A_{D'}$ in this case.
\end{proof}

Note that, by Lemma~\ref{l:twistinginvariance}, we see that $A_D$ is isomorphic
to $\End_B(T_D)$ for any Postnikov diagram $D$; see Theorem~\ref{t:isomorphism}. Since the isomorphism in Lemma~\ref{l:twistinginvariance} sends the arrow ideal
to the arrow ideal, it follows that
$\widehat{A}_D\cong \End_{\widehat{B}}(\widehat{T}_D)$ also.

\begin{figure}
\[
\begin{tikzpicture}[scale=0.6
,baseline=(bb.base),
  strand/.style={black,dashed,thick},
  path/.style={black,dotted}
   quivarrow/.style={black, -latex, thick},
  doublearrow/.style={black, latex-latex, very thick}]
\newcommand{\strarrow}{\arrow{angle 60}}
\newcommand{\quiverarrow}{\arrow{black, -latex, thick}}
\newcommand{\dotrad}{0.1cm} 
\newcommand{\bdrydotrad}{{0.8*\dotrad}} 
\path (0,0) node (bb) {}; 


\draw [strand] plot[smooth]
coordinates {(-3.14,1.0) (-2.36,0.71) (-1.57,0) (-0.79,-0.71) (0,-1) (0.79,-0.71) (1.57,0) (2.36, 0.71) (3.14,1)}
[ postaction=decorate, decoration={markings,
  mark= at position 0.12 with \strarrow, mark= at position 0.5 with \strarrow, mark= at position 0.88 with \strarrow}];


\draw [strand] plot[smooth]
coordinates {(3.14,-1.0) (2.36,-0.71) (1.57,0) (0.79,0.71) (0,1) (-0.79,0.71)
(-1.57,0) (-2.36, -0.71) (-3.14,-1)}
[ postaction=decorate, decoration={markings,
  mark= at position 0.15 with \strarrow, mark= at position 0.5 with \strarrow,
mark= at position 0.91 with \strarrow}];


\draw (-3.14,0) node {$F_1$};
\draw (0,0) node {$F_2$};
\draw (3.14,0) node {$F_3$};


\draw (0,-2.5) node {$I$};
\draw (0,2.5) node {$J$};


\draw [black,thick] plot[smooth]
coordinates {(0.25,-2.25) (1.57,0) (0.25,2.25)}
[postaction=decorate, decoration={markings,
  mark= at position 0.995 with \arrow{latex}}];

\draw [black,thick] plot[smooth]
coordinates {(-0.25,2.25) (-1.57,0) (-0.25,-2.25)}
[postaction=decorate, decoration={markings,
  mark= at position 0.995 with \arrow{latex}}];


\draw (-1.1,1.7) node {$\alpha$};
\draw (1.15,1.7) node {$\beta$};


\draw [dash pattern=on 3pt off 1pt on \the\pgflinewidth off 1pt] plot[smooth]
coordinates {(0.35,2.5) (4,1.7) (4.5,0) (4,-1.7) (0.35, -2.5)}
[postaction=decorate, decoration={markings,
mark= at position 0.5 with \strarrow}];

\draw [dash pattern=on 3pt off 1pt on \the\pgflinewidth off 1pt] plot[smooth]
coordinates {(-0.35,-2.5) (-4,-1.7) (-4.5,0) (-4,1.7) (-0.35, 2.5)}
[postaction=decorate, decoration={markings,
mark= at position 0.5 with \strarrow}];


\draw (4.3,1.9) node {$p$};
\draw (-4.3,1.9) node {$q$};


\begin{scope}[shift={(12.8,0)}]

\draw [strand] plot
coordinates {(-3.14,1.0) (3.14,1.0)}
[ postaction=decorate, decoration={markings,
  mark= at position 0.12 with \strarrow, mark= at position 0.5 with \strarrow, mark= at position 0.88 with \strarrow}];


\draw [strand] plot
coordinates {(3.14,-1.0) (-3.14,-1.0)}
[ postaction=decorate, decoration={markings,
  mark= at position 0.15 with \strarrow, mark= at position 0.5 with \strarrow, mark= at position 0.91 with \strarrow}];


\draw (0,0) node {$F'$};


\draw (0,-2.5) node {$I$};
\draw (0,2.5) node {$J$};


\draw [dash pattern=on 3pt off 1pt on \the\pgflinewidth off 1pt] plot[smooth]
coordinates {(0.35,2.5) (4,1.7) (4.5,0) (4,-1.7) (0.35, -2.5)}
[postaction=decorate, decoration={markings,
mark= at position 0.5 with \strarrow}];

\draw [dash pattern=on 3pt off 1pt on \the\pgflinewidth off 1pt] plot[smooth]
coordinates {(-0.35,-2.5) (-4,-1.7) (-4.5,0) (-4,1.7) (-0.35, 2.5)}
[postaction=decorate, decoration={markings,
mark= at position 0.5 with \strarrow}];


\draw (4.3,2) node {$p'$};
\draw (-4.3,1.9) node {$q'$};
\end{scope}


\draw [doublearrow] (6,0) -- (7,0);


\begin{scope}[shift={(-1,-7)}]


\draw (-1.57,0) circle(\bdrydotrad) [fill=gray];


\draw [strand] plot[smooth]
coordinates {(-1.57,0) (-0.79,0.71) (0,1) (0.79,0.71) (1.57,0) (2.36, -0.71) (3.14,-1)}
[ postaction=decorate, decoration={markings,
  mark= at position 0.33 with \strarrow, mark= at position 0.83 with \strarrow}];


\draw [strand] plot[smooth]
coordinates {(3.14,1) (2.36,0.71) (1.57,0) (0.79,-0.71) (0,-1) (-0.79,-0.71) (-1.57,0)}
[ postaction=decorate, decoration={markings,
  mark= at position 0.2 with \strarrow, mark= at position 0.7 with \strarrow}];


\draw (0,-2.5) node {$I$};
\draw (0,2.5) node {$J$};


\draw [black,thick] plot[smooth]
coordinates {(0.25,2.25) (1.57,0) (0.25,-2.25)}
[postaction=decorate, decoration={markings,
  mark= at position 0.995 with \arrow{latex}}];

\draw [black,thick] plot[smooth]
coordinates {(-0.25,-2.25) (-1.57,0) (-0.25,2.25)}
[postaction=decorate, decoration={markings,
  mark= at position 0.995 with \arrow{latex}}];


\draw (-1.1,1.7) node {$\alpha$};
\draw (1.15,1.7) node {$\beta$};


\draw [dash pattern=on 3pt off 1pt on \the\pgflinewidth off 1pt] plot[smooth]
coordinates {(0.35,-2.5) (4,-1.7) (4.5,0) (4,1.7)
(0.35, 2.5)}[ postaction=decorate, decoration={markings,
mark= at position 0.49 with \strarrow}];


\draw (4.3,1.9) node {$p$};

\begin{scope}[shift={(12.5,0)}]


\draw [strand] plot[smooth]
coordinates {(-1.57,0) (-0.79,-0.71) (0,-1) (1.57,-1) (3.14,-1)}
[ postaction=decorate, decoration={markings,
  mark= at position 0.54 with \strarrow}];


\draw [strand] plot[smooth]
coordinates {(3.14,1) (1.57,1) (0,1) (-0.79,0.71) (-1.57,0)}
[ postaction=decorate, decoration={markings,
  mark= at position 0.5 with \strarrow}];


\draw (-1.57,0) circle(\bdrydotrad) [fill=gray];


\draw (0,-2.5) node {$I$};
\draw (0,2.5) node {$J$};


\draw [black,thick] plot[smooth]
coordinates {(-0.25,2.25) (-1.57,0) (-0.25,-2.25)}
[postaction=decorate, decoration={markings,
  mark= at position 0.995 with \arrow{latex}}];


\draw (-1.1,1.7) node {$\beta'$};


\draw [dash pattern=on 3pt off 1pt on \the\pgflinewidth off 1pt] plot[smooth]
coordinates {(0.35,-2.5) (4,-1.7) (4.5,0) (4,1.7)
(0.35, 2.5)}[ postaction=decorate, decoration={markings,
mark= at position 0.49 with \strarrow}];


\draw (4.3,2) node {$p'$};

\end{scope}


\draw [doublearrow] (7,0) -- (8,0);

\end{scope}

\end{tikzpicture}
\]
\caption{The effect of an untwisting move
or boundary untwisting move on the quiver of a Postnikov diagram}
\label{f:twistquiver}
\end{figure}
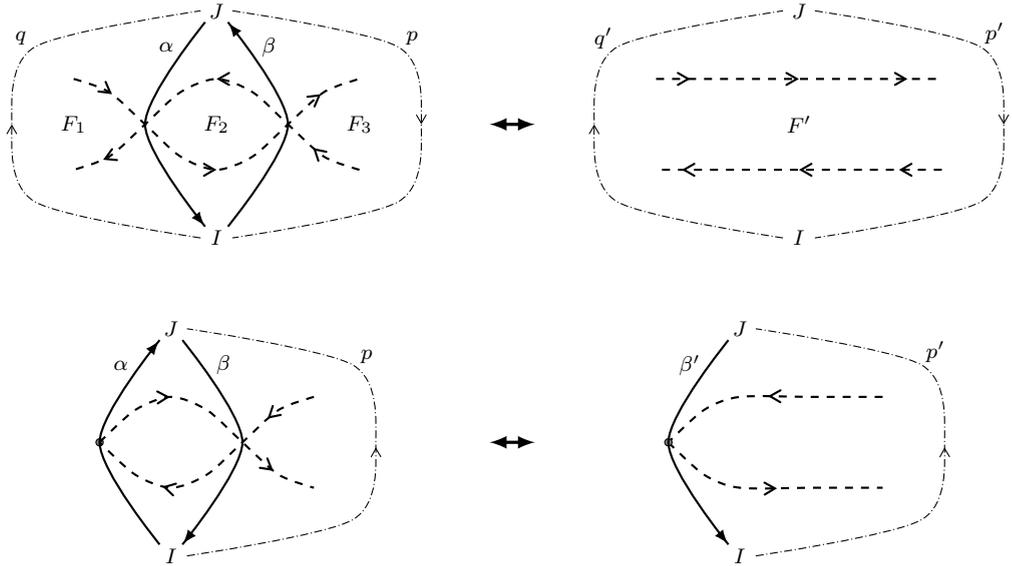

There is a notion of \emph{geometric exchange} on a Postnikov diagram, i.e.\ applying the local rule shown in
Figure~\ref{f:geomexquiver}~\cite[\S14]{postnikov}
(see also~\cite[\S3]{scott06}); we also illustrate the change in the quiver.
The effect of this transformation on the plabic graph (respectively, quiver) is sometimes referred to as \emph{urban renewal}. Urban renewal is discussed in~\cite{ciucu98,KPW00}, whose authors refer to unpublished work of G. Kuperberg. See also the discussion in~\cite[1.7]{GK}, where it is referred to as a \emph{spider move}. The effect on the quiver is \emph{Seiberg duality}~\cite[\S6]{fhkvw}; see also Remark~\ref{r:mutation} below.
Note that the $k$-subsets labelling the vertices remain unchanged under a geometric exchange except for the central region, which gets a new
label.

\begin{proposition} \label{p:geomexisomorphism}
Let $D$ and $D'$ be Postnikov diagrams and suppose that
$D'$ is obtained from $D$ by applying the geometric exchange move
at an internal vertex labelled by a $k$-subset $I$,
as in Figure~\ref{f:geomexquiver}. Let $I'$ be the label of the
corresponding vertex in $D'$. Let
$$e^I=\sum_{J\in D,\ J\not=I}e_J,
\quad\quad e^{I'}=\sum_{J\in D',\ J\not=I'}e_J.$$
Then $e^IA_De^I\cong e^{I'}A_{D'}e^{I'}$.
\end{proposition}

\begin{figure}
\[
\begin{tikzpicture}[scale=0.6,baseline=(bb.base),
  strand/.style={black,dashed,thick},
   quivarrow/.style={black, -latex, thick},
  doublearrow/.style={black, latex-latex, very thick}]
\newcommand{\strarrow}{\arrow{angle 60}}
\newcommand{\dotrad}{0.1cm} 
\newcommand{\bdrydotrad}{{0.8*\dotrad}} 
\path (0,0) node (bb) {}; 


\draw (0,0) node {$I$};


\draw [quivarrow,shorten <=8pt,shorten >=5pt]
(0,0)--(0,4);

\draw [quivarrow,shorten <=8pt,shorten >=5pt]
(0,0)--(0,-4);

\draw [quivarrow,shorten <=5pt,shorten >=8pt]
(4,0)--(0,0);

\draw [quivarrow,shorten <=5pt,shorten >=8pt]
(-4,0)--(0,0);


\draw [strand] plot coordinates {(4,-2) (-2,4)}[postaction=decorate,decoration={markings,
mark= at position 0.17 with \strarrow,
mark= at position 0.5 with \strarrow,
mark= at position 0.83 with \strarrow}];

\draw [strand] plot coordinates {(-4,2) (2,-4)}[postaction=decorate,decoration={markings,
mark= at position 0.17 with \strarrow,
mark= at position 0.5 with \strarrow,
mark= at position 0.83 with \strarrow}];

\draw [strand] plot coordinates {(-4,-2) (2,4)}[postaction=decorate,decoration={markings,
mark= at position 0.17 with \strarrow,
mark= at position 0.5 with \strarrow,
mark= at position 0.83 with \strarrow}];

\draw [strand] plot coordinates {(4,2) (-2,-4)}[postaction=decorate,decoration={markings,
mark= at position 0.17 with \strarrow,
mark= at position 0.5 with \strarrow,
mark= at position 0.83 with \strarrow}];


\draw [dash pattern=on 3pt off 1pt on \the\pgflinewidth
off 1pt] plot[smooth]
coordinates {(0.2,4) (2,5) (4.5,4.5) (5,2) (4,0.2)}[ postaction=decorate, decoration={markings,
mark= at position 0.5 with \strarrow}];

\draw [dash pattern=on 3pt off 1pt on \the\pgflinewidth
off 1pt] plot[smooth]
coordinates {(-0.2,4) (-2,5) (-4.5,4.5) (-5,2) (-4,0.2)}[ postaction=decorate, decoration={markings,
mark= at position 0.5 with \strarrow}];

\draw [dash pattern=on 3pt off 1pt on \the\pgflinewidth
off 1pt] plot[smooth]
coordinates {(0.2,-4) (2,-5) (4.5,-4.5) (5,-2) (4,-0.2)}[ postaction=decorate, decoration={markings,
mark= at position 0.5 with \strarrow}];

\draw [dash pattern=on 3pt off 1pt on \the\pgflinewidth
off 1pt] plot[smooth]
coordinates {(-0.2,-4) (-2,-5) (-4.5,-4.5) (-5,-2) (-4,-0.2)}[ postaction=decorate, decoration={markings,
mark= at position 0.5 with \strarrow}];


\draw (4.9,4.7) node {$p$};
\draw (4.9,-4.7) node {$q$};
\draw (-4.9,-4.7) node {$r$};
\draw (-4.9,4.7) node {$s$};


\draw (0.25,3) node {$\alpha$};
\draw (3,0.25) node {$\beta$};
\draw (0.25,-3) node {$\gamma$};
\draw (-3,0.25) node {$\delta$};


\begin{scope}[shift={(12,0)}]

\draw (0,0) node {$I'$};


\draw [quivarrow,shorten <=5pt,shorten >=8pt]
(0,4)--(0,0);

\draw [quivarrow,shorten <=5pt,shorten >=8pt]
(0,-4)--(0,0);

\draw [quivarrow,shorten <=8pt,shorten >=5pt]
(0,0)--(4,0);

\draw [quivarrow,shorten <=8pt,shorten >=5pt]
(0,0)--(-4,0);

\draw [quivarrow,shorten <=5pt,shorten >=5pt]
(4,0)--(0,-4);

\draw [quivarrow,shorten <=5pt,shorten >=5pt]
(-4,0)--(0,-4);

\draw [quivarrow,shorten <=5pt,shorten >=5pt]
(4,0)--(0,4);

\draw [quivarrow,shorten <=5pt,shorten >=5pt]
(-4,0)--(0,4);


\draw [strand] plot[smooth] coordinates {(4,-2) (2,-2) (0,-2) (-2,0) (-2,2) (-2,4)}[postaction=decorate,decoration={markings,
mark= at position 0.1 with \strarrow,
mark= at position 0.27 with \strarrow,
mark= at position 0.5 with \strarrow,
mark= at position 0.74 with \strarrow,
mark= at position 0.92 with \strarrow
}];

\draw [strand] plot[smooth] coordinates {(-4,2) (-2,2) (0,2) (2,0) (2,-2) (2,-4)}[postaction=decorate,decoration={markings,
mark= at position 0.1 with \strarrow,
mark= at position 0.27 with \strarrow,
mark= at position 0.5 with \strarrow,
mark= at position 0.74 with \strarrow,
mark= at position 0.92 with \strarrow}];

\draw [strand] plot[smooth] coordinates {(-4,-2) (-2,-2) (0,-2) (2,0) (2,2) (2,4)}[postaction=decorate,decoration={markings,
mark= at position 0.1 with \strarrow,
mark= at position 0.27 with \strarrow,
mark= at position 0.5 with \strarrow,
mark= at position 0.74 with \strarrow,
mark= at position 0.92 with \strarrow}];

\draw [strand] plot[smooth] coordinates {(4,2) (2,2) (0,2) (-2,0) (-2,-2) (-2,-4)}[postaction=decorate,decoration={markings,
mark= at position 0.1 with \strarrow,
mark= at position 0.27 with \strarrow,
mark= at position 0.5 with \strarrow,
mark= at position 0.74 with \strarrow,
mark= at position 0.92 with \strarrow}];


\draw [dash pattern=on 3pt off 1pt on \the\pgflinewidth
off 1pt] plot[smooth]
coordinates {(0.2,4) (2,5) (4.5,4.5) (5,2) (4,0.2)}[ postaction=decorate, decoration={markings,
mark= at position 0.5 with \strarrow}];

\draw [dash pattern=on 3pt off 1pt on \the\pgflinewidth
off 1pt] plot[smooth]
coordinates {(-0.2,4) (-2,5) (-4.5,4.5) (-5,2) (-4,0.2)}[ postaction=decorate, decoration={markings,
mark= at position 0.5 with \strarrow}];

\draw [dash pattern=on 3pt off 1pt on \the\pgflinewidth
off 1pt] plot[smooth]
coordinates {(0.2,-4) (2,-5) (4.5,-4.5) (5,-2) (4,-0.2)}[ postaction=decorate, decoration={markings,
mark= at position 0.5 with \strarrow}];

\draw [dash pattern=on 3pt off 1pt on \the\pgflinewidth
off 1pt] plot[smooth]
coordinates {(-0.2,-4) (-2,-5) (-4.5,-4.5) (-5,-2) (-4,-0.2)}[ postaction=decorate, decoration={markings,
mark= at position 0.5 with \strarrow}];


\draw (4.9,4.7) node {$p$};
\draw (4.9,-4.7) node {$q$};
\draw (-4.9,-4.7) node {$r$};
\draw (-4.9,4.7) node {$s$};


\draw (0.35,2.7) node {$\alpha^*$};
\draw (2.7,0.25) node {$\beta^*$};
\draw (0.35,-2.7) node {$\gamma^*$};
\draw (-2.7,0.25) node {$\delta^*$};

\draw (3.08,1.3) node {$\varepsilon_p$};
\draw (-3.13,1.3) node {$\varepsilon_s$};
\draw (3.18,-1.3) node {$\varepsilon_q$};
\draw (-3.13,-1.3) node {$\varepsilon_r$};

\end{scope}


\draw [doublearrow] (5.5,0) -- (6.5,0);

\end{tikzpicture}
\]
\caption{A geometric exchange and the corresponding change in the quiver}
\label{f:geomexquiver}
\end{figure}
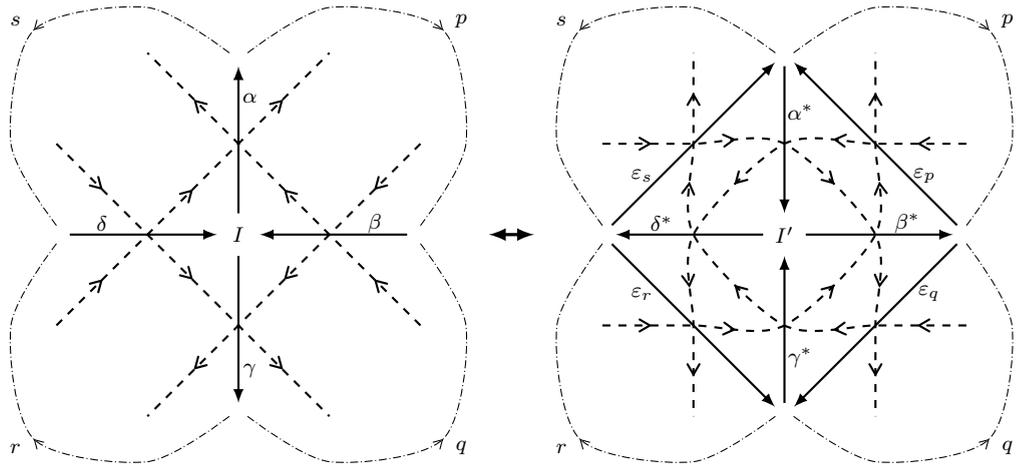

\begin{proof}
We use the labelling of arrows as in Figure~\ref{f:geomexquiver}.
Since the vertex $I$ is internal, the arrows $\alpha,\beta,\gamma,\delta$ are
not boundary arrows, so there is a face $F$ whose boundary is $p\alpha\beta$
for some path $p$. In a similar way, we let $q$ be the completion of $\gamma\beta$,
$r$ the completion of $\gamma\delta$ and $s$ the completion of $\alpha\delta$ to face
boundaries. We compute $e^IA_De^I$ as a quiver with relations.

Let $\I_D$ be the ideal of relations defining $A_D$.
Then we have
$$e^IA_De^I=e^I\left( \frac{\F Q(D)}{\I_D} \right)e^I=\frac{e^I\F Q(D)e^I}{e^I\I_D e^I}.$$
It is easy to check that $e^I\F Q(D)e^I$ is isomorphic to $\F \Gamma$, where
$\Gamma$ is obtained from $Q(D)$ by removing the vertex $I$ and all incident
arrows and adding new arrows $\varepsilon_p$, $\varepsilon_q$, $\varepsilon_r$ and $\varepsilon_s$,
corresponding to the paths $\alpha\beta$, $\gamma\beta$, $\gamma\delta$ and $\alpha\delta$ respectively. Note that
$\varepsilon_p$ goes between the same vertices that $p$ does, only in the opposite
direction; similarly for $\varepsilon_q$, $\varepsilon_r$ and $\varepsilon_s$. We shall
see that these arrows correspond to the arrows in $Q(D')$ with the same names.

The new relations generating $e^I\I_De^I$ can be taken to be the old relations
between vertices other than $I$ together with new relations coming from the old
relations between $I$ and itself or other vertices, obtained by premultiplying
or postmultiplying the old relations by an arrow.
The relations we have to consider are $\beta p=\delta s$, $p\alpha =q\gamma$,
$\beta q=\delta r$ and $s\alpha =r\gamma $, corresponding to the arrows
$\alpha$, $\beta$, $\gamma$ and $\delta$, respectively.

The relation $\beta p=\delta s$ gives $\alpha\beta p=\alpha\delta s$ and
$\gamma\beta p=\gamma\delta s$, i.e.\ $\varepsilon_p p=\varepsilon_s s$ and $\varepsilon_q p=\varepsilon_r s$.
The other relations give
$p\varepsilon_p=q\varepsilon_q$ and $p\varepsilon_s=q\varepsilon_r$,
$\varepsilon_q q=\varepsilon_r r$ and $\varepsilon_p q=\varepsilon_s r$,
and $s\varepsilon_s=r \varepsilon_r$ and $s\varepsilon_p=r\varepsilon_q$
respectively.

We next do a similar computation for $e^{I'}A_{D'}e^{I'}$. Let $p,q,r,s$
be the paths in $Q(D')$ corresponding to the paths with the same names in $Q(D)$.
Let $\I_{D'}$ be the ideal of relations defining $A_{D'}$.
Then $e^{I'}\F Q(D') e^{I'}$ is isomorphic to $\F \Gamma'$,
where $\Gamma'$ is obtained from $Q(D')$ by removing the vertex $I'$
and adding extra arrows
$\zeta_p$, $\zeta_q$, $\zeta_r$ and $\zeta_s$ corresponding to the paths
$\beta^*\alpha^*$, $\beta^*\gamma^*$, $\delta^*\gamma^*$ and $\delta^*\alpha^*$
respectively. We must consider the relations
$\varepsilon_s \delta^*=\varepsilon_p \beta^*$,
$\alpha^*\varepsilon_p=\gamma^* \varepsilon_q$,
$\varepsilon_q \beta^*=\varepsilon_r \delta^*$ and
$\gamma^* \varepsilon_r=\alpha^*\varepsilon_s$,
corresponding to the arrows
$\alpha^*$, $\beta^*$, $\gamma^*$ and $\delta^*$, respectively.

The relation $\varepsilon_s \delta^*=\varepsilon_p \beta^*$ gives
$\varepsilon_s \delta^* \alpha^*=\varepsilon_p\beta^*\alpha^*$ and
$\varepsilon_s \delta^* \gamma^*=\varepsilon_p \beta^* \gamma^*$, i.e.\
$\varepsilon_s \zeta_s = \varepsilon_p \zeta_p$ and $\varepsilon_s \zeta_r = \varepsilon_p \zeta_q$.
The other relations give
$\zeta_p \varepsilon_p=\zeta_q \varepsilon_q$ and $\zeta_s \varepsilon_p=\zeta_r \varepsilon_q$,
$\varepsilon_q\zeta_q=\varepsilon_r\zeta_r$ and $\varepsilon_q\zeta_p=\varepsilon_r\zeta_s$,
$\zeta_r\varepsilon_r=\zeta_s \varepsilon_s$ and $\zeta_q\varepsilon_r=\zeta_p \varepsilon_s$,
respectively.

Thus $e^{I'}A_{D'}e^{I'}$ is isomorphic to the quotient of $\F \Gamma'$ by the
ideal generated by the above relations and the old relations in $A_{D'}$
between vertices not equal to $I'$.

This means that in $e^{I'}A_{D'}e^{I'}$, we also have the relations
$\beta^*\alpha^*=p$, $\beta^*\gamma^*=q$,
$\delta^*\gamma^*=r$ and $\delta^*\alpha^*=s$ coming from the arrows
$\varepsilon_p$, $\varepsilon_q$, $\varepsilon_r$ and $\varepsilon_s$ respectively. Thus, we have
$\zeta_p=p$, $\zeta_q=q$, $\zeta_r=r$ and $\zeta_s=s$ in $e^{I'}A_{D'}e^{I'}$, so we
can remove the arrows $\zeta_p$, $\zeta_q$, $\zeta_r$ and $\zeta_s$ from
the quiver $\Gamma'$ and replace them with $p,q,r,s$ in the above relations.
These relations become:
$\varepsilon_s s = \varepsilon_p p$, $\varepsilon_s r = \varepsilon_p q$,
$p \varepsilon_p=q \varepsilon_q$, $s \varepsilon_p=r \varepsilon_q$,
$\varepsilon_q q=\varepsilon_r r$, $\varepsilon_q p=\varepsilon_r s$,
$r\varepsilon_r=s \varepsilon_s$ and $q\varepsilon_r=p \varepsilon_s$,
corresponding to the relations defining $e^I A_D e^I$ computed above. Since
the other defining relations in $e^I A_D e_I$ and $e^{I'} A_{D'} e^{I'}$ are the
same, we see that $e^IA_{D}e^I$ is isomorphic to $e^{I'}A_{D'}e^{I'}$ as required.
\end{proof}

We remark that an alternative proof of Proposition~\ref{p:geomexisomorphism}
can be given by using Theorem~\ref{t:isomorphism} and noting that
$e^IA_De^I\cong \End_{\B}(T_D/\widehat{\module}_I)$.

\begin{remark} \label{r:mutation}
The effect on $Q(D)$ of applying the geometric exchange is to carry out the first two
steps of Fomin-Zelevinsky quiver mutation~\cite{FZ02} at the vertex $I$, i.e.
\begin{enumerate}[(a)]
\item
For all paths of length two (with multiplicity) $J\to I\to K$, add an arrow $J\to K$.
\item
Reverse all arrows incident with $I$.
\end{enumerate}
The third step would usually be to cancel all two-cycles appearing after the first two
steps. Instead, we carry out a slightly modified version of the third step, corresponding to applying Lemma~\ref{l:twistinginvariance}, i.e.
\begin{enumerate}
\item[(c)]
Cancel all two-cycles consisting of non-boundary arrows.
\item[(d)]
For all two-cycles consisting of a boundary arrow and a non-boundary arrow,
remove the boundary arrow and convert the non-boundary arrow into a boundary arrow.
\end{enumerate}
The abstract rewriting system describing the individual moves in (c) and (d) is
clearly terminating (as the number of arrows decreases with every step) and
it is also easy to check that it is locally confluent. Hence, by the Diamond Lemma,
it is confluent, and thus convergent. In other words, it does not matter in which
order the individual steps in (c) and (d) are carried out; the resulting quiver
will be independent of the order.
\end{remark}

By Corollary~\ref{c:obtainJKS}, the boundary
algebra $eA_De$ is not dependent on the choice of
$D$ up to isomorphism. Lemmas~
\ref{l:twistinginvariance} and~
\ref{p:geomexisomorphism} give an alternative proof
of this fact.

\begin{corollary} \label{c:Dindependent}
Let $D,D'$ be any two
$(k,n)$-Postnikov diagrams.
Then the corresponding boundary algebras $eA_De$
and $e'A_{D'}e'$ are isomorphic.
\end{corollary}

\begin{proof}
By Lemmas~\ref{l:twistinginvariance} and~\ref{p:geomexisomorphism},
this holds whenever $D'$ can be obtained
from $D$ by a geometric exchange: if the exchange takes place at a vertex
$I$, replacing it with $I'$, then we have
$e'A_{D'}e'=e'e^{I'}A_{D'}e^{I'}e'$ is isomorphic to $ee^IA_De^Ie$.
In the general case, by~\cite[\S14]{postnikov} (see also~\cite{scott06}),
there is a sequence of geometric exchanges and untwisting or twisting moves or
boundary untwisting or twisting moves  taking $D$ to $D'$, and the result follows.
\end{proof}

\section{Surfaces} 
\label{s:surfaces}

In this section, we generalize the context we are working in to surfaces with boundary. We note that dimer models (bipartite field theories) on surfaces with boundary have also been considered in independent work of S. Franco~\cite{francopre12},
and Postnikov diagrams on surfaces are also considered
in~\cite{GK,marshscott}.

Let $(X,M)$ be a marked oriented Riemann surface with nonempty 
boundary, where $M$ is the set of marked points. We may assume 
that each boundary component is a circle. We also 
suppose that each marked point lies on a boundary component 
and that each boundary component has at least one marked 
point. This is sometimes referred to as the `unpunctured 
case'.
We assume that $(X,M)$ is not a disk with
$1$ or $2$ marked points. Label the boundary components $\mathcal{C}_1,\mathcal{C}_2,\ldots ,\mathcal{C}_b$;
suppose there are $r_i$ marked points on boundary component $\mathcal{C}_i$, for each $i$.
We label the marked points around a boundary component $\mathcal{C}_i$ anticlockwise around
the component as $p_{i1},p_{i2}, \ldots ,p_{ir_i}$.

\begin{definition}
We define a \emph{weak Postnikov diagram} $D$ on $(X,M)$ to be a diagram consisting
of directed curves embedded in $(X,M)$, one starting at each marked point and ending on the
same boundary component on which it starts and exactly one strand ending and one strand
starting at each marked point. The diagram must satisfying the local axioms (a1)--(a3) in Definition~\ref{d:asd} and is considered up to isotopy. It need not necessarily be of reduced type.
We say that $D$ is a \emph{Postnikov diagram} if, in addition, the global axioms
(b1) and (b2) also hold. We say that a (weak) Postnikov diagram $D$ has degree
$k$ if the strand starting at $p_{ij}$ ends at $p_{i,j+k}$, where the second subscript is
interpreted modulo $r_i$.
\end{definition}

Thus, the $(k,n)$-Postnikov diagrams of Definition~\ref{d:asd} are Postnikov diagrams of degree $k$ on a disk with $n$ marked points on its boundary, considered up to the untwisting and twisting moves (see Figure~\ref{f:twisting}).

If $D$ is a weak Postnikov diagram on a marked surface $(X,M)$,
we define the corresponding dimer model $Q(D)$ as in the disk case,
following Definition~\ref{d:quiver} and Remark~\ref{r:plabicdual}.
Let $A_D=A_{Q(D)}$ denote the corresponding dimer algebra, defined as in the disk case.
Let $e$ be the sum of the idempotents in $A_D$ corresponding to the boundary vertices of $Q(D)$.
Then, as in the disk case, we may define the \emph{boundary algebra}
of $D$ to be $B_D=eA_De$. 
Note that it is not clear whether it is independent of the choice of $D$.

\begin{remark}\label{r:bdryidemp}
Note that, as in the disk case, the boundary vertices of $Q(D)$ 
correspond to the alternating regions of $D$ that have as one (unoriented) edge 
an interval on the boundary between two marked points.
Thus the idempotents of the boundary algebra are naturally labelled by
the \emph{boundary intervals} of the marked surface, 
and not the marked (boundary) points.
To draw the quiver $Q(D)$ on the surface, it is natural to locate
each boundary vertex on the corresponding boundary interval.
\end{remark}

We recall a construction of Scott~\cite[\S3]{scott06}. Given a triangulation
$\mathcal{T}$ of the disk with $n$ marked points on its boundary, each triangle is
replaced with a local configuration of strands, as in
Figure~\ref{f:scotttriangle}, to produce a global configuration $D(\mathcal{T})$ of
strands.

\begin{figure}
\[
\begin{tikzpicture}[scale=0.9,baseline=(bb.base),
  strand/.style={black,dashed,thick},
   quivarrow/.style={black, -latex, thick},
  doublearrow/.style={black, latex-latex, very thick},
  anarrow/.style={black,-latex,thick}]
\newcommand{\strarrow}{\arrow{angle 60}}
\newcommand{\dotrad}{0.1cm} 
\newcommand{\bdrydotrad}{{0.8*\dotrad}} 
\path (0,0) node (bb) {}; 


\draw [black] (0,0) -- (3,0) -- (1.5,2.60) -- (0,0);

\draw [anarrow] (4,1.25) -- (4.5,1.25);

\begin{scope}[shift={(5.5,0)}]


\draw [black] (0,0) -- (3,0) -- (1.5,2.6) -- (0,0);


\draw [strand] plot
coordinates{(1.75,2.17) (0.5,0)}
[ postaction=decorate, decoration={markings,
mark= at position 0.5 with \strarrow}];

\draw [strand] plot
coordinates{(2.5,0) (1.25,2.17)}
[ postaction=decorate, decoration={markings,
mark= at position 0.52 with \strarrow}];

\draw [strand] plot
coordinates{(0.25,0.43) (2.75,0.43)}
[ postaction=decorate, decoration={markings,
mark= at position 0.52 with \strarrow}];

\end{scope}
\end{tikzpicture}
\]
\caption{Scott's construction}
\label{f:scotttriangle}
\end{figure}
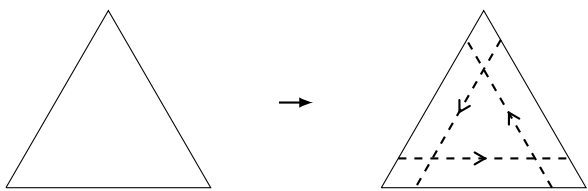

We modify the conventions of Scott slightly, since we are following
Postnikov~\cite{postnikov}. We apply the above rule for an internal
triangle of the triangulation, i.e.\ one all of whose edges are internal to
$X$. For triangles with boundary edges, we apply the same rule except that the
intersection of a strand with an edge of the triangle which is part of the
boundary of $X$ is slid along to the adjacent corner of the triangle.
The upper diagram in Figure~\ref{f:modifiedtriangle} illustrates this in
the case where the horizontal edge in the figure is a boundary edge (and the others are internal); the middle diagram illustrates this in the case where the upper two edges are boundary edges,
while the case where
all three edges of the triangle are boundary
edges is shown in the lower diagram.
In each case the boundary edges are indicated by dotted lines, and internal edges by full lines.

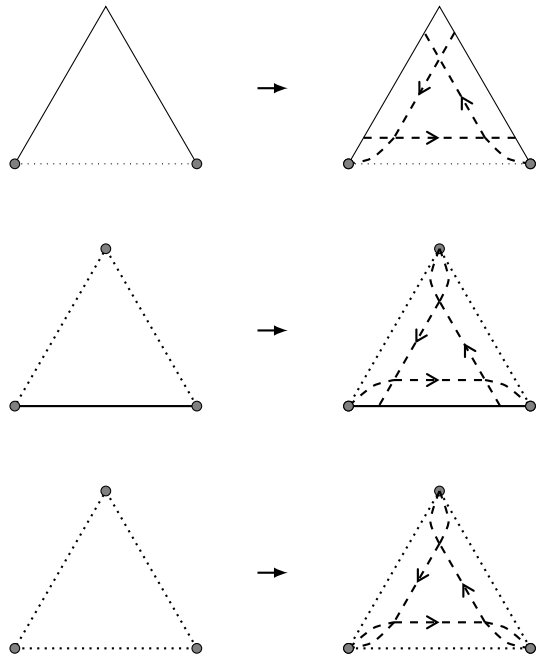
\begin{figure}
\[
\begin{tikzpicture}[scale=0.8,baseline=(bb.base),
  strand/.style={black,dashed,thick},
   quivarrow/.style={black, -latex, thick},
  doublearrow/.style={black, latex-latex, very thick},
  anarrow/.style={black,-latex,thick}]
\newcommand{\strarrow}{\arrow{angle 60}}
\newcommand{\dotrad}{0.1cm} 
\newcommand{\bdrydotrad}{{0.8*\dotrad}} 
\path (0,0) node (bb) {}; 



\draw [black,dotted] (0,0) -- (3,0);
\draw [black] (3,0) -- (1.5,2.598) -- (0,0);


\draw (0,0) circle(\bdrydotrad) [fill=gray];
\draw (3,0) circle(\bdrydotrad) [fill=gray];

\draw [anarrow] (4,1.249) -- (4.5,1.249);

\begin{scope}[shift={(5.5,0)}]


\draw [black,dotted] (0,0) -- (3,0);
\draw [black] (3,0) -- (1.5,2.598) -- (0,0);


\draw (0,0) circle(\bdrydotrad) [fill=gray];
\draw (3,0) circle(\bdrydotrad) [fill=gray];


\draw [strand] plot
coordinates{(0.25,0.43) (2.75,0.43)}
[ postaction=decorate, decoration={markings,
mark= at position 0.5 with \strarrow}];

\draw [strand] plot
coordinates{(1.75,2.17) (0.75,0.43)}
[ postaction=decorate, decoration={markings,
mark= at position 0.6 with \strarrow}];

\draw [strand] plot[smooth]
coordinates{(0.75,0.43) (0.4,0.1) (0,0)};

\draw [strand] plot
coordinates{(2.25,0.43) (1.25,2.17)}
[ postaction=decorate, decoration={markings,
mark= at position 0.4 with \strarrow}];

\draw [strand] plot[smooth]
coordinates{(3,0) (2.6,0.1) (2.25,0.433)};

\draw (0,0) circle(\bdrydotrad) [fill=gray];
\draw (3,0) circle(\bdrydotrad) [fill=gray];

\end{scope}
\begin{scope}[shift={(0,-4)}]



\draw [black,thick] (0,0) -- (3,0);
\draw [black,thick,dotted] (3,0) -- (1.5,2.598) -- (0,0);


\draw (0,0) circle(\bdrydotrad) [fill=gray];
\draw (3,0) circle(\bdrydotrad) [fill=gray];
\draw (1.5,2.598) circle(\bdrydotrad) [fill=gray];

\draw [anarrow] (4,1.249) -- (4.5,1.249);

\begin{scope}[shift={(5.5,0)}]


\draw [black,thick] (0,0) -- (3,0);
\draw [black,thick,dotted] (3,0) -- (1.5,2.598) -- (0,0);


\draw (0,0) circle(\bdrydotrad) [fill=gray];
\draw (3,0) circle(\bdrydotrad) [fill=gray];
\draw (1.5,2.598) circle(\bdrydotrad) [fill=gray];


\draw [strand] plot
coordinates{(1.5,1.73) (0.5,0)}
[ postaction=decorate, decoration={markings,
mark= at position 0.4with \strarrow}];

\draw [strand] plot[smooth]
coordinates{(1.5,2.6) (1.65,2.1) (1.5,1.73)};

\draw [strand] plot
coordinates{(2.5,0) (1.5,1.73)}
[ postaction=decorate, decoration={markings,
mark= at position 0.6 with \strarrow}];

\draw [strand] plot[smooth]
coordinates{(1.5,1.73) (1.35,2.1) (1.5,2.6)};

\draw [strand] plot
coordinates{(0.75,0.43) (2.25,0.43)}
[ postaction=decorate, decoration={markings,
mark= at position 0.5 with \strarrow}];

\draw [strand] plot[smooth]
coordinates{(0,0) (0.4,0.33) (0.75,0.43)};

\draw [strand] plot[smooth]
coordinates{(2.25,0.43) (2.6,0.333) (3,0)};

\draw (0,0) circle(\bdrydotrad) [fill=gray];
\draw (3,0) circle(\bdrydotrad) [fill=gray];

\end{scope}
\end{scope}

\begin{scope}[shift={(0,-8)}]



\draw [black,thick,dotted] (0,0) -- (3,0);
\draw [black,thick,dotted] (3,0) -- (1.5,2.598) -- (0,0);


\draw (0,0) circle(\bdrydotrad) [fill=gray];
\draw (3,0) circle(\bdrydotrad) [fill=gray];
\draw (1.5,2.598) circle(\bdrydotrad) [fill=gray];

\draw [anarrow] (4,1.249) -- (4.5,1.249);

\begin{scope}[shift={(5.5,0)}]


\draw [black,thick,dotted] (0,0) -- (3,0);
\draw [black,thick,dotted] (3,0) -- (1.5,2.598) -- (0,0);


\draw (0,0) circle(\bdrydotrad) [fill=gray];
\draw (3,0) circle(\bdrydotrad) [fill=gray];
\draw (1.5,2.598) circle(\bdrydotrad) [fill=gray];


\draw [strand] plot
coordinates{(1.5,1.73) (0.75,0.43)}
[ postaction=decorate, decoration={markings,
mark= at position 0.5 with \strarrow}];

\draw [strand] plot[smooth]
coordinates{(1.5,2.6) (1.65,2.1) (1.5,1.73)};

\draw [strand] plot
coordinates{(2.25,0.43) (1.5,1.73)}
[ postaction=decorate, decoration={markings,
mark= at position 0.5 with \strarrow}];

\draw [strand] plot[smooth]
coordinates{(1.5,1.73) (1.35,2.1) (1.5,2.6)};

\draw [strand] plot
coordinates{(0.75,0.43) (2.25,0.43)}
[ postaction=decorate, decoration={markings,
mark= at position 0.5 with \strarrow}];

\draw [strand] plot[smooth]
coordinates{(0,0) (0.4,0.33) (0.75,0.43)};

\draw [strand] plot[smooth]
coordinates{(2.25,0.43) (2.6,0.333) (3,0)};

\draw [strand] plot[smooth]
coordinates{(0.75,0.43) (0.4,0.1) (0,0)};

\draw [strand] plot[smooth]
coordinates{(3,0) (2.6,0.1) (2.25,0.433)};

\draw (0,0) circle(\bdrydotrad) [fill=gray];
\draw (3,0) circle(\bdrydotrad) [fill=gray];

\end{scope}
\end{scope}

\end{tikzpicture}
\]
\caption{Modified version of Scott's construction. The dotted lines indicate
boundary edges}
\label{f:modifiedtriangle}
\end{figure}

By~\cite[Cor.~2]{scott06}, the map $\mathcal{T}\mapsto D(\mathcal{T})$ gives a bijection between
triangulations of the disk with $n$ marked points on its boundary and Postnikov diagrams of degree $2$ on the disk (with the same marked points).
The vertices of $Q(D(\mathcal{T}))$ correspond to the edges in $\mathcal{T}$ (including boundary edges),
and applying the geometric exchange at a vertex $I$ corresponds to applying a quadrilateral
flip at the corresponding edge.

We can generalize the map $\mathcal{T}\mapsto D(\mathcal{T})$ to a map from triangulations of
$(X,M)$ to weak Postnikov diagrams on $(X,M)$ (from the construction,
it is clear that axioms (a1)-(a3) in
Definition~\ref{d:asd} all hold for $D(\mathcal{T})$). We have:

\begin{lemma} \label{l:triangulationdiagram}
Let $(X,M)$ be a marked surface with all marked points on the boundary, and
let $\mathcal{T}$ be a triangulation of $(X,M)$. Then $D(\mathcal{T})$ is a weak
Postnikov diagram of degree $2$.
\end{lemma}

\begin{proof}
Axioms (a1) to (a3) in Definition~\ref{d:asd} follow from the local construction of $D(\mathcal{T})$ (and the way
in which two triangles are fitted together). So we just need to check that
the weak Postnikov diagram constructed is of degree $2$.
Let $p_{ij}$ be a marked point on boundary component $\mathcal{C}_i$ of $(X,M)$.
Anticlockwise of $p_{ij}$ are the marked points $p_{i,j+1}$ and $p_{i,j+2}$
(allowing the possibility that one or both of them coincides with $p_{ij}$).
Figure~\ref{f:nearboundary} shows the two triangles adjacent to the boundary
arcs between $p_{ij}$ and $p_{i,j+1}$ and between $p_{i,j+1}$ and $p_{i,j+2}$,
together with all the arcs incident with $p_{i,j+1}$. Note that the two triangles may coincide.
The dotted line at the
base of the figure indicates the boundary of the surface (but note that some
of the edges at the top of the figure may also be on the boundary).
The figure also shows the strand which starts at $p_{ij}$; we see that it ends
at $p_{i,j+2}$ as required; note that its path is not affected by whether the
edges at the top of the figure are boundary edges or not.
\end{proof}

\begin{figure}
\[
\begin{tikzpicture}[scale=0.8,baseline=(bb.base),
  strand/.style={black,dashed,thick},
   quivarrow/.style={black, -latex, thick},
  doublearrow/.style={black, latex-latex, very thick},
  anarrow/.style={black,-latex,thick}]
\newcommand{\strarrow}{\arrow{angle 60}}
\newcommand{\dotrad}{0.1cm} 
\newcommand{\bdrydotrad}{{0.8*\dotrad}} 
\path (0,0) node (bb) {}; 


\draw [black,dotted] (-4,0) -- (4,0);


\draw [black] (0:4) -- (51.429:4);
\draw [black] (51.429:4) -- (77.143:4);
\draw [black] (77.143:4) -- (102.857:4);
\draw [black] (102.857:4) -- (128.571:4);
\draw [black] (128.571:4) -- (180:4);


\draw [black] (0,0) -- (51.429:4);
\draw [black] (0,0) -- (77.143:4);
\draw [black] (0,0) -- (102.857:4);
\draw [black] (0,0) -- (128.571:4);


\draw [strand] plot[smooth]
coordinates{(4,0) (3.6,0.3) (51.429:3.5) (77.143:3.5) (102.857:3.5)(128.571:3.5) (-3.6,0.3) (-4,0)}
[ postaction=decorate, decoration={markings,
mark= at position 0.16 with \strarrow,
mark= at position 0.37 with \strarrow,
mark= at position 0.5 with \strarrow,
mark= at position 0.63 with \strarrow,
mark= at position 0.84 with \strarrow}];


\draw (0,0) circle(\bdrydotrad) [fill=gray];
\draw (4,0) circle(\bdrydotrad) [fill=gray];
\draw (-4,0) circle(\bdrydotrad) [fill=gray];


\draw (4,-0.4) node {$p_{ij}$};
\draw (0,-0.4) node {$p_{i,j+1}$};
\draw (-4,-0.4) node {$p_{i,j+2}$};

\end{tikzpicture}
\]
\caption{Path of a strand in $D(\mathcal{T})$ for $\mathcal{T}$ a triangulation}
\label{f:nearboundary}
\end{figure}
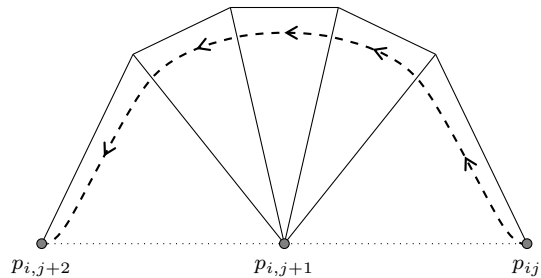

We may thus associate a dimer model $Q(D(\mathcal{T}))$ to a triangulation $\mathcal{T}$ as in
Lemma~\ref{l:triangulationdiagram}. As in Remark~\ref{r:potential}, we have an associated
potential. Deleting the boundary arrows of the quiver,
and all terms in the potential containing them, 
we obtain the quiver associated to $\mathcal{T}$ in~\cite{fst08} with potential as in~\cite{lf09}.

As an example, we consider the annulus.
Fix positive integers $n,m$.
Let $\Lambda_{n,m}=\F Q_{n,m}/\I$ be the algebra defined as follows.
The quiver $Q_{n,m}$ is embedded into an annulus.
The vertices are $1,2,\ldots ,n$ clockwise on the outer boundary and
$\overline{1},\overline{2},\ldots ,\overline{m}$ clockwise on the inner boundary.
There are arrows $x_i{\colon}i-1\to i$ and $y_i{\colon}i\to i-1$ on the outer boundary
(end points taken mod $n$) and arrows
$\overline{x}_i{\colon}\overline{i-1}\to \overline{i}$ and
$\overline{y}_i{\colon}\overline{i}\to \overline{i-1}$ on the inner boundary
(end points taken mod $m$), as well as arrows $r{\colon} 1\to \overline{1}$
from the outer to inner boundary and $s{\colon}\overline{m}\to n$ from the inner
to outer boundary. See Figure~\ref{f:annulus}.

The relations are given by the following, where we omit the subscripts for
$x$ and $y$ where they are determined by the starting vertex.
Firstly, we have the relations:
$$
xy=yx, \quad \quad \overline{x}\,\overline{y}=\overline{y}\,\overline{x},
$$
where the first relation (respectively, the second relation) starts at an
arbitrary vertex on the outer boundary (respectively, the inner boundary).
In addition, we have:
\begin{align}
y^2 &= x^{n-1-i} s\overline{x}^{m+1}rx^{i} \label{e:rel1} \\
\overline{y}^2 &= \overline{x}^{m-1-i} r x^{n+1} s \overline{x}^{i} \label{e:rel2} \\
r &= \overline{x}^m r x^n \\
s &= x^n s \overline{x}^m \\
y x s &= s \overline{x}\, \overline{y} \\
\overline{x}\, \overline{y} r &= r y x
\end{align}

There is an instance of relation~\eqref{e:rel1}
for each vertex on the outer boundary;
the exponent $i\geq 0$ is the minimum power of $x$
such that $x^i$, when starting at that vertex, ends at vertex $1$, 
the starting vertex of $r$. 
Similarly, there is an instance of relation~\eqref{e:rel2}) 
for each vertex on the inner boundary; 
the exponent $i\geq 0$ is the minimum power of $\overline{x}$ 
such that $\overline{x}^i$, when starting at that vertex, 
ends at vertex $\overline{m}$, the starting vertex of $s$.

\begin{remark}\label{r:Lambda-symmetry}
The algebra $\Lambda_{n,m}$ is 
more symmetric than its presentation suggests. 
If we define $\Lambda'_{n,m}$ analogously to $\Lambda_{n,m}$, but with
$r$ replaced by an arrow $R$ from $2$ to $\overline{1}$
and $s$ replaced by an arrow $S$ from $\overline{m}$ to
$1$, then the map $\Lambda_{n,m} \to\Lambda'_{n,m}$
taking $r$ to $\overline{x}^m R x$ and
$s$ to $x^{n-1}S$ is an isomorphism;
it has an inverse of a similar form. 
A corresponding isomorphism can be
constructed that moves the endpoints of
$r$ and $s$ on the inner boundary.
Thus the algebra $\Lambda_{n,m}$ may be presented by any quiver
`similar' to Figure~\ref{f:annulus}, 
i.e. with the same \emph{labelled} boundary,
but with the arrows $r,s$ joining any pair of adjacent vertices
on one boundary to any pair of adjacent vertices
on the other boundary, 
and subject to essentially the same relations.

The existence of such similar presentations
implies that $\Lambda_{n,m}$ has automorphisms that cycle the 
vertex idempotents on either boundary.
In particular, it has non-trivial automorphisms that fix all the
vertex idempotents, corresponding naturally to non-trivial elements of
the (boundary fixing) mapping class group of the annulus.

In fact, the algebra $\Lambda_{n,m}$ also has other presentations
in which the arrows $r$ and $s$ go between non-adjacent boundary vertices
and the relations are adjusted a little. 
However, these will not be so relevant here.
\end{remark}

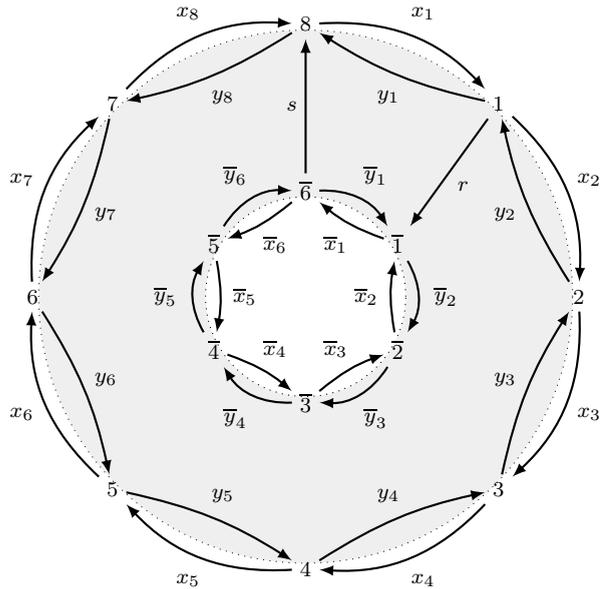
\begin{figure} 
\[
\begin{tikzpicture}[scale=1,
quivarrow/.style={black, -latex, thick}]

\newcommand{\outcr}{3.6}
\newcommand{\outxeps}{0.45}
\newcommand{\outyeps}{0.75}
\newcommand{\incr}{1.4}
\newcommand{\inxeps}{0.6}
\newcommand{\inyeps}{0.45}
\newcommand{\epscr}{0.08}


\draw[black,dotted] (0,0) circle(\outcr-\epscr);
\draw[fill=gray!40,draw=none,opacity=0.3] (0,0) circle (\outcr-\epscr);
\draw[fill=white,draw=black,dotted] (0,0) circle (\incr-\epscr);


\foreach \i in {1,...,8}
{\draw (90-45*\i:\outcr) node (outb\i) {$\i$};
 \coordinate (outxpos\i) at (112.5-45*\i:\outcr+\outxeps);
 \coordinate (outypos\i) at (112.5-45*\i:\outcr-\outyeps);}

\foreach \i in {1,...,6}
{\draw (90-60*\i:\incr) node (inb\i) {$\overline{\i}$};
 \coordinate (inxpos\i) at (120-60*\i:\incr-\inxeps);
 \coordinate (inypos\i) at (120-60*\i:\incr+\inyeps);}


\foreach \i/\j in {1/2,2/3,3/4,4/5,5/6,6/7,7/8,8/1}
{\draw [quivarrow]  (outb\i) edge [bend left=25] (outb\j);
 \draw [quivarrow]  (outb\j) edge [bend left=10] (outb\i);}


\foreach \i/\j in {1/2,2/3,3/4,4/5,5/6,6/1}
{\draw [quivarrow]  (inb\i) edge [bend left=30] (inb\j);
 \draw [quivarrow]  (inb\j) edge [bend left=10] (inb\i);}


\foreach \i in {1,...,8}
{\draw (outxpos\i) node {$x_{\i}$};
 \draw (outypos\i) node {$y_{\i}$};}

\foreach \i in {1,...,6}
{\draw (inxpos\i) node {$\overline{x}_{\i}$};
 \draw (inypos\i) node {$\overline{y}_{\i}$};}


\draw [quivarrow] (inb6) edge node [auto] {$s$} (outb8);
\draw [quivarrow] (outb1) edge node [auto] {$r$} (inb1);

\end{tikzpicture}
\]
\caption{The boundary algebra of an annulus}
\label{f:annulus}
\end{figure}

\begin{proposition}
\label{p:annulus}
Let $\mathcal{T}$ be any triangulation of the annulus,
with $n>0$ marked points on the outer boundary component 
and $m>0$ marked points on the inner boundary component,
as above, and let $D(\mathcal{T})$ be the corresponding weak Postnikov diagram
of degree $2$, as in Lemma~\ref{l:triangulationdiagram}.
Then the boundary algebra $B_{D(\mathcal{T})}$ is isomorphic to $\Lambda_{n,m}$.
\end{proposition}
\begin{proof}
Let $\mathcal{T}$ be a triangulation of the annulus
and set $A=A_{D(\mathcal{T})}$. Then $\mathcal{T}$
must include an edge linking the two boundary components;
we choose such an edge, and set $e_0$ to be the
corresponding idempotent in $A$. Since the complement of
this edge in the annulus is the interior of a disk, we
can apply Corollary~\ref{c:obtainJKS} (with some 
identifications) to get a description of $(e+e_0)A(e+e_0)$ as 
a quiver with relations, where $e$ is the sum of the boundary 
idempotents in $A$. An explicit description of the boundary 
algebra $eAe$, of the above form, can then be obtained
directly from this.
\end{proof}

By Remark~\ref{r:Lambda-symmetry}, the isomorphism in Proposition~\ref{p:annulus}
can (and indeed should ) be chosen 
to be compatible with an \emph{a priori}
identification of the vertex idempotents of~$\Lambda_{n,m}$ with
the boundary vertices of $Q(D(\mathcal{T}))$, i.e. with the boundary intervals of~$(X,M)$
(cf. Remark~\ref{r:bdryidemp}).
Thus, if we consider a second triangulation, then the two isomorphisms
will yield an isomorphism of boundary algebras that 
preserves their (labelled) idempotents.
In other words, the boundary algebra of a dimer model coming from a 
triangulated annulus does not depend on the triangulation
(up to such an isomorphism).

Recall that, by Lemma~\ref{l:twistinginvariance} 
and Proposition~\ref{p:geomexisomorphism},
or by Corollary~\ref{c:obtainJKS} and Remark~\ref{r:idempotents},
the boundary algebra of a Postnikov diagram of degree $k$ on a disk 
also does not depend on the choice of diagram. 
Furthermore, it might be expected that Lemma~\ref{l:twistinginvariance} 
and Proposition~\ref{p:geomexisomorphism} hold for the surface case. 
Thus it seems reasonable to make the following conjecture:

\begin{conjecture}
Let $(X,M)$ be a marked surface with nonempty boundary with marked points on
its boundary only which is not a monogon or digon. 
Suppose that there is at least one marked point on each boundary component. 
Then the boundary algebra of a weak Postnikov diagram on $(X,M)$ does 
not depend on the choice of diagram,
up to isomorphism preserving the idempotents.
\end{conjecture}

\affiliationone{K.~Baur\\
Institut f\"{u}r Mathematik und Wissenschaftliches Rechnen,\\ 
Universit\"{a}t Graz, NAWI Graz,\\
Heinrichstrasse 36,\\
A-8010 Graz, Austria
\email{baurk@uni-graz.at}}
\affiliationtwo{A.~King\\
Mathematical Sciences,\\
University of Bath,\\
Claverton Down,\\
Bath BA2 7AY, U.K.
\email{A.D.King@bath.ac.uk}}
\affiliationthree{R.J.~Marsh\\
School of Mathematics,\\
University of Leeds,\\
Leeds LS2 9JT, U.K.
\email{R.J.Marsh@leeds.ac.uk}}
\end{document}